\documentclass[a4paper]{amsart}

\usepackage{latexsym}
\usepackage{a4wide}
\usepackage{amscd}
\usepackage{graphics}
\usepackage{amsmath}
\usepackage{amssymb}
\usepackage{mathrsfs}
\usepackage{enumerate}
\usepackage{hyperref}
\usepackage{pgf,tikz}
\usepackage[utf8]{inputenc}

\numberwithin{equation}{section} 

\newcommand{\R}{\mathbb R}
\newcommand{\C}{\mathbb C}
\newcommand{\N}{\mathbb N}
\newcommand{\Z}{\mathbb Z}

\newcommand{\eps}{\varepsilon}
\newcommand{\abs}[1]{\left\vert #1 \right\vert}

\renewcommand{\Re}{\mathop{\mathfrak{Re}}}
\renewcommand{\Im}{\mathop{\mathfrak{Im}}}

\newtheorem{Theorem}{Theorem}[section]
\newtheorem{Corollary}[Theorem]{Corollary}
\newtheorem{Lemma}[Theorem]{Lemma}
\newtheorem{Proposition}[Theorem]{Proposition}
 \theoremstyle{definition}

 \newtheorem{remark}[Theorem]{Remark}

\definecolor{lilla}{rgb}{0.6, 0.4, 0.8}

\begin{document}

\title[Aharonov-Bohm operators with varying pole]{Estimates
  for eigenvalues of Aharonov-Bohm operators with varying poles and
  non-half-integer circulation}

\author{Laura Abatangelo}
\address{Laura Abatangelo 
\newline \indent Dipartimento di Matematica e Applicazioni, 
Università degli Studi di Milano-Bicocca,
\newline \indent  Via Cozzi 55, 20125 Milano, Italy.}
\email{laura.abatangelo@unimib.it}

\author{Veronica Felli}
\address{Veronica Felli 
\newline \indent Dipartimento di Scienza dei Materiali, Università
degli Studi di Milano-Bicocca,
\newline \indent Via Cozzi 55, 20125 Milano, Italy.}
\email{veronica.felli@unimib.it}

\author{Benedetta Noris}
\address{Benedetta Noris
\newline \indent LAMFA: Laboratoire Amiénois de Mathématique
Fondamentale et Appliquée,
\newline \indent
UPJV Université de Picardie Jules Verne, 33 rue Saint-Leu, 80039 Amiens, France.}
\email{benedetta.noris@u-picardie.fr}

\author{Manon Nys}
\address{Manon Nys 
\newline \indent Dipartimento di Matematica Giuseppe Peano, Università degli Studi di Torino, 
\newline \indent Via Carlo Alberto 10, 10123 Torino, Italy.}
\email{manonys@gmail.com}

\date{\today}

\begin{abstract}
  We study the behavior of eigenvalues of a magnetic Aharonov-Bohm
  operator with non-half-integer circulation and Dirichlet boundary
  conditions in a planar domain.  As the pole is moving in the
  interior of the domain, we estimate the rate of the eigenvalue
  variation in terms of the vanishing order of the limit eigenfunction
  at the limit pole. We also provide an accurate blow-up analysis for
  scaled eigenfunctions and prove a sharp estimate for their rate of
  convergence.
\end{abstract}

\maketitle

\section{Introduction and statement of the main results}\label{sec:introduction}

An infinitely long thin solenoid perpendicular to the plane $(x_1,x_2)$ at the point
$a=(a_1,a_2)\in\R^2$ produces a 
    point-like magnetic field as the radius of the solenoid goes to zero and
the magnetic flux remains constantly equal to
$\alpha\in \R\setminus\Z$. This  magnetic field is a
  $2\pi\alpha$-multiple of the Dirac delta at $a$  orthogonal to the
  plane $(x_1,x_2)$ and  is  generated by the Aharonov-Bohm vector
potential
\[
A_{a}(x) = \alpha \bigg( -\frac{x_2 - a_2}{(x_1 - a_1)^2 + (x_2 -
  a_2)^2} , \frac{x_1 - a_1}{(x_1 - a_1)^2 + (x_2 - a_2)^2} \bigg) ,
\quad x = (x_1,x_2) \in \R^2\setminus\{a\},
\]
see e.g. \cite{AdamiTeta1998,AB,MOR}.  We are interested in the
spectral properties of Schrödinger operators with Aharonov-Bohm vector
potentials, i.e.  of operators
\[
(i\nabla +A_{a})^2 := -\Delta + 2 i A_{a} \cdot \nabla + |A_{a}|^2.
\]
Since $\mathop{\rm curl} A_a\equiv 0$ in
$\R^2\setminus\{a\}$, the magnetic field is concentrated
at the pole $a$. If the circulation $\alpha$ is an integer number,
then the potential $A_a$ can be gauged away by a phase transformation
so that the operator $(i\nabla +A_{a})^2$ becomes spectrally
equivalent to the standard Laplacian. On the other hand, if
$\alpha\not\in\Z$, the vector potential $A_{a}$ cannot be eliminated
by gauge transformations and the spectrum of the operator is modified 
by the presence of the magnetic field: this produces the so-called 
Aharonov-Bohm effect, i.e. the magnetic potential affects charged quantum particles moving in the 
region $\Omega\setminus\{a\}$, even if the magnetic field
$B_a=\mathop{\rm curl} A_a$ is zero there.

The dependence on the pole $a$ of the spectrum of the Schr\"odinger 
operator $(i\nabla +A_{a})^2$ in a bounded domain $\Omega$ was investigated 
in \cite{AbatangeloFelli2015-1,AbatangeloFelli2015-2,AbatangeloFelliNorisNys2016,
BonnaillieNorisNysTerracini2014,NorisNysTerracini2015,NorisTerracini2010} 
under homogeneous Dirichlet boundary conditions. In particular, in 
\cite{AbatangeloFelli2015-1,AbatangeloFelli2015-2} sharp asymptotic 
estimates for eigenvalues were given in the case of half-integer circulation 
$\alpha \in \Z+\frac{1}{2}$ as the pole $a$ moves towards a fixed point 
$\bar a\in\Omega$; analogous sharp estimates were derived in 
\cite{AbatangeloFelliNorisNys2016} in the case $\bar a\in\partial\Omega$.

The case $\alpha\in\Z+\frac{1}{2}$ studied in the aforementioned
papers presents several peculiarities which allow approaching the
problem with a perspective and a technique which are not completely
adaptable to a general circulation $\alpha\in\R\setminus\Z$. Indeed,
if $\alpha\in\Z+\frac{1}{2}$ the problem can be reduced by gauge
transformation to the case $\alpha=\frac12$ and, in this case, the
eigenfunctions of $(i\nabla +A_{a})^2$ can be identified, up a complex
phase, with the antisymmetric eigenfunctions of the Laplace Beltrami
operator on the twofold covering manifold of $\Omega$, see
\cite{HHHO1999,NorisTerracini2010}. As a consequence, if
$\alpha=\frac12$, the magnetic eigenfunctions have an odd number of
nodal lines ending at the pole $a$. It has been proved in
\cite{HelfferHO2013} that the corresponding nodal domains are related
to optimal partition problems. We refer to \cite{BHH} and references
therein for related numerical simulations.

The special features characterizing Aharonov-Bohm operators with circulation 
$\frac12$ played a crucial role in \cite{AbatangeloFelli2015-1,AbatangeloFelli2015-2,
AbatangeloFelliNorisNys2016,BonnaillieNorisNysTerracini2014,NorisNysTerracini2015,
NorisTerracini2010}. In particular, in \cite{NorisNysTerracini2015} local energy 
estimates for eigenfunctions near the limit pole are performed by studying an 
Almgren type quotient (see \cite{Almgren1983}), which is estimated using a 
representation formula by Green's functions for solutions to the corresponding 
Laplace problem on the twofold covering. Moreover, in \cite{AbatangeloFelli2015-1,
AbatangeloFelli2015-2,AbatangeloFelliNorisNys2016} a limit profile vanishing on 
the special directions determined by the nodal lines of limit eigenfunctions is 
constructed: this allows establishing a sharp relation between the asymptotics 
of the eigenvalue function and the number of nodal lines, which is strongly related 
to the order of vanishing of the limit eigenfunction. 

In this paper we will focus on the case of non-integer and non-half-integer 
circulation, i.e. we will assume $\alpha \in \R\setminus \frac\Z2$. 
A reduction to the Laplacian on the twofold 
covering manifold is no more available in this case; moreover, magnetic eigenfunctions vanish 
at the pole $a$ but they do not have nodal lines ending at
$a$ (see Proposition \ref{prop:fft}).  
The lack of the special features of Aharonov-Bohm operators with half-integer 
circulation described above requires alternative methods and produces a less 
precise estimate. In particular,  in order to estimate the Almgren frequency 
function,  
we will give a detailed description of the behaviour of eigenfunctions 
at the pole  and we will study  the dependence of the coefficients 
of their asymptotic expansion 
with respect to the moving pole $a$, see Lemma \ref{lemma:stimbe}.

By gauge invariance, if $\alpha\in \R\setminus \frac\Z2$ it is not restrictive to 
assume that 
\begin{equation}\label{eq:alpha}
\alpha \in(0,1)\setminus\Big\{\frac12\Big\}.
\end{equation}
Let $\Omega\subset\R^2$ be a bounded, open and simply connected domain. For 
every $a \in \Omega$, we introduce the functional space $H^{1 ,a}(\Omega,\C)$ 
as the completion of
\[
\{ u \in H^1(\Omega,\C) \cap C^\infty(\Omega,\C)  : \, u \text{ vanishes in 
a neighborhood of } a \}
\] 
with respect to the norm 
\begin{equation}\label{eq:norm}
\| u \|_{ H^{1,a}(\Omega,\C) } = 
\left( \left\| \nabla u \right\|^2_{L^2(\Omega,\C^2)} 
+ \| u \|^2_{ L^2(\Omega,\C) } 
+ \Big\| \frac{u}{|x-a|} \Big\|^2_{ L^2(\Omega,\C) } \right)^{\!\!1/2}.
\end{equation}
The norm \eqref{eq:norm} is equivalent, under condition \eqref{eq:alpha}, to the norm
\begin{equation*}
\left( \left\| ( i \nabla + A_a ) u \right\|^2_{ L^2(\Omega,\C^2) } 
+ \| u \|^2_{ L^2(\Omega,\C) } \right)^{\!\!1/2},
\end{equation*}
in view of the Hardy type inequality proved in \cite{LaptevWeidl1999} 
(see also \cite{Alziary2003} and \cite[Lemma 3.1 and Remark 3.2]{FelliFerreroTerracini2011})
\begin{equation}\label{eq:hardy}
\int_{ D_r(a) } | (i\nabla+A_a) u |^2 \,dx \geq \Big( \min_{j \in \Z} |j-\alpha| \Big)^2 
\int_{ D_r(a) } \frac{ |u(x)|^2 }{|x-a|^2} \, dx,
\end{equation}
which holds for all $r > 0$, $a \in \R^2$  and $u \in H^{1,a}(D_r(a),\C)$. 
Here we denote as $D_r(a)$ the disk of center $a$ and radius $r$; we will 
denote as $D_r := D_r(0)$ the disk with radius $r$ centered at the origin.

It is also worth mentioning the following formulation
  of the magnetic Hardy inequality proved in \cite[Lemma
  4.1]{Alziary2003}:
 for all $r_1>r_2 > 0$, $a \in \R^2$,  and $u \in
 H^{1,a}(D_{r_1}(a)\setminus D_{r_2}(a),\C)$, 
\begin{equation}\label{eq:anello}
\int_{ D_{r_1}(a)\setminus D_{r_2}(a) } | (i\nabla+A_a) u |^2 \,dx \geq \Big( \min_{j \in \Z} |j-\alpha| \Big)^2 
\int_{ D_{r_1}(a)\setminus D_{r_2}(a)} \frac{ |u(x)|^2 }{|x-a|^2} \, dx.
\end{equation}

We also consider the space $H^{1 ,a}_{0}(\Omega,\C)$ as the completion of 
$C^\infty_{\rm c}(\Omega\setminus\{a\},\C)$ with respect to the norm 
$\| \cdot \|_{H^{1}_{a}(\Omega,\C)}$, so that 
\[
H^{1,a}_{0}(\Omega,\C) = 
\left\{ u \in H^1_0(\Omega,\C) \, : \, \frac{u}{|x-a|} \in L^2(\Omega,\C) \right\}.
\]
From classical spectral theory, for every $a\in\Omega$, the eigenvalue problem
\begin{equation} \label{eq:eige_equation_a} \tag{$E_a$}
\begin{cases}
(i\nabla + A_{a})^2 u = \lambda u,  & \text{in } \Omega,\\
u = 0, & \text{on } \partial \Omega
\end{cases}
\end{equation}
admits a diverging sequence of real eigenvalues $\{\lambda_k^a\}_{k\geq 1}$
with finite multiplicity; in the enumeration
\[
\lambda_1^a \leq \lambda_2^a \leq \dots \leq \lambda_j^a \leq \dots
\]
we repeat each eigenvalue as many times as its multiplicity.  We are interested 
in the behavior of the function $a\mapsto \lambda_j^a$ in a neighborhood of a 
fixed point $\bar a\in\Omega$. Up to a translation and a dilation, it is not restrictive to assume 
that $\bar a= 0 \in \Omega$ and $\overline{D_2} \subset \Omega$.

Let us assume that there exists $n_0 \geq 1$ such that 
\begin{equation}\label{eq:1}
\lambda_{n_0}^0 \quad \text{ is simple},
\end{equation}
and denote 
\begin{equation*}
\lambda_0 = \lambda_{n_0}^0 \quad \text{and} \quad \lambda_a = \lambda_{n_0}^a
\end{equation*}
for any $a\in\Omega$. In \cite[Theorem 1.3]{Lena2015} it is proved that
\begin{equation}\label{eq:lena}
 \text{if $\lambda_j^0$ is simple, the function $a \mapsto \lambda_j^a$ 
is \emph{analytic} in a neighborhood of $0$}.
\end{equation}
In particular the function $a \mapsto \lambda_a$ is 
continuous and, if $a \to 0$, then $\lambda_a \to \lambda_0$ (see also 
\cite{BonnaillieNorisNysTerracini2014}). Let $\varphi_0 \in H^{1,0}_{0}(\Omega,\C) \setminus \{ 0 \}$ 
be a $L^2(\Omega,\C)$-normalized eigenfunction of problem $(E_0)$ associated to 
the eigenvalue $\lambda_0 = \lambda_{n_0}^0$, i.e. satisfying
\begin{equation} \label{eq:equation_lambda0}
\begin{cases}
(i\nabla + A_0)^2 \varphi_0 = \lambda_0 \varphi_0,  &\text{in }\Omega,\\
\varphi_0 = 0, &\text{on }\partial \Omega,\\
\int_\Omega |\varphi_0(x)|^2\,dx=1.
\end{cases}
\end{equation}
From \cite[Theorem 1.3]{FelliFerreroTerracini2011} (see also  Proposition~\ref{prop:fft}) 
it is known that
\begin{equation}\label{eq:37}
\varphi_0 \text{ vanishes at } 0 \text{ with a vanishing order equal to } |\alpha-k| 
\text{ for some } k\in \Z,
\end{equation}
in the sense that there exist $k \in \Z$ and $\beta \in \C \setminus\{ 0 \}$ such that
\begin{equation}\label{eq:131}
r^{ -|\alpha-k| } \varphi_0( r (\cos t,\sin t) ) \to \beta \frac{e^{i k t}}{\sqrt{2\pi}} 
\quad \text{in } C^{1,\tau}([0,2\pi],\C)
\end{equation}
as $r \to 0^+$ for any $\tau\in (0,1)$. 

Our first result provides an estimate of the rate of convergence of $\lambda_0-\lambda_a$ 
in terms of  the order of vanishing of $\varphi_0$ at $0$; in particular we have that 
higher vanishing orders imply faster convergence of eigenvalues.

\begin{Theorem}\label{thm:eigenvalues}
  Let $\alpha \in(0,1)\setminus\big\{\frac12\big\}$ and
  $\Omega\subset\R^2$ be a bounded, open and simply connected 
  domain such that $0 \in \Omega$.  Let $n_0 \in \mathbb{N}$ be such
  that the $n_0$-th eigenvalue $\lambda_{n_0}^0=\lambda_0$ of problem
  $(E_0)$ is simple and let
  $\varphi_0 \in H^{1,0}_0(\Omega,\C)$ be an
  associated eigenfunction satisfying \eqref{eq:equation_lambda0}. Let
  $k \in \mathbb{Z}$ be such that $|\alpha - k|$ is the order of
  vanishing of $\varphi_0$ at $0$ as in \eqref{eq:131}. For
  $a \in \Omega$, let $\lambda_{n_0}^a=\lambda_a$ be the $n_0$-th
  eigenvalue of problem \eqref{eq:eige_equation_a}.  Then
\[
  | \lambda_a - \lambda_0 | 
  = O \left( |a|^{ 1 + \big\lfloor 2|\alpha-k|\big\rfloor} \right) 
  \quad \text{ as } |a| \to 0,
\]
where $\lfloor\cdot\rfloor$ denotes the floor function $\lfloor
t\rfloor:=\max\{k\in\Z:k\leq t\}$.
\end{Theorem}

To prove Theorem \ref{thm:eigenvalues}, we will study the
quotient 
\begin{equation}\label{eq:4}
\frac{ \lambda_0 - \lambda_a }{ |a|^{2 |\alpha - k|} }
\end{equation}
as $a$ approaches the origin along a straight line $\{t p \, : \, t>0\}$ 
for any direction
\[
p \in \mathbb{S}^1:=\{x\in\R^2: \ |x|=1\}.
\]
We will prove that, for every $p \in \mathbb{S}^1$, the quotient 
\eqref{eq:4} is bounded as $a=|a| p\to 0$. Then 
  \eqref{eq:lena} and the fact that $2 |\alpha - k|$ 
is non-integer imply that  the Taylor polynomials of the function $\lambda_0 - \lambda_a$ 
with center $0$ and degree less or equal than $\lfloor 2|\alpha-k|\rfloor$ vanish, 
thus yielding the conclusion of Theorem \ref{thm:eigenvalues}.

In the case of half-integer
circulation $\alpha=\frac12$ the special nodal structure of the limit
problem allows proving instead that 
the limit 
\[
\lim_{a=|a|p\to 0}\frac{ \lambda_0 - \lambda_a }{ |a|^{2 |\alpha - k|} }
=\lim_{a=|a|p\to 0}\frac{ \lambda_0 - \lambda_a }{ |a|^{|1 -2 k|} }
\]
is different from $0$ along some special directions 
$p$ corresponding to tangents to the nodal lines of the
limit eigenfunction.
As a consequence, the leading term of the Taylor expansion of the eigenvalue 
variation $\lambda_0 - \lambda_a$ has order exactly $|1-2k|$, i.e. 
\begin{equation*}
\lambda_0 - \lambda_a = P(a) + o(|a|^{|1 -2 k|}), \quad \text{ as } |a| \to 0^+,
\end{equation*}
for some homogeneous polynomial $P\not\equiv 0$ of degree $|1 -2 k|$, see 
\cite[Theorem 1.2]{AbatangeloFelli2015-1}. In \cite[Theorem 2]{AbatangeloFelli2015-2} 
the  exact value of all coefficients of the polynomial $P$ is determined proving 
that $P(|a|(\cos t,\sin t))=C_0|a|^{|1 -2 k|}\cos(|1 -2 k| (t-t_0))$ for some $t_0$ and $C_0>0$. 
In particular the leading polynomial $P$ is  harmonic.

In this paper we will also describe the behaviour of the eigenfunctions as $a\to0$, proving a blow-up 
result for scaled eigenfunctions and giving a sharp rate of the convergence to the 
limit eigenfunction $\varphi_0$. In order to state these results more precisely, 
we need to introduce some notations.

For every $b=(b_1,b_2)=|b|(\cos\vartheta,\sin\vartheta)\in \R^2\setminus\{0\}$ 
with $\vartheta\in[0,2\pi)$, we define the polar angle centered at $b$, 
$\theta_b:\R^2\setminus\{b\}\to [\vartheta,\vartheta+2\pi)$ as
\begin{equation} \label{eq:theta-b}
\theta_b(b+r(\cos t,\sin t))=t
\quad\text{for all }r>0\text{ and }t\in [\vartheta,\vartheta+2\pi),
\end{equation}
and the function $\theta_0^b:\R^2\setminus\{0\}\to [\vartheta,\vartheta+2\pi)$ as
\begin{equation}\label{eq:theta_0^b}
\theta_0^b (r (\cos t,\sin t)) = t \quad \text{ for all } r > 0 
\text{ and } t \in  [\vartheta, \vartheta + 2\pi),
\end{equation}
in such a way that the difference function $\theta_0^b - \theta_b$ is regular 
except for the segment 
\begin{equation} \label{eq:gamma-b}
\Gamma_b := \{ tb \, : \, t \in [0,1] \}.
\end{equation}
For all $a \in \Omega$, let $\varphi_a \in H^{1,a}_{0}(\Omega,\C)\setminus\{0\}$ 
be an eigenfunction of problem \eqref{eq:eige_equation_a} associated to the eigenvalue 
$\lambda_a$, i.e. solving
\begin{equation}\label{eq:equation_a}
\begin{cases}
(i\nabla + A_a)^2 \varphi_a = \lambda_a \varphi_a,  &\text{in }\Omega,\\
\varphi_a = 0, &\text{on }\partial \Omega,
\end{cases}
\end{equation}
such that the following normalization conditions hold
\begin{equation}\label{eq:normalization}
\int_\Omega |\varphi_a(x)|^2 \,dx = 1 \quad \text{ and } \quad 
\int_\Omega e^{i \alpha (\theta_0^a-\theta_a)(x)} \varphi_a(x) \overline{\varphi_0(x)} \, dx 
\text{ is a positive real number}.
\end{equation}
Using \eqref{eq:1}, 
\eqref{eq:equation_lambda0}, \eqref{eq:equation_a}, \eqref{eq:normalization}, and 
standard elliptic estimates (see e.g. \cite[Theorem 8.10]{gilbarg-trudinger}), it is easy to prove that
\begin{equation}\label{eq:convergence-varphi-a-1}
\varphi_a \to \varphi_0 \quad \text {in } H^1(\Omega,\C) \text{ and in
} C^2_{\rm loc}(\Omega\setminus\{0\},\C), 
\end{equation}
and 
\begin{equation}\label{eq:convergence-varphi-a-2}
(i\nabla+A_a) \varphi_a \to (i\nabla+A_0) \varphi_0 
\quad \text{ in } L^2(\Omega,\C).
\end{equation}
To give a precise description of the behavior of the eigenfunction
$\varphi_a$ for $a$ close to $0$, we consider a homogeneous scaling of
order $|a|^{|\alpha-k|}$ of $\varphi_a$ along a fixed direction
 $p\in {\mathbb S}^1$. Theorem \ref{t:bu}
below gives the convergence of scaled eigenfunctions to a nontrivial
limit profile $\Psi_{p} \in H^{1 ,{p}}_{\rm loc}(\R^2,\C)$, which can
be characterized as the unique solution to the problem
\begin{equation}\label{eq:Psip1}
(i\nabla + A_{p})^2 \Psi_{p} = 0 
\quad \text{ in } \R^2 \text{ in a weak } H^{1,p} \text{-sense},
\end{equation}
satisfying \begin{equation}\label{eq:Psip2}
\int_{\R^2\setminus D_1} |(i\nabla + A_{p}) 
( \Psi_{p} - e^{ i \alpha ( \theta_{p} 
- \theta_0^{p}) } \psi_k)|^2 \, dx 
< + \infty ,
\end{equation}
where $\psi_k:\R^2\to\C$ is defined as 
\begin{equation}\label{eq:psi_j_definition}
\psi_k(r(\cos t,\sin t))=  r^{|\alpha-k|}  \frac{e^{i k t}}{\sqrt{2\pi}}.
\end{equation}
The existence and uniqueness of a limit profile satisfying \eqref{eq:Psip1}
and \eqref{eq:Psip2} will be proved in Lemma~\ref{lem:uniqueness}. 
We notice that the function $\psi_k$ in \eqref{eq:psi_j_definition} 
is the unique (up to a multiplicative constant) $H^{1 ,{0}}_{\rm loc}(\R^2,\C)$-solution to $(i\nabla+A_0)^2\psi_k=0$ 
in $\R^2$ which is homogeneous of degree $|\alpha-k|$.

\begin{Theorem}\label{t:bu}
  Under the same assumptions as in Theorem \ref{thm:eigenvalues}, for
  all $a \in \Omega$ let $\varphi_a \in H^{1,a}_{0}(\Omega,\C)$ be an
  eigenfunction of problem \eqref{eq:eige_equation_a} associated to
  the eigenvalue $\lambda_a$ and satisfying \eqref{eq:normalization}. Let
  $p\in {\mathbb S}^1$. Then
\[
  \frac{\varphi_a(|a|x)}{|a|^{|\alpha-k|}} \to \beta \Psi_{p}
  \quad \text{ as } a=|a|p\to 0,
\]
  in $H^{1,p}(D_R, \C)$ for every $R > 1$, almost everywhere in
  $\R^2$ and in $C^2_{\rm loc}(\R^2\setminus\{p \},\C)$, 
  with $\beta \neq 0$ and $k \in \Z$ being as in \eqref{eq:131} and $\Psi_{p}$ 
  being as in \eqref{eq:Psip1}--\eqref{eq:Psip2}.
\end{Theorem}

Finally, we describe the sharp rate of convergence  \eqref{eq:convergence-varphi-a-2}, 
which also turns out to depend strongly on the order of vanishing of $\varphi_0$ 
at $0$, as stated in the following theorem. 

\begin{Theorem}\label{thm:eigenfunctions}
Under the same assumptions as in Theorems \ref{thm:eigenvalues} and \ref{t:bu}, for every 
$p\in{\mathbb S}^1$ there exists ${\mathfrak L}_{p}>0$ 
such that 
\begin{equation*}
|a|^{-2|\alpha-k|} \left\| (i\nabla+A_a) \varphi_a 
- e^{i\alpha(\theta_a-\theta_0^a)} (i\nabla+A_0) \varphi_0 \right\|^2_{L^2(\Omega,\C)}
\to |\beta|^2\,{\mathfrak L}_{p},
\quad \text{ as } a = |a| p \to0.
\end{equation*}

\end{Theorem}
We observe that Theorem~\ref{thm:eigenfunctions} extends to the case
of non-half-integer circulation an analogous result obtained in
\cite{AbatangeloFelli2017} for half-integer circulation.

The paper is organized as follows. In Section~\ref{sec:local_asymptotics} 
we perform a detailed description of the behavior of the eigenfunction 
$\varphi_a$ near the pole $a$, which is crucial in Section~\ref{sec:monotonicity} 
to prove an Almgren type monotonicity formula and to derive local energy estimates 
for eigenfunctions uniformly with respect to the moving pole. In Section~
\ref{sec:estim-texorpdfstr-} we obtain some upper and lower bounds for the difference 
$\lambda_0-\lambda_a$ by exploiting the Courant-Fisher minimax characterization of 
eigenvalues and testing the Rayleigh quotient with suitable competitor functions. 
Section~\ref{sec:blow-up-analysis} contains a blow-up analysis for scaled eigenfunctions 
which allows proving Theorems~\ref{thm:eigenvalues} and \ref{t:bu}. Finally, in Section~
\ref{sec:rate-conv-eigenf} we prove Theorem~\ref{thm:eigenfunctions}.

\medskip

\noindent {\bf Notation. }
We list below some notation used throughout the paper.\par
\begin{itemize}
\item[-] For all $r > 0$ and $a \in \R^2$, $D_r(a) = \{ x \in \R^2 \, : \, |x-a|<r \}$ 
denotes the disk of center $a$ and radius $r$.
\item[-] For all $r > 0$, $D_r = D_r(0)$ and $\mathbb{S}^1=\partial D_1$.
\item[-] $ds$ denotes the arc length on $\partial D_r(a)$.
\item[-] For every complex number $z \in \C$, $\overline{z}$ denotes its complex conjugate.
\item[-] For $z \in \C$, $\Re z$ denotes its real part and $\Im z$ its imaginary part.
\end{itemize}

\section{Local asymptotics of eigenfunctions}\label{sec:local_asymptotics}

We recall from \cite{FelliFerreroTerracini2011} the description of the asymptotics at 
the singularity of solutions to elliptic equations with Aharonov-Bohm potentials. In 
the case of Aharonov-Bohm potentials with circulation $\alpha\in(0,1) \setminus 
\{ \frac12 \}$, such asymptotics is described in terms of eigenvalues and eigenfunctions 
of the following operator $\mathcal H$ acting on $2\pi$-periodic functions
\[
\mathcal H \psi = -\psi'' + 2 i \alpha \psi' + \alpha^2 \psi.
\]
It is easy to verify that the eigenvalues of $\mathcal H$ are $\big\{ (\alpha-j)^2  : \, 
j \in \Z \big\}$; each eigenvalue $(\alpha-j)^2$ has multiplicity $1$ and the eigenspace 
associated is generated by the function $\frac{e^{ijt}}{\sqrt{2\pi}}$. Let us enumerate 
the eigenvalues $(\alpha-j)^2$ as $\big\{ (\alpha-j)^2 \, : \, j \in \Z \big\} = 
\{\mu_j \, : \, j=1,2,\dots \}$ with $\mu_1 < \mu_2 < \mu_3 < \ldots$, so that
\begin{equation}\label{eq:mu1_def}
\mu_1 = \min\{ \alpha^2,(1-\alpha)^2 \}
\end{equation}
and $\mu_2 = \max\{ \alpha^2 , (1-\alpha)^2\}$.

\begin{Proposition}[\cite{FelliFerreroTerracini2011}, Theorem 1.3] \label{prop:fft}
Let $\Omega \subset \R^2$ be a bounded open set containing $b$, $\lambda \in \R$, 
and $u \in H^{1,b}_0(\Omega,\C)$ be a nontrivial weak solution to the problem
\begin{equation*}
(i \nabla + A_b)^2 u = \lambda u,  \quad \text{ in } \Omega,
\end{equation*}
i.e. 
\[
\int_\Omega (i \nabla + A_b) u \cdot \overline{ ( i \nabla +A_b) v} \, dx 
= \lambda \int_\Omega u \overline{v} \, dx
\quad \text{ for all } v \in H^{1,b}_{0}(\Omega,\C).
\]
Then there exists $j \in \Z$ such that 
\[
\lim_{r \to 0^+} \dfrac{r \int_{D_r(b)} \big( \abs{(i\nabla+A_b)u(x)}^2 
- \lambda \abs{u(x)}^2 \big) \, dx }{\int_{\partial D_r(b)} \abs{u}^2\, ds} 
= |\alpha-j|.
\]
Furthermore, there exists $\beta(b,u,\lambda)\neq 0$ such that
\begin{equation}\label{eq:5}
r^{-|\alpha-j|} u(b + r(\cos t,\sin t)) \to \beta(b,u,\lambda) \frac{e^{ijt}}{\sqrt{2\pi}}
\quad \text{ in } C^{1,\tau}([0,2\pi],\C)
\end{equation}
as $r\to0^+$ for any $\tau\in (0,1)$.
\end{Proposition}

Let us fix $n\in \N$, $n\geq1$. For all $a\in\Omega$, let $\varphi_n^a\in
H^{1,a}_{0}(\Omega,\C)\setminus\{0\}$ be an eigenfunction of problem
\eqref{eq:eige_equation_a} associated to the eigenvalue $\lambda_n^a$,
i.e. solving
\begin{equation}\label{eq_eigenfunction_n}
 \begin{cases}
   (i\nabla + A_a)^2 \varphi_n^a = \lambda_n^a \varphi_n^a,  &\text{in }\Omega,\\
   \varphi_n^a = 0, &\text{on }\partial \Omega,
 \end{cases}
\end{equation}
such that 
\begin{equation}\label{eq:23_n}
  \int_\Omega |\varphi_n^a(x)|^2\,dx=1.
\end{equation}
Since $a\in\Omega\mapsto \lambda_n^a$ admits a
continuous extension on $\overline{\Omega}$ as proved in \cite[Theorem
1.1]{BonnaillieNorisNysTerracini2014}, we have that
\begin{equation}\label{eq:Lambda_n}
  \Lambda_n=\sup_{a\in\Omega}\lambda_n^a\in(0,+\infty).
\end{equation}
Moreover, from \eqref{eq_eigenfunction_n}, \eqref{eq:23_n},  and \eqref{eq:hardy} it follows that 
\begin{equation*}
\{\varphi_n^a\}_{a\in \Omega} \text{ is bounded in } H^1(\Omega,\C),
\end{equation*}
which, by \eqref{eq_eigenfunction_n} and classical
elliptic regularity theory, implies that, for each  $\omega
  \subset\subset\Omega\setminus\{0\}$, there exists $\rho_\omega>0$
  such that 
\begin{equation}\label{eq:13}    
\{\varphi_n^a\}_{|a|\leq\rho_\omega} \text{ is bounded in }
C^{2,\sigma}(\omega,\C)\text{ for every $\sigma\in(0,1)$}.
\end{equation}
The following lemma provides a detailed description of the behaviour
of the Fourier coefficients of the function $t\mapsto
\varphi_n^a(a+r(\cos t,\sin t))$ as $a$ is close to $0$.
\begin{Lemma}\label{lemma:stimbe}
For $n\geq1$ fixed and $a$ varying in $\Omega$, let $\varphi_n^a\in
H^{1,a}_{0}(\Omega,\C)\setminus\{0\}$ satisfy
\eqref{eq_eigenfunction_n} and \eqref{eq:23_n}.
For all $j \in \Z$ and $a \in \Omega$, let 
\begin{equation}\label{eq:def_vja}
v_j^a(r) = \frac1{\sqrt{2\pi}}\int_0^{2\pi}\varphi_n^a(a+r(\cos t,\sin t)) e^{-ijt}\,dt.
\end{equation}
Then there exists $\rho_0>0$ such that, for all $a$ with
$|a|\leq\rho_0$,  the following properties hold.
\begin{enumerate}[\rm (i)]
\item For all $j \in \Z$, $v_j^a(r) = O(r^{|\alpha-j|})$ as $r \to 0^+$. 
In particular, for  all $j \in \Z$ and for all $R > 0$ such that 
$\{ x \in \R^2 \, : \, |x-a| \leq R \} \subset \Omega$, the value
\begin{equation}\label{eq:beta-2}
\beta_j^a  = \frac{ v_j^a(R) }{ R^{|\alpha-j|} } + 
\frac{\lambda_n^a}{2 |\alpha-j|} \int_0^R \bigg( s^{1-|\alpha-j|} 
-\frac{ s^{ 1 + |\alpha-j| } }{ R^{2 |\alpha-j|} } \bigg) v_j^a(s) \,ds
\end{equation}
is well defined and independent of $R$.
\item For all $j\in \Z$, 
$|\beta_j^a|\leq B$ for some $B>0$ 
independent of $j$ and $a$.
\item For all $j \in \Z$,
\[  
v_j^a(r) = r^{|\alpha-j|} \beta_j^a ( 1 + R_{j,a}(r)) 
\quad \text{and} \quad 
(v_j^a)'(r) = |\alpha-j| \beta_j^a r^{|\alpha-j|-1} ( 1 + \widetilde R_{j,a}(r))
\] 
where $|R_{j,a}(r)| + |\widetilde R_{j,a}(r)| \leq {\rm const\,} r^2$ for some
${\rm const} > 0$ 
independent of $j$ and $a$. 
\item $\varphi_n^a$ can be expanded as
\[
\varphi_n^a (a + r (\cos t,\sin t)) = \frac1{\sqrt{2\pi}}\sum_{j \in \Z} 
r^{|\alpha-j|}\beta_j^a(1+R_{j,a}(r))
e^{ijt}
\]
being $R_{j,a}(r)$ as in {\rm (iii)}, where the convergence of the above
series is uniform on disks $D_R(a)$ for each $R\in(0,1)$.
\item  Letting ${\boldsymbol{\nu}}(t)=(\cos t,\sin t)$ and
  ${\boldsymbol{\tau}}(t)=(-\sin t,\cos t)$, $(i\nabla+A_a)\varphi_n^a$ can be expanded as
\begin{align*}
(i\nabla+A_a) & \varphi_n^a(a + r (\cos t,\sin t)) \\
& = \frac1{\sqrt{2\pi}} \sum_{j \in \Z} \beta_j^a r^{|\alpha-j|-1} 
\Big(i|\alpha-j|(1+\widetilde R_{j,a}(r)){\boldsymbol{\nu}}(t) 
+ (\alpha-j) ( 1 + R_{j,a}(r)) {\boldsymbol{\tau}}(t)\Big) e^{ijt}
\end{align*}
being $R_{j,a}(r),\widetilde R_{j,a}(r)$ as in {\rm (iii)}, where the above
series converges absolutely in $L^2(D_R(a),\C)$ and point-wise in $D_R(a)$ 
for each $R\in(0,1)$.
\end{enumerate}
\end{Lemma}

\begin{proof}
The functions $\big\{\frac{e^{ijt}}{\sqrt{2\pi}}\big\}_{j\in\Z}$ form an orthonormal basis of 
$L^2((0,2\pi),\C)$. Hence, recalling that we are assuming that
$\overline{D_2} \subset \Omega$,  if $|a|$ sufficiently small
$\varphi_n^a$ can be expanded as
\begin{equation}\label{eq:Fourier}
\varphi_n^a (a + r (\cos t,\sin t)) = \sum_{j \in \Z} v_j^a(r) \frac{e^{ijt}}{\sqrt{2\pi}} 
\quad \text{ in } L^2((0,2\pi),\C) \text{ for all } r\in(0,1],
\end{equation}
where $v_j^a$ is defined in \eqref{eq:def_vja}. Equation 
\eqref{eq:equation_a} implies that, for every $j \in \mathbb{Z}$,
\begin{equation}\label{eq:45}
-(v_j^a)''(r) - \frac{1}{r}(v_j^a)'(r) + \frac{(\alpha-j)^2}{r^2}v_j^a(r) 
= \lambda_n^a v_j^a(r), \quad \text{ for all } r \in (0,1],
\end{equation}
or equivalently
\[
- r^{|\alpha - j| - 1} \left( r^{1 - 2 |\alpha - j|} \left( r^{|\alpha - j|} v_j^a \right)' \right)' 
= \lambda_n^a v_j^a(r), \quad \text{ for all } r \in (0,1].
\]
Integrating twice between $r$ and $1$, we obtain, for some $c_{1,j}^a,c_{2,j}^a \in \C$,
\begin{equation}\label{eq:vja}
v_j^a(r) = r^{|\alpha-j|} \bigg( c_{1,j}^a 
+ \lambda_n^a \int_r^{1} \frac{s^{-|\alpha-j|+1}}{2|\alpha-j|} v_j^a(s) \, ds \bigg) 
+ r^{-|\alpha-j|} \bigg( c_{2,j}^a 
- \lambda_n^a \int_r^{1} \frac{s^{|\alpha-j|+1}}{2|\alpha-j|} v_j^a(s) \,ds \bigg),
\end{equation}
for all $r\in(0,1]$.

The convergence \eqref{eq:5} in Proposition~\ref{prop:fft} implies that, for all $a$, 
$|\varphi_n^a(a+r(\cos t,\sin t))| = O(r^{\sqrt{\mu_1}})$ as $r\to0^+$, with $\mu_1$ as in 
\eqref{eq:mu1_def} (not necessarily uniformly with respect to $a$). 
Hence, for every $a$ 
in a sufficiently small neighborhood of $0$, there exists a constant
  $C(a)>0$ such that, for all 
 $j\in\Z$,
 \begin{equation}\label{eq:first_est_vja}
|v_j^a(r)|\leq C(a) r^{\sqrt{\mu_1}} \quad \text{ for all } r \in [0,1].
\end{equation}
We deduce that each function $v_j^a$ is bounded near $0$,
hence \eqref{eq:vja} necessarily yields
\begin{equation}\label{eq:c2}
c_{2,j}^a = \lambda_n^a \int_0^{1} \frac{s^{|\alpha-j|+1}}{2|\alpha-j|} v_j^a(s) \, ds.
\end{equation}
We can therefore rewrite
\begin{align} \label{eq:7-1}
v_j^a(r) &=r^{|\alpha-j|} \bigg( c_{1,j}^a 
+ \frac{\lambda_n^a}{2|\alpha -j|} \int_r^{1} s^{-|\alpha-j|+1} v_j^a(s) \, ds \bigg)\\
\notag&\quad+ \frac{\lambda_n^a}{2|\alpha-j|} r^{-|\alpha-j|} \int_0^r s^{|\alpha-j|+1} v_j^a(s) \, ds.
\end{align}
If
  $\sqrt{\mu_1}+2\geq |\alpha-j|$,  using
    \eqref{eq:first_est_vja} to estimate the right hand side of
    \eqref{eq:7-1} we obtain the improved estimate
  $|v_j^a(r)|\leq C(j,a) r^{|\alpha-j|}$. Otherwise, if
  $\sqrt{\mu_1}+2<|\alpha-j|$, we can use
    \eqref{eq:first_est_vja} to estimate the right hand side of
    \eqref{eq:7-1} to obtain the improved estimate
  $|v_j^a(r)|\leq C(j,a) r^{\sqrt{\mu_1}+2}$, for some
    constant $C(j,a)>0$ depending on $a$ and $j$. By iterating the process $m+1$ times, with $m$ the
  largest natural number such that $\sqrt{\mu_1}+2m<|\alpha-j|$, we
  obtain that $|v_j^a(r)| \leq C(j,a) r^{|\alpha - j|}$, possibly for
  a different constant $C(j,a)$.  We deduce that the quantity
$\beta_j^a$ introduced in \eqref{eq:beta-2} is well
defined.  The fact that $\beta_j^a$ is independent of $R$ is a direct 
consequence of \eqref{eq:45} and \eqref{eq:7-1}. This proves statement {\rm(i)}. 

Using the independence of $\beta_j^a$ with respect to $R$, we choose $R=1$ in \eqref{eq:beta-2} 
and $r = 1$ in \eqref{eq:7-1} and obtain that
\begin{equation}\label{eq:7}
\beta_j^a=c_{1,j}^a+\frac{\lambda_n^a}{2|\alpha-j|} \int_0^{1}
s^{-|\alpha-j|+1} v_j^a(s) \, ds,
\end{equation}
so that \eqref{eq:7-1} can be rewritten as 
\begin{equation}\label{eq:7-2}
v_j^a(r)=r^{|\alpha-j|}
\bigg(\beta_j^a-\lambda_n^a\int_0^r\frac{s^{-|\alpha-j|+1}}{2|\alpha-j|}
v_j^a(s)\,ds\bigg)+\lambda_n^a r^{-|\alpha-j|}
\int_0^r\frac{s^{|\alpha-j|+1}}{2|\alpha-j|}
v_j^a(s)\,ds.
\end{equation}
From
\eqref{eq:7-2} it follows that, for all $r\in(0,1]$, 
\begin{align*}
r^{-|\alpha-j|} |v_j^a(r)| &\leq |\beta_{j}^a| +
\frac{\lambda_n^a }{2|\alpha -j|} \int_0^r s^{-|\alpha-j|} |v_j^a(s)| \,
ds \\
&\qquad +\frac{\lambda_n^a }{2|\alpha-j|} r^{-2|\alpha-j|} \int_0^r
s^{2|\alpha-j|}s^{-|\alpha-j|} |v_j^a(s)| \, ds,\\
&\leq |\beta_{j}^a| +
\frac{\lambda_n^a }{|\alpha -j|} \int_0^r s^{-|\alpha-j|} |v_j^a(s)| \,
ds .
\end{align*}
Hence the Gronwall Lemma applied to the function
$r\mapsto r^{-|\alpha-j|} |v_j^a(r)|$ yields that
\begin{equation}\label{eq:3-1}
 r^{-|\alpha-j|} |v_j^a(r)| \leq C|\beta_{j}^a|
\quad \text{ for all } r \in (0,1] \text{ and } j \in \Z,
\end{equation}
for some constant $C>0$ independent of $j$, $a$, and $r$.

From  \eqref{eq:vja}, \eqref{eq:def_vja}, and \eqref{eq:13} it follows that 
\[
|c_{1,j}^a+c_{2,j}^a|=|v_j^a(1)|=
\frac1{\sqrt{2\pi}}\left|\int_0^{2\pi}\varphi_n^a(a+(\cos t,\sin t))
e^{-ijt}\,dt\right|\leq \textrm{const\,}
\]
for some $\textrm{const\,}>0$ independent of $j$ and $a$; moreover from 
\eqref{eq:c2} and \eqref{eq:23_n} we deduce that
\[
|c_{2,j}^a|\leq \frac{\lambda_n^a}{2|\alpha-j|}\int_0^{1}s|v_j^a(s)| \,ds 
\leq \frac{\lambda_n^a}{2|\alpha-j|\sqrt{2\pi}}\int_{D_{1}(a)} |\varphi_n^a|\,dx\leq \textrm{const},
\]
for some $\textrm{const\,}>0$ independent of $j$ and $a$. Hence 
\begin{equation}\label{eq:8}
|c_{1,j}^a|\leq \widetilde C
\end{equation}
for some $\widetilde C>0$ independent of $j$ and $a$.  

Let $K>0$ be such that
\[
\frac{\Lambda_n C}{2K}<\frac12
\]
with $C$ being as in \eqref{eq:3-1} and $\Lambda_n$ being as in \eqref{eq:Lambda_n}. Hence, from \eqref{eq:Lambda_n}, \eqref{eq:7}, \eqref{eq:3-1} and \eqref{eq:8} it follows that, if $|\alpha-j|>K$, then
\begin{equation}\label{eq:6}
\frac12 |\beta_j^a|\leq
\left(1-\frac{\Lambda_n C}{2K}\right)|\beta_j^a|\leq
|c_{1,j}^a|\leq \widetilde C.
\end{equation}
Let us choose $R_0\in(0,1)$ such that  
\begin{equation*}
\frac{\Lambda_n C R_0^2}{2\sqrt{\mu_1}}<\frac12. 
\end{equation*}
From \eqref{eq:beta-2} and \eqref{eq:3-1} it follows that, if
$|\alpha-j|\leq K$,
\begin{align*}
 R_0^{-K} |v_j^a(R_0)|\geq R_0^{-|\alpha-j|} |v_j^a(R_0)|&  =\left| \beta_j^a -
\frac{\lambda_n^a}{2 |\alpha-j|} \int_0^{R_0} \bigg( s^{1-|\alpha-j|} 
-\frac{ s^{ 1 + |\alpha-j| } }{ R_0^{2 |\alpha-j|} } \bigg) v_j^a(s)
                                 \,ds\right|\\
&\geq |\beta_j^a|-\frac{\Lambda_n C
  R_0^2}{2\sqrt{\mu_1}}|\beta_j^a|\geq \frac12 |\beta_j^a|
\end{align*}
Since, in view of \eqref{eq:13}, $v_j^a(R_0)$ is bounded uniformly with respect to $a$ and $j$, 
we conclude that, for all $j$ such that  $|\alpha-j|\leq K$, $|\beta_j^a|$ is bounded 
uniformly with respect to $a$ and $j$. This, together with \eqref{eq:6}, yields {\rm (ii)}.

From \eqref{eq:7-2} and \eqref{eq:3-1} it follows that 
\begin{equation}\label{eq:38}
v_j^a(r)=r^{|\alpha-j|}\beta_j^a(1+R_{j,a}(r))
\end{equation}
where $|R_{j,a}(r)|\leq  \textrm{const\,}r^2$ for 
some $\textrm{const\,}>0$ independent of $j$ and $a$, 
thus proving the first estimate in (iii). 
Differentiating \eqref{eq:7-2} and using the above estimate~\eqref{eq:38},
we easily obtain that 
\begin{equation*}
(v_j^a)'(r)=|\alpha-j|\beta_j^a r^{|\alpha-j|-1}(1+\widetilde R_{j,a}(r))
\end{equation*}
where $|\widetilde R_{j,a}(r)|\leq  \textrm{const\,}r^2$ for 
some $\textrm{const\,}>0$ independent of $j$ and $a$. Hence the proof of (iii) is complete.  

From \eqref{eq:Fourier} and (iii) we have that the series
$\frac1{\sqrt{2\pi}}\sum_{j \in \Z}
r^{|\alpha-j|}\beta_j^a(1+R_{j,a}(r)) e^{ijt}$ converges in
$L^2((0,2\pi),\C)$ to
$\varphi_n^a (a + r (\cos t,\sin t))$ for all $r\in(0,1]$. In view of the estimates
obtained in (ii)-(iii), Weierstrass M-Test ensures that the series  
is uniformly convergent in $D_R(a)$ for every
$R\in(0,1)$, thus proving (iv).

Let 
$f_j^a(a + r (\cos t, \sin t )) =
v_j^a(r)\frac{e^{ijt}}{\sqrt{2\pi}}$. Since 
\[
(i\nabla +A_a)f_j^a(a+r(\cos t,\sin t))=
\Big(i (v_j^a)'(r){\boldsymbol{\nu}}(t) 
+ (\alpha - j)\tfrac{v_j^a(r)}r{\boldsymbol{\tau}}(t)\Big)\frac{e^{ijt}}{\sqrt{2\pi}},
\]
the above estimates also imply that, for every
$R\in(0,1)$, the series of functions $\sum_j (i\nabla +A_a)f_j^a$
is convergent absolutely in $L^2(D_R(a),\C)$ and point-wise in $D_R(a)$ to
$(i\nabla +A_a) \varphi_n^a$ for every
$R\in(0,1)$. Hence (v) follows from (iii).
\end{proof}

\begin{Corollary}\label{c:cor1}
  Under the same assumptions and with the same notation as in Lemma
  \ref{lemma:stimbe}, let $R\in(0,1)$. Then, for all $r\in(0,R)$ and $t\in[0,2\pi]$,
\begin{align}\label{eq:3}
\varphi_n^a(a+r(\cos t,\sin t))
& = \tfrac{1}{\sqrt{2\pi}} \left(r^{\alpha} \beta_0^a + r^{1-\alpha}\beta_1^a e^{it}\right) + \mathcal R_a(r,t),\\
\label{eq:14} (i\nabla +A_a)\varphi_n^a(a+r(\cos t,\sin t))
& = \tfrac{1}{\sqrt{2\pi}} r^{\alpha-1} \beta_0^a \alpha 
\Big( i {\boldsymbol{\nu}}(t) + {\boldsymbol{\tau}}(t) \Big) 
\\
\notag &\quad + \tfrac{1}{\sqrt{2\pi}} r^{-\alpha}\beta_1^a (1-\alpha)
\Big( i {\boldsymbol{\nu}}(t) - {\boldsymbol{\tau}}(t) \Big) 
e^{it} + \widetilde{\mathcal R}_a(r,t)
\end{align}
where 
$|\mathcal R_a(r,t)|\leq  \text{\rm const\,}r^{1+\sqrt{\mu_1}}$
and
$|\widetilde{\mathcal R}_a(r,t)|\leq  \text{\rm const\,}r^{\sqrt{\mu_1}}$ for
some $\text{\rm const\,}>0$ independent of $a,r,t$.  
\end{Corollary}
\begin{proof}
  From part (iv) of Lemma \ref{lemma:stimbe} we have that 
\begin{equation*}
\varphi_n^a(a + r (\cos t,\sin t))= \frac1{\sqrt{2\pi}}
\left(\beta_0^a
  r^{\alpha} + \beta_1^a
  r^{1-\alpha} e^{it}\right)
+\mathcal R_a(r,t), \quad r\in(0,1),\ t\in[0,2\pi],
\end{equation*}
where 
\[
\mathcal R_a(r,t)
=\frac1{\sqrt{2\pi}}
\left(\beta_0^a
  r^{\alpha} R_{0,a}(r)+ \beta_1^a
  r^{1-\alpha} R_{1,a}(r)e^{it}
\right)+ \frac1
{\sqrt{2\pi}}\sum_{\substack{ j \in \Z \, : \\ |\alpha-j|>1} } 
\beta_j^a
  r^{|\alpha-j|}\big(1+R_{j,a}(r)\big)e^{ijt}.
\]
Let us fix $R\in(0,1)$. Estimates (ii)--(iii)  of Lemma \ref{lemma:stimbe} imply that, for
some $\textrm{const\,}>0$ independent of $a,r,t$ (possibly varying
from line to line), 
\begin{align*}
|\mathcal R_a(r,t)|&\leq
\textrm{const\,}\bigg(r^{\alpha+2}+r^{3-\alpha}
+\sum_{  \substack{j \in \Z \, :\\  |\alpha-j|\geq 1+\sqrt{\mu_1}}  }
  r^{|\alpha-j|}\bigg)\\
& \leq
\textrm{const\,}r^{1+\sqrt{\mu_1}},
\end{align*}
for all $r\in(0,R)$,
thus proving \eqref{eq:3}.  

  From part (v) of Lemma \ref{lemma:stimbe} we have that 
\begin{multline*}
(i\nabla+A_a)\varphi_n^a(a + r (\cos t,\sin t))\\ = \frac\alpha{\sqrt{2\pi}}
\beta_0^a
  r^{\alpha-1}\Big(i{\boldsymbol{\nu}}(t)+
 {\boldsymbol{\tau}}(t)\Big) + \frac{1-\alpha}
{\sqrt{2\pi}}\beta_1^a
  r^{-\alpha}\Big(i{\boldsymbol{\nu}}(t)-
  {\boldsymbol{\tau}}(t)\Big) e^{it}
+\widetilde{\mathcal R}_a(r,t)
\end{multline*}
where 
\begin{align*}
\widetilde{\mathcal R}_a(r,t)
&=\frac\alpha{\sqrt{2\pi}}
\beta_0^a
  r^{\alpha-1}\Big(i\widetilde R_{0,a}(r){\boldsymbol{\nu}}(t)+
  R_{0,a}(r) 
{\boldsymbol{\tau}}(t)\Big)\\
&\quad + \frac{1-\alpha}
{\sqrt{2\pi}}\beta_1^a
  r^{-\alpha}\Big(i\widetilde R_{1,a}(r){\boldsymbol{\nu}}(t)-
   R_{1,a}(r)
{\boldsymbol{\tau}}(t)\Big) e^{it}\\
&\quad+ \frac1
{\sqrt{2\pi}}\sum_{  \substack{j \in \Z \, :\\   |\alpha-j|>1}  }
\beta_j^a
  r^{|\alpha-j|-1}\Big(i|\alpha-j|(1+\widetilde R_{j,a}(r)){\boldsymbol{\nu}}(t)+(\alpha-j)
  ( 1 + R_{j,a}(r)) 
{\boldsymbol{\tau}}(t)\Big) e^{ijt}.
\end{align*}
From Lemma \ref{lemma:stimbe} (ii)--(iii) we have that, for all $r\in(0,R)$,
\begin{align*}
|\widetilde{\mathcal R}_a(r,t)|&\leq
\textrm{const\,}\left(r^{\alpha+1}+r^{2-\alpha}
+\sum_{  \substack{j \in \Z \, :\\  |\alpha-j|\geq 1+\sqrt{\mu_1}}  }
  |\alpha-j| r^{|\alpha-j|-1}\right)\\
& \leq
\textrm{const\,}r^{\sqrt{\mu_1}}
\end{align*}
 for
some $\text{\rm const\,}>0$ independent of $a,r,t$ (possibly varying
from line to line), thus proving \eqref{eq:14}.  
\end{proof}

We now describe some consequences of Lemma \ref{lemma:stimbe} and
Corollary \ref{c:cor1}, which will be needed in section
\ref{sec:monotonicity} to prove a monotonicity type formula.

\begin{Lemma}\label{lemma:estep}
  Under the same assumptions and with the same notation as in Lemma \ref{lemma:stimbe}, we have
  that 
\begin{multline*}
\lim_{\eps \to 0^+} \left\{ \bigg|\frac12\int_{\partial D_\eps(a)} 
| ( i \nabla + A_a ) \varphi_n^a |^2 x \cdot \nu \, ds \bigg|
+ \bigg| \int_{\partial D_\eps(a)} ( i \nabla + A_a ) \varphi_n^a \cdot \nu 
\, \overline{ ( i \nabla + A_a ) \varphi_n^a \cdot x} \, ds \bigg| \right\}\\
\leq  2\alpha(1-\alpha)|a| |\beta_0^a||\beta_1^a|.
\end{multline*}
\end{Lemma}
\begin{proof}
Let $R\in(0,1)$ fixed. From \eqref{eq:14} we have that, for all $r\in(0,R)$, 
\begin{align*}
|(i\nabla +A_a)\varphi_n^a(a+r(\cos t,r\sin t))|^2 
& = r^{2 (\alpha - 1)} |\beta_0^a|^2 \frac{\alpha^2}{\pi} 
+ r^{-2\alpha} |\beta_1^a|^2 \frac{(1-\alpha)^2}{\pi} + \widehat{\mathcal R}_a(r,t)
\end{align*}
where $|\widehat{\mathcal R}_a(r,t)| \leq  \textrm{const\,}r^{2\sqrt{\mu_1}-1}$ for 
some $\textrm{const\,}>0$ independent of $a,r,t$. It follows that 
\[
\lim_{\eps\to 0^+}\int_{\partial
  D_\eps(a)}|(i\nabla+A_a)\varphi_n^a|^2 x\cdot\nu\,ds = 0.
\]
Moreover, from \eqref{eq:14} we have that 
\begin{align*}
&  (i\nabla+A_a)\varphi_n^a(a+\eps(\cos t,\sin t))\cdot
{\boldsymbol{\nu}}(t)\overline{(i\nabla+A_a)\varphi_n^a(a+\eps(\cos
  t,\sin t))\cdot (a+\eps {\boldsymbol{\nu}}(t))} \\
& = \eps^{2(\alpha-1)} |\beta_0^a|^2 \frac{\alpha^2}{2 \pi} 
( \boldsymbol{\nu}(t) + i  \boldsymbol{\tau}(t) )\cdot a + \eps^{-2 \alpha} |\beta_1^a|^2 \frac{(1-\alpha)^2}{2\pi} 
( \boldsymbol{\nu}(t) - i \boldsymbol{\tau}(t) ) \cdot a \\
& + 2 \eps^{-1} \mathfrak{Re} \left( \beta_0^a \overline{\beta_1^a} e^{-it} \frac{\alpha (1 - \alpha)}{2\pi} 
( \boldsymbol{\nu}(t) - i \boldsymbol{\tau}(t) )\cdot a \right)
+ O(\eps^{2\sqrt{\mu_1}-1})
\end{align*}
as $\eps\to0^+$, and hence, taking into account that $\int_0^{2\pi}a\cdot
{\boldsymbol{\nu}}(t)\,dt=\int_0^{2\pi}a\cdot
{\boldsymbol{\tau}}(t)\,dt=0$, we obtain
\[
\lim_{\eps\to 0^+}\int_{\partial D_\eps(a)} (i\nabla+A_a)\varphi_a \cdot \nu
\overline{(i\nabla+A_a)\varphi_a\cdot x}\,ds = 2 \alpha (1- \alpha)  
\mathfrak{Re} ( \beta_0^a \overline{\beta_1^a} (a_1 - ia_2) )
\]
from which the conclusion follows.
\end{proof}

\begin{Lemma} \label{lemma:pre-lunette}
For $n\geq1$ fixed and $a$ varying in $\Omega$, let $\varphi_n^a\in
H^{1,a}_{0}(\Omega,\C)\setminus\{0\}$ satisfy
\eqref{eq_eigenfunction_n} and \eqref{eq:23_n}. Let us assume that
$\varphi_n^a\to \varphi_n^0$ in $L^2(\Omega,\C)$ as $a\to0$ (or
respectively 
along a sequence $a_\ell\to0$). Let
  $k \in \mathbb{Z}$ be such that $|\alpha - k|$ is the order of
  vanishing of $\varphi_n^0$ at $0$. 
For all $j \in \Z$ and $a \in \Omega$, let $v_j^a$ be as in
\eqref{eq:def_vja} and $\beta_j^a$ be as in \eqref{eq:beta-2}.
Then there exists $\rho_0>0$ such that, for all $a$  with
$|a|\leq\rho_0$ (respectively for $a=a_\ell$ with $\ell$ sufficiently large),  the following properties hold.
\begin{enumerate}[\rm (i)]
\item For all $j \in \Z$, $\beta_j^a \to \beta_j^0$ as $a\to0$ (respectively 
along the sequence $a_\ell\to0$).
\item There holds
\begin{align*}
\int_0^{2\pi}&|\varphi_n^a( a + r(\cos t,\sin t))|^2\,dt\\
& =\bigg( \sum_{\substack{ j \in \Z \, : \\ |\alpha-j| < |\alpha-k| } } 
r^{2|\alpha-j|} |\beta_j^a|^2 | 1 + R_{j,a}(r) |^2\bigg)
+ r^{2|\alpha-k|} |\beta_k^a|^2 (1+\widehat R_a(r)),
\end{align*}
where $|\widehat R_{a}(r)|\leq  h(r)$ 
for some function $h(r)$
independent of $a$ such that $h(r)\to0$ as $r \to 0^+$, 
and
\begin{align*}
\varphi_n^a&( a + r(\cos t,\sin t)) \\
&=\frac1{\sqrt{2\pi}}\bigg( \sum_{\substack{ j \in \Z \, : \\ |\alpha-j| < |\alpha-k| } } 
r^{|\alpha-j|} \beta_j^a ( 1 + R_{j,a}(r) ) e^{ijt} \bigg)
+ \frac1{\sqrt{2\pi}} r^{|\alpha-k|} \beta_k^a \left( e^{ikt} + R_{a}(r,t) \right),
\end{align*}
where $|R_{j,a}(r)| \leq  {\rm const\,}r^2$ for some ${\rm const} > 0$ independent 
of $j$ and $a$, and $|R_{a}(r,t)|\leq  f(r)$ for some function $f(r)$
independent of $a$ and $t$
such that $f(r)\to0$ as $r \to 0^+$. 
\item Let ${\boldsymbol{\nu}}(t)=(\cos t,\sin t)$ and ${\boldsymbol{\tau}}(t)=(-\sin t,\cos t)$. 
There holds
\begin{align*}
\int_0^{2\pi}&|(i\nabla+A_a)\varphi_n^a( a + r(\cos t,\sin t))|^2\,dt
 \\&=\bigg(\sum_{\substack{ j \in \Z \, : \\ |\alpha-j| < |\alpha-k| } } 
r^{2|\alpha-j|-2} |\beta_j^a|^2 |\alpha-j|^2\Big(| 1 + R_{j,a}(r)
     |^2+| 1 + \widetilde R_{j,a}(r) |^2\Big)\bigg)\\
&\quad 
+r^{2|\alpha-k|-2} |\beta_k^a|^2 |\alpha-k|^2(1+\widetilde R_a(r))
\end{align*}
where $|\widetilde R_{a}(r)|\leq p (r)$ 
for some function $p(r)$
independent of $a$ such that $p(r)\to0$ as $r \to 0^+$, 
and

\begin{align*}
(i\nabla + A_a) \varphi_n^a & ( a + r (\cos t,\sin t) ) \\
& = \frac{1}{\sqrt{2\pi}} \sum_{\substack{ j \in \Z \, : \\ |\alpha-j| < |\alpha-k| } } 
r^{|\alpha-j|-1} \beta_j^a \Big( i |\alpha-j| {\boldsymbol{\nu}}(t) 
+ (\alpha -j) {\boldsymbol{\tau}}(t) 
+ {\mathbf R}_{j,a}(r) \Big) e^{ijt} \\
&\qquad + \frac{1}{\sqrt{2\pi}} r^{|\alpha-k|-1} \beta_k^a \bigg( \Big( i |\alpha-k| {\boldsymbol{\nu}}(t) 
+ (\alpha - k) {\boldsymbol{\tau}}(t)\Big) e^{ikt}
+ {\widetilde{\mathbf R}}_{a}(r,t) \bigg)
\end{align*}
where $|{\mathbf R}_{j,a}(r,t)| \leq \textrm{const
}r^2$ for some positive constant  $\textrm{const }>0$ 
independent of $j$ and $a$ and $|{\widetilde{\mathbf R}}_{a}(r,t)|\leq  g(r)$ for some function 
$g(r)$ independent of  $a$ and $t$ such that $g(r) \to 0$ as $r \to
0^+$. 
\end{enumerate}

\end{Lemma}

\begin{proof}
In order to prove statement (i), we notice that
\eqref{eq:beta-2} evaluated at $R=1$ provides
\begin{equation}\label{eq:11}
\beta_j^a  = v_j^a(1) + 
\frac{\lambda_n^a}{2 |\alpha-j|} \int_0^1 \big( s^{1-|\alpha-j|} 
- s^{ 1 + |\alpha-j| } \big) v_j^a(s) \,ds.
\end{equation}
From  Lemma
\ref{lemma:stimbe} (statements (ii) and (iii)) it follows that, for
$|a|\leq\rho_0$ with $\rho_0>0$ sufficiently small,
\begin{equation}\label{eq:22}
|v_j^a(r)|\leq C' r^{|\alpha-j|} \quad \text{ for all } r \in (0,1] \text{ and }j\in\Z,
\end{equation}
for some constant $C'>0$ independent of $j$, $a$, and $r$. 
Moreover \eqref{eq_eigenfunction_n}, \eqref{eq:23_n}, the convergence
$\varphi_n^a\to \varphi_n^0$ in $L^2(\Omega,\C)$, and 
standard elliptic estimates (see e.g. \cite[Theorem
8.10]{gilbarg-trudinger}), imply that 
\begin{equation}\label{eq:2}
\varphi_n^a \to \varphi_n^0 \quad \text {in } H^1(\Omega,\C) 
\text{ and } C^2_{\rm loc}(\Omega\setminus\{0\},\C), \quad\text{as  $a\to0$ (or
along the sequence $a_\ell\to0$)}.
\end{equation}
From \eqref{eq:11},  \eqref{eq:22}, \eqref{eq:2},  and the Dominated
Convergence Theorem we obtain that, for all $j\in\Z$,
\begin{align*}
& \lim_{a\to 0}\beta_j^a 
= 
v_j^0(1) + 
\frac{\lambda_n^0}{2 |\alpha-j|} \int_0^1 \big( s^{1-|\alpha-j|} 
- s^{ 1 + |\alpha-j| } \big) v_j^0(s) \,ds
= \beta_j^0,
\end{align*}
thus proving (i).

If 
  $k \in \mathbb{Z}$ is such that $|\alpha - k|$ is the order of
  vanishing of $\varphi_n^0$ at $0$, from Lemma \ref{lemma:stimbe} (iii) it follows that $\beta_k^0\neq0$ 
and $\beta_j^0=0$ for all $j\in\Z$ such that $|\alpha-j|<|\alpha-k|$;
in particular, in view of (i), we have that $\lim_{a\to
  0}\beta_k^a\neq0$ and hence $\inf_{|a|\leq\rho_0}|\beta_k^a|>0$ for $\rho_0$ sufficiently small.
Then, from Lemma \ref{lemma:stimbe} (iv) and the
Parseval identity we deduce that
\begin{align*}
\int_0^{2\pi}&|\varphi_n^a( a + r(\cos t,\sin t))|^2\,dt
 =\sum_{  j \in \Z}
r^{2|\alpha-j|} |\beta_j^a|^2 | 1 + R_{j,a}(r) |^2\\
&=
\left( \sum_{  \substack{j \in \Z \, :\\   |\alpha-j|<|\alpha-k|}  }
r^{2|\alpha-j|} |\beta_j^a|^2 | 1 + R_{j,a}(r) |^2\right)
+ r^{2|\alpha-k|} |\beta_k^a|^2 (1+\widehat R_a(r)),
\end{align*}
with 
\[
\widehat R_a(r)=|R_{k,a}(r) |^2+2 \mathfrak{Re}(R_{k,a}(r) )+\sum_{  \substack{j \in \Z \, :\\   |\alpha-j|>|\alpha-k|}  }
\frac{|\beta_j^a|^2}{|\beta_k^a|^2} r^{2|\alpha-j|-2|\alpha-k|} | 1 + R_{j,a}(r) |^2,
\]
so that the first estimate in (ii) follows from Lemma
\ref{lemma:stimbe} (ii) and (iii). 
From Lemma \ref{lemma:stimbe} (iii) we also deduce that
\[
\frac1{\sqrt{2\pi}}
  \sum_{  \substack{j \in \Z \, :\\   |\alpha-j|\geq|\alpha-k|}  }
r^{|\alpha-j|}\beta_j^a(1+R_{j,a}(r))e^{ijt}=
\frac1{\sqrt{2\pi}}r^{|\alpha-k|}\beta_k^a \left(e^{ikt}+R_{a}(r,t)\right) 
\]
where $|R_{a}(r,t)|\leq  f(r)$ for some function $f(r)$ 
independent of $a$ and $t$ such that $f(r)\to0$ as $r\to0$.  
Then the second estimate in (ii) follows from Lemma \ref{lemma:stimbe}
(iv).

From Lemma \ref{lemma:stimbe} (v) and the
Parseval identity we deduce that
\begin{align*}
\int_0^{2\pi}&|(i\nabla+A_a)\varphi_n^a( a + r(\cos t,\sin t))|^2\,dt
 \\&=\sum_{  j \in \Z}
r^{2|\alpha-j|-2} |\beta_j^a|^2 |\alpha-j|^2\Big(| 1 + R_{j,a}(r) |^2+| 1 + \widetilde R_{j,a}(r) |^2\Big)
\end{align*}
so that the first estimate in (iii) follows from Lemma
\ref{lemma:stimbe} (ii) and (iii) arguing as above. 
In a similar way, the second  estimate in (iii) follows from statements (iii) and (v) of Lemma \ref{lemma:stimbe}.
\end{proof}

\begin{remark}\label{rem:van_ord}
  In the particular case $n=n_0$ with $n_0$ such that \eqref{eq:1}
  holds, the above lemma applies to the family of eigenfunctions
  $\varphi_a=\varphi_{n_0}^a$ satisfying \eqref{eq:equation_a} and
  \eqref{eq:normalization}. Indeed, in this case
  \eqref{eq:convergence-varphi-a-1} holds, i.e. the eigenfunctions
  $\varphi_a$ converge as $a\to0^+$, so that the assumptions of Lemma
  \ref{lemma:pre-lunette} are fulfilled.  In particular we deduce
  that, if $\varphi_0$ satisfies \eqref{eq:equation_lambda0},
  \eqref{eq:37}, and \eqref{eq:131} and if $\varphi_a$ is as in
  \eqref{eq:equation_a}--\eqref{eq:normalization}, then, for $a$
  sufficiently close to $0$, the vanishing order of $\varphi_a$ is
  less or equal than the vanishing order of $\varphi_0$.
\end{remark}

\begin{Lemma} \label{lemma:prepre-lunette}
For $n\geq1$ fixed and $a$ varying in $\Omega\setminus\{0\}$, let $\varphi_n^a\in
H^{1,a}_{0}(\Omega,\C)\setminus\{0\}$ satisfy
\eqref{eq_eigenfunction_n} and \eqref{eq:23_n}. Then there exist
$\sigma>0$ and $C>0$ such that, for all $R>1$ and $a\in\Omega$ such
that $0<|a|<\frac{\sigma}{R}$,
\[
\frac1{|a|}\int_{D_{(R+1)|a|}(a) \setminus D_{R|a|}(a)}|\varphi_n^a|^2\,dx\leq C \int_{\partial
  D_{R|a|}(a)} |\varphi_n^a|^2 \,ds
\]
and
\[
\int_{D_{(R+1)|a|}(a) \setminus D_{R|a|}(a)}|(i\nabla+A_a)\varphi_n^a|^2\,dx\leq \frac{C}{R^2|a|} \int_{\partial
  D_{R|a|}(a)} |\varphi_n^a|^2 \,ds.
\]
\end{Lemma}
\begin{proof}
  Let us prove the first estimate arguing by contradiction: assume that there exist sequences
  $R_\ell>1$ and $a_\ell\in\Omega$ such that 
$R_\ell|a_\ell|<\frac1{\ell}$ and 
 \begin{equation}\label{eq:9}
\frac1{|a_\ell|}\int_{D_{(R_\ell+1)|a_\ell|}(a_\ell) \setminus D_{R_\ell|a_\ell|}(a_\ell)}|\varphi_n^{a_\ell}|^2\,dx> \ell \int_{\partial
  D_{R_\ell|a_\ell|}(a_\ell)} |\varphi_n^{a_\ell}|^2 \,ds.
\end{equation}
It is easy to verify that, up to extracting a subsequence,
$\varphi_n^{a_\ell}\to \varphi_n^0$ in $L^2(\Omega,\C)$ as $\ell\to\infty$ for some
$\varphi_n^0\in H^{1,0}_{0}(\Omega,\C)\setminus\{0\}$ satisfying 
\begin{equation}\label{eq:10}
 \begin{cases}
   (i\nabla + A_0)^2 \varphi_n^0 = \lambda_n^0 \varphi_n^0,  &\text{in }\Omega,\\
   \varphi_n^0 = 0, &\text{on }\partial \Omega,\\
  \int_\Omega |\varphi_n^0(x)|^2\,dx=1.
\end{cases}
\end{equation}
Let
  $k \in \mathbb{Z}$ be such that $|\alpha - k|$ is the order of
  vanishing of $\varphi_n^0$ at $0$. Then, from Lemma \ref{lemma:pre-lunette} (first estimate in (ii)) it
follows that, for $\ell$ sufficiently large, 
\begin{align*}
\frac1{|a_\ell|}&\int_{D_{(R_\ell+1)|a_\ell|}(a_\ell) \setminus D_{R_\ell|a_\ell|}(a_\ell)}|\varphi_n^{a_\ell}|^2\,dx=  \frac1{|a_\ell|}\int_{R_\ell|a_\ell|}^{(R_\ell+1)|a_\ell|}r\left(\int_0^{2\pi}|\varphi_n^{a_\ell}( a_\ell
  + r(\cos t,\sin t))|^2\,dt\right)\,dr\\
&\leq
  \frac{2}{|a_\ell|}\int_{R_\ell|a_\ell|}^{(R_\ell+1)|a_\ell|}r\left(
\sum_{  \substack{j \in \Z \, :\\  |\alpha-j| \leq |\alpha-k|}  } 
\!\!\!\!r^{2|\alpha-j|} |\beta_j^{a_\ell}|^2 \right)\,dr\leq {\rm
  const\,}\!\!\!\!
\sum_{  \substack{j \in \Z \, :\\  |\alpha-j| \leq |\alpha-k|}  }  \!\!\!\!
(R_\ell|a_\ell|)^{1+2|\alpha-j|} |\beta_j^{a_\ell}|^2
\end{align*}
 for some positive constant  $\textrm{const }>0$ 
independent of $\ell$,
while
\begin{align}\label{eq:12}
\int_{\partial
  D_{R_\ell|a_\ell|}(a_\ell)} |\varphi_n^{a_\ell}|^2 \,ds&=
R_\ell|a_\ell|\int_0^{2\pi}|\varphi_n^{a_\ell}( a_\ell
  + R_\ell|a_\ell| (\cos t,\sin t))|^2\,dt\\
&\notag\geq \frac{R_\ell|a_\ell|}2 \sum_{  \substack{j \in \Z \, :\\  |\alpha-j| \leq |\alpha-k|}  }  \!\!\!\!
(R_\ell|a_\ell|)^{2|\alpha-j|} |\beta_j^{a_\ell}|^2,
\end{align}
thus contradicting \eqref{eq:9} as $\ell\to\infty$.

To prove the second estimate, let us assume by contradiction that there exist sequences
  $R_\ell>1$ and $a_\ell\in\Omega$ such that 
$R_\ell|a_\ell|<\frac1{\ell}$ and 
 \begin{equation}\label{eq:9-b}
\int_{D_{(R_\ell+1)|a_\ell|}(a_\ell) \setminus D_{R_\ell|a_\ell|}(a_\ell)}|(i\nabla+A_{a_\ell})\varphi_n^{a_\ell}|^2\,dx> \frac{\ell}{R_\ell^2|a_\ell|} \int_{\partial
  D_{R_\ell|a_\ell|}(a_\ell)} |\varphi_n^{a_\ell}|^2 \,ds.
\end{equation}
As above we have that, up to extracting a subsequence,
$\varphi_n^{a_\ell}\to \varphi_n^0$ in $L^2(\Omega,\C)$ as $\ell\to\infty$ for some
$\varphi_n^0\in H^{1,0}_{0}(\Omega,\C)\setminus\{0\}$ satisfying \eqref{eq:10}.
Then, from Lemma \ref{lemma:pre-lunette} (first estimate in (iii)) it
follows that, for $\ell$ sufficiently large and  for some positive constant  $\textrm{const }>0$ 
independent of $\ell$,
\begin{align*}
&\int_{D_{(R_\ell+1)|a_\ell|}(a_\ell) \setminus
  D_{R_\ell|a_\ell|}(a_\ell)}|(i\nabla+A_{a_\ell})\varphi_n^{a_\ell}|^2\,dx\\
&=\int_{R_\ell|a_\ell|}^{(R_\ell+1)|a_\ell|}r\left(\int_0^{2\pi}|(i\nabla+A_{a_\ell})\varphi_n^{a_\ell}( a_\ell
  + r(\cos t,\sin t))|^2\,dt\right)\,dr\\
&\leq
  3 |\alpha-k|^2\int_{R_\ell|a_\ell|}^{(R_\ell+1)|a_\ell|}r\left(
\sum_{  \substack{j \in \Z \, :\\  |\alpha-j| \leq |\alpha-k|}  } 
\!\!\!\!r^{2|\alpha-j|-2} |\beta_j^{a_\ell}|^2 \right)\,dr\leq \frac{\rm
  const\,}{R_\ell}\!\!\!\!
\sum_{  \substack{j \in \Z \, :\\  |\alpha-j| \leq |\alpha-k|}  }  \!\!\!\!
(R_\ell|a_\ell|)^{2|\alpha-j|} |\beta_j^{a_\ell}|^2
\end{align*}
which, in view of \eqref{eq:12}, contradicts \eqref{eq:9-b} as $\ell\to\infty$.
\end{proof}

\begin{remark} \label{rem:prepre-lunette}
Arguing as in Lemma \ref{lemma:prepre-lunette}, we can also prove the
following similar estimate (possibly taking a smaller $\sigma$ and a
larger $C$ if necessary):  for all $R>1$ and $a\in\Omega$ such
that $0<|a|<\frac{\sigma}{R}$
\[
\frac1{|a|}\int_{D_{(R+1)|a|}(a) \setminus D_{R|a|}(a)}|\varphi_n^a|^2\,dx\leq C \int_{\partial
  D_{(R+1)|a|}(a)} |\varphi_n^a|^2 \,ds
\]
and
\[
\int_{D_{(R+1)|a|}(a) \setminus D_{R|a|}(a)}|(i\nabla+A_a)\varphi_n^a|^2\,dx\leq \frac{C}{R^2|a|} \int_{\partial
  D_{(R+1)|a|}(a)} |\varphi_n^a|^2 \,ds.
\]
\end{remark}

\begin{Lemma} \label{lemma:lunette}
For $n\geq1$ fixed, let $\varphi_n^a$ be a solution to 
\eqref{eq_eigenfunction_n}--\eqref{eq:23_n}. 
Let 
$\sigma>0$ and $C>0$ be as in Lemma
\ref{lemma:prepre-lunette} and Remark \ref{rem:prepre-lunette}. Then, for all $R>2$ and $a\in\Omega$ such
that $0<|a|<\frac{\sigma}{R}$,
\[
\left| \int_{\partial D_{R|a|}(0)}|\varphi_n^a|^2\,ds-
\int_{\partial D_{R|a|}(a)}|\varphi_n^a|^2\,ds \right|
\leq \frac{1+6C}{R-2}\int_{\partial D_{R|a|}(a)} |\varphi_n^a|^2 \,ds.
\]
\end{Lemma}

\begin{proof}
We note that 
\begin{equation} \label{eq:difference-2-terms}
\int_{\partial D_{R|a|}(0)}|\varphi_n^a|^2\,ds-
\int_{\partial D_{R|a|}(a)}|\varphi_n^a|^2\,ds
=\int_{\partial \mathcal L_{1,R}^a}|\varphi_n^a|^2\tilde\nu\cdot\hat\nu ds-
\int_{\partial \mathcal L_{2,R}^a}|\varphi_n^a|^2\tilde\nu\cdot(-\hat\nu) ds.
\end{equation}
where 
\[
\mathcal L_{1,R}^a=D_{R|a|}(0)\setminus D_{R|a|}(a),\quad 
\mathcal L_{2,R}^a=D_{R|a|}(a)\setminus D_{R|a|}(0),
\]
and
\begin{align*}
&\hat\nu(x)=
\begin{cases}
  \frac{x}{|x|},&\text{on }\partial D_{R|a|}(0),\\
  -\frac{x-a}{|x-a|},&\text{on }\partial D_{R|a|}(a),
\end{cases}
\quad
\tilde\nu(x)=
\begin{cases}
  \frac{x}{|x|},&\text{on }\partial D_{R|a|}(0),\\
  \frac{x-a}{|x-a|},&\text{on }\partial D_{R|a|}(a).
\end{cases}
\end{align*}
We note that $\hat\nu$ is the outer unit normal vector on 
$\partial \mathcal L_{1,R}^a$ and $-\hat\nu$ is the outer 
unit normal vector on $\partial \mathcal L_{2,R}^a$.
By denoting $\nu_1(x)=x/|x|$, we can rewrite the right hand side of \eqref{eq:difference-2-terms} as
\begin{multline}\label{eq:15}
\int_{\partial \mathcal{L}_{1,R}^a} |\varphi_n^a|^2 (\tilde{\nu} - \nu_1) \cdot \hat\nu \, ds + 
\int_{\partial \mathcal{L}_{1,R}^a} |\varphi_n^a|^2 \nu_1 \cdot \hat\nu \, ds\\
+ \int_{\partial \mathcal{L}_{2,R}^a} |\varphi_n^a|^2 (\tilde{\nu} - \nu_1) \cdot \hat\nu \, ds 
- \int_{\partial \mathcal{L}_{2,R}^a} |\varphi_n^a|^2 \nu_1 \cdot
(-\hat\nu) \, ds\\
=\int_{\partial \mathcal{L}_{1,R}^a} |\varphi_n^a|^2 \nu_1 \cdot \hat\nu \, ds
- \int_{\partial \mathcal{L}_{2,R}^a} |\varphi_n^a|^2 \nu_1 \cdot
(-\hat\nu) \, ds
\\+\int_{\partial D_{R|a|}(0)} |\varphi_n^a|^2 (\tilde{\nu} - \nu_1)
\cdot \hat\nu \, ds +\int_{\partial D_{R|a|}(a)} |\varphi_n^a|^2
(\tilde{\nu} - \nu_1) \cdot \hat\nu \, ds .
\end{multline}
We observe that 
\begin{equation*}
(\tilde{\nu}(x) - \nu_1(x)) \cdot \hat\nu(x) 
= 
\begin{cases}
0, & \text{on } \partial D_{R|a|}(0), \\
- 1 + \frac{x \cdot (x-a)}{|x| |x-a|}, &\text{on } \partial D_{R|a|}(a).
\end{cases}
\end{equation*}
Moreover, since $\nu_1$ is smooth in $\overline{\mathcal{L}_{1,R}^a}\cup\overline{\mathcal{L}_{2,R}^a}$, 
we can apply the Divergence Theorem to the first two terms in the
right hand side of \eqref{eq:15}, thus rewriting the right hand side of \eqref{eq:difference-2-terms} as
\begin{equation}\label{eq:diff-2-terms2}
- \int_{\partial D_{R|a|}(a)} |\varphi_n^a|^2 
\left(1 - \frac{x \cdot (x-a)}{|x| |x-a|} \right) \,ds 
+ \int_{\mathcal{L}_{1,R}^a} \operatorname{div} ( |\varphi_n^a|^2 \nu_1 ) \,dx 
- \int_{\mathcal{L}_{2,R}^a} \operatorname{div} ( |\varphi_n^a|^2 \nu_1 ) \,dx.
\end{equation}

\noindent \textbf{Estimate of the first term in \eqref{eq:diff-2-terms2}.}  
Parametrizing $\partial D_{R|a|}(a)$ as $x= a +R|a| (\cos t, \sin t)$ 
and writing $a = |a| (\cos \theta_a, \sin \theta_a)$ for some angle 
$\theta_a \in [0,2\pi)$, we get
\[
\left| 1 - \frac{x \cdot (x-a)}{|x| |x-a|} \right| 
=\left|  1 - \frac{R + \cos(t-\theta_a)}
{\left( R^2+2R\cos(t-\theta_a)+1 \right)^{1/2}} \right| \leq \frac{1}{R-1}
\]
on $\partial D_{R|a|}(a)$. Therefore,
\begin{equation} \label{eq:1term}
\left| - \int_{\partial D_{R|a|}(a)} |\varphi_n^a|^2  
\left( 1 - \frac{x \cdot (x-a)}{|x| |x-a|} \right) \,ds \right| 
\leq\frac{1}{R-1} \int_{\partial D_{R|a|}(a)} |\varphi_n^a|^2 \,ds.
\end{equation}

\noindent \textbf{Estimate of the second term in \eqref{eq:diff-2-terms2}.}
The second term in \eqref{eq:diff-2-terms2} splits into two parts
\begin{equation*}
\int_{\mathcal{L}_{1,R}^a} \operatorname{div} \left( |\varphi_n^a|^2 \nu_1 \right) \, dx 
= \int_{\mathcal{L}_{1,R}^a} \frac{|\varphi_n^a|^2}{|x|} \, dx + 
\int_{\mathcal{L}_{1,R}^a} 2 \mathfrak{Re} 
(i \varphi_n^a  \overline{(i \nabla + A_a)\varphi_n^a} \cdot \nu_1) \, dx.
\end{equation*}
Since $D_{R|a|}(0) \subset
D_{(R+1)|a|}(a)$, we have that $\mathcal{L}_{1,R}^a\subseteq
D_{(R+1)|a|}(a) \setminus D_{R|a|}(a)$. 
Let $\sigma>0$ and $C>0$ be as in Lemma
\ref{lemma:prepre-lunette} and Remark \ref{rem:prepre-lunette}. Hence by Lemma
\ref{lemma:prepre-lunette} we have that, for all $R>1$ and $a\in\Omega$ such
that $0<|a|<\frac{\sigma}{R}$,
\begin{align*}
 \left| \int_{\mathcal{L}_{1,R}^a} \frac{|\varphi_n^a|^2}{|x|} \,dx \right| 
&\leq \int_{D_{(R+1)|a|}(a) \setminus D_{R|a|}(a)} 
\frac{|\varphi_n^a|^2}{|x|} \,dx\\
&\leq \frac{1}{(R-1)|a|} \int_{D_{(R+1)|a|}(a) \setminus D_{R|a|}(a)} 
|\varphi_n^a|^2 \,dx 
 \leq  \frac{C}{R-1} \int_{\partial
  D_{R|a|}(a)} |\varphi_n^a|^2 \,ds
\end{align*}
and 
\begin{align*}
&\left| \int_{\mathcal{L}_{1,R}^a} 2 \mathfrak{Re} 
  (i \varphi_n^a  \overline{(i \nabla + A_a)\varphi_n^a} \cdot \nu_1) \, dx \right| \\
&\quad\leq  2
\left(\int_{D_{(R+1)|a|}(a) \setminus D_{R|a|}(a)}|\varphi_n^a|^2\,dx\right)^{1/2}
\left(\int_{D_{(R+1)|a|}(a) \setminus D_{R|a|}(a)}|(i\nabla+A_a)\varphi_n^a|^2\,dx\right)^{1/2} \\
&\quad\leq  \frac{2C}{R} \int_{\partial
  D_{R|a|}(a)} |\varphi_n^a|^2 \,ds.
\end{align*}
Therefore,
\begin{equation} \label{eq:2term}
\left| \int_{\mathcal{L}_{1,R}^a} \operatorname{div} (|\varphi_n^a|^2 \nu_1) \,dx \right| 
\leq \frac{3C}{R-1} \int_{\partial D_{R|a|}(a)} |\varphi_n^a|^2 \,ds,
\end{equation}
for all $R>1$ and $a\in\Omega$ such
that $0<|a|<\frac{\sigma}{R}$.

\noindent\textbf{Estimate of the third term in \eqref{eq:diff-2-terms2}.}
The estimate of the third term can be derived in a similar way,
observing that, since $D_{R|a|}(0) \supset
D_{(R-1)|a|}(a)$, $\mathcal{L}_{2,R}^a\subseteq
D_{R|a|}(a) \setminus D_{(R-1)|a|}(a)$ and using Remark
\ref{rem:prepre-lunette} to obtain 
\begin{equation} \label{eq:3term}
\left| \int_{\mathcal{L}_{2,R}^a} \operatorname{div} (|\varphi_n^a|^2 \nu_1) \,dx \right| 
\leq  \frac{3C}{R-2} \int_{\partial D_{R|a|}(a)} |\varphi_n^a|^2 \,ds,
\end{equation}
 for all $R>2$ and $a\in\Omega$ such
that $0<|a|<\frac{\sigma}{R}$ (by possibly changing $C$ and $\sigma$).

Therefore combining \eqref{eq:1term}, \eqref{eq:2term} and 
\eqref{eq:3term} we complete the proof.
\end{proof}

\section{Monotonicity formula}\label{sec:monotonicity}

\subsection{Almgren type frequency function}

Arguing as in \cite[Lemma 3.1]{AbatangeloFelli2015-1}, 
one can easily prove the following \emph{Poincaré type inequality}
\begin{equation}\label{eq:poincare}
\dfrac{1}{r^2} \int_{D_r} \abs{u}^2\,dx \leq \dfrac1r 
\int_{\partial D_r} \abs{u}^2 \,ds + 
\int_{D_r} \abs{(i\nabla +A_{ a})u}^2 \, dx,
\end{equation} 
which holds for every $r>0$, $a\in D_r$, and $u \in H^{1,a}(D_r,
\C)$. 
Furthermore, defining, for every $b\in D_1$, 
\[
m_{b}: = \inf_{\stackrel{v\in H^{1,b}(D_1,\C)}{v\not\equiv 0}} 
\dfrac{\int_{D_1} \abs{(i\nabla +A_{b})v}^2\,dx}{\int_{\partial D_1} \abs{v}^2\,ds},
\] 
we have that the infimum $m_b$ is attained and $m_b>0$. 
Arguing as in \cite{AbatangeloFelli2015-1}, we can prove that 
$b\mapsto m_b$ is continuous in $D_1$ and that 
$m_0 = \sqrt{\mu_1}$ (with $\mu_1$ as in \eqref{eq:mu1_def}).
Therefore a standard dilation argument yields that,
for any $\delta \in(0,\sqrt{\mu_1})$, there exists some sufficiently 
large 
$\Upsilon_\delta>1$ such that, for every $r>0$ 
and $a\in D_r$ such that $\frac{|a|}r<\frac1{\Upsilon_\delta}$, 
\begin{equation}\label{eq:poincare2}
\frac{ \sqrt{\mu_1}-\delta}r\int_{\partial D_r} \abs{u}^2 \,ds 
\leq \int_{D_r} \abs{(i\nabla +A_{a})u}^2\,dx \quad
\text{ for all }u\in H^{1,a}(D_r,\C).
\end{equation}
For $\lambda \in \R$, $b \in \R^2$, $u \in H^{1,b}(D_r,\C)$ and 
$r>|b|$, we define the Almgren type frequency function as
\[
\mathcal{N}(u,r,\lambda,A_b) = \dfrac{E(u,r,\lambda,A_b)}{H(u,r)},
\]
where
\begin{align*}
E(u,r,\lambda,A_b) 
& = \int_{D_r} \Big[ \abs{ (i\nabla + A_b)u}^2 - \lambda  \abs{u}^2 \Big] \,dx, \\
H(u,r) 
& = \dfrac1r \int_{\partial D_r} \abs{u}^2\,ds .
\end{align*}
 For all $1\leq n\leq n_0$ and $a\in\Omega$, let $\varphi_n^a\in
H^{1,a}_{0}(\Omega,\C)\setminus\{0\}$ be an eigenfunction of problem
\eqref{eq:eige_equation_a} associated to the eigenvalue $\lambda_n^a$,
i.e. solving \eqref{eq_eigenfunction_n}, 
such that 
\begin{equation}\label{eq:23}
  \int_\Omega |\varphi_n^a(x)|^2\,dx=1\quad\text{and}\quad 
  \int_\Omega \varphi_n^a(x)\overline{\varphi_\ell^a(x)}\,dx=0\text{ if }n\neq\ell.
\end{equation}
For $n=n_0$, we choose 
\[
  \varphi_{n_0}^a=\varphi_a,
\]
with $\varphi_a$ as in
\eqref{eq:equation_a}--\eqref{eq:normalization}.  Let 
\begin{equation*}
  \Lambda=\sup_{\substack{a\in\Omega\\1\leq n\leq n_0}}\lambda_n^a\in(0,+\infty).
\end{equation*}
We recall that $\Lambda$ is finite in view of the continuity result of
the eigenvalue function $a\mapsto \lambda_n^a$ in $\overline{\Omega}$ proved in \cite[Theorem
1.1]{BonnaillieNorisNysTerracini2014}.

Arguing as in \cite[Lemma 5.2]{AbatangeloFelli2015-1}, we can prove that there exists
$0 < R_0 < \left( \Lambda \left(1 + \tfrac{2}{\sqrt{\mu_1}} \right) \right)^{-1/2}$ such
that $D_{R_0}\subset\Omega$ and, if $|a|<R_0$,
\begin{equation}\label{eq:70}
  H(\varphi_n^a,r)>0\quad\text{for all }r\in(|a|,R_0) \text{ and }1\leq
  n\leq n_0.
\end{equation}
Furthermore, for every $r\in(0,R_0]$ there exist $C_r>0$ and $\alpha_r\in(0,r)$ such that 
\begin{equation} \label{eq:lower-bound-H}
H(\varphi_n^a,r)\geq C_r\quad\text{for all $a$ with }|a|<\alpha_r \text{ and }1\leq
n\leq n_0.
\end{equation}
Thanks to \eqref{eq:70}, the function $r\mapsto
N(\varphi_n^a,r,\lambda_n^a,A_a)$ is well defined in~$(|a|,R_0)$.
By direct calculations (see
  \cite{NorisNysTerracini2015} for details), we can prove that
\begin{align}\label{eq:74}
& \dfrac{d}{dr} H(\varphi_n^a,r) 
= \dfrac2r E(\varphi_n^a,r,\lambda_n^a,A_a), \\ \label{eq:77}
& \dfrac{d}{dr} E(\varphi_n^a, r,\lambda_n^a,A_a) 
= 2\int_{\partial D_r} \abs{(i\nabla +A_a)\varphi_n^a \cdot\nu}^2\,ds
- \frac2r \left( M_n^a + \lambda_n^a \int_{D_r} \abs{\varphi_n^a}^2\,dx \right)
 \end{align}
where 
\begin{equation}\label{eq:Mja_def}
M_n^a =  \lim_{\eps \to 0^+}\int_{\partial D_\eps(a)}
\bigg( \Re \big( (i\nabla+A_a) \varphi_n^a \cdot \nu 
\overline{ (i\nabla+A_a) \varphi_n^a \cdot x} \big)
-\frac12 | (i\nabla+A_a) \varphi_n^a|^2 x \cdot \nu \bigg) \, ds.
\end{equation}
Lemma~\ref{lemma:lunette} together with Lemmas~\ref{lemma:stimbe} and \ref{lemma:estep} 
allow us to give an estimate of the quantity $M_n^a$ defined in \eqref{eq:Mja_def}.
We notice that the techniques used in \cite{AbatangeloFelli2015-1,NorisNysTerracini2015}
to estimate the term $M_n^a$ for $\alpha=1/2$ were based on the possibility of
rewriting the problem as a Laplace equation on the twofold covering;
hence it is not possibile here to
extend such proofs to the case $\alpha\not\in\frac{\Z}{2}$ and a new
strategy of proof is needed. 

\begin{Lemma} \label{lemma:estimate-Mja}
There exist $\sigma_0>0$ and $c_0>2$  such that, for every $1 \leq n \leq n_0$,
$R>c_0$ and $a\in \Omega$ such
 that $|a|<\frac{\sigma_0}{R}$,
\[
\frac{|M_n^a|}{H (\varphi_n^a, R|a|) } \leq \frac{2\alpha(1-\alpha)}{R-c_0}.
\]
\end{Lemma}
\begin{proof}
Let us fix $n\in \{1,2,\dots,n_0\}$ and define, for $|a|$ small and $r\in(0,1]$, 
\[
\widetilde{H}(\varphi_n^a, r) = \frac{1}{r} 
\int_{\partial D_{r}(a)} |\varphi_n^a|^2 \,ds.
\]
From the Parseval identity and Lemma \ref{lemma:stimbe} (iv) it follows that there exists
$\sigma_n>0$ such that, for every 
$R>2$ and $a\in \Omega$ such
 that $|a|<\frac{\sigma_n}{R}$,
 \begin{align}\label{eq:16}
   \widetilde{H}(\varphi_n^a, R|a|)&=\int_0^{2\pi}|\varphi_n^a( a + R|a|(\cos t,\sin t))|^2\,dt =\sum_{  j \in \Z}
(R|a|)^{2|\alpha-j|} |\beta_j^a|^2 \big| 1 + R_{j,a}(R|a|) \big|^2\\
\notag&\geq 
(R|a|)^{2\alpha} |\beta_0^a|^2 \big| 1 + R_{0,a}(R|a|)\big|^2+(R|a|)^{2(1-\alpha)} |\beta_1^a|^2 \big| 1 + R_{1,a}(R|a|)\big|^2\\
\notag&\geq \frac12\left(|\beta_0^a|^2 (R|a|)^{2\alpha}+|\beta_1^a|^2 (R|a|)^{2(1-\alpha)}\right)
 \end{align}
where the $\beta_j^a$'s are the coefficients defined in
\eqref{eq:beta-2} for the eigenfunction $\varphi_n^a$ (with $n$
fixed). From the elementary inequality $ab\leq \frac12(a^2+b^2)$, it
follows that 
\begin{equation}\label{eq:17}
  |\beta_0^a| |\beta_1^a| |a|=\frac1R |\beta_0^a|
  (R|a|)^{\alpha}|\beta_1^a| (R|a|)^{1-\alpha}
\leq \frac1{2R}\Big( |\beta_0^a|^2
  (R|a|)^{2\alpha}+|\beta_1^a|^2 (R|a|)^{2(1-\alpha)}\Big).
\end{equation}
Combining \eqref{eq:16} and \eqref{eq:17} we obtain that 
\begin{equation}\label{eq:MoverH1-1}
\frac{|\beta_0^a| |\beta_1^a| |a|}{\widetilde{H}(\varphi_n^a, R|a|)}\leq\frac1R.
\end{equation}
Moreover, Lemma \ref{lemma:estep} implies that
\begin{equation}\label{eq:MoverH1-2}
|M_n^a| \leq 2\alpha(1-\alpha) |\beta_0^a| |\beta_1^a| |a|. 
\end{equation}
 Lemma~\ref{lemma:lunette} provides some constant $c_n$ (independent of
$a$ and $R$) such that, for a possibly smaller $\sigma_n$ and for all
$R>2$ and $a\in\Omega$ such that $0<|a|<\frac{\sigma_n}R$,
\begin{equation}\label{eq:MoverH2}
\left| H(\varphi_n^a, R|a|) - \widetilde{H}(\varphi_n^a, R|a|) \right| 
\leq \frac{c_n}{R-2} \widetilde{H}(\varphi_n^a, R|a|) .
\end{equation}
Therefore, by combining \eqref{eq:MoverH1-1}, \eqref{eq:MoverH1-2}, and \eqref{eq:MoverH2}, we obtain
\[
\frac{|M_n^a|}{H(\varphi_n^a, R|a|)} \leq \frac{2\alpha(1-\alpha)}{R} 
\frac{1}{1 +  \frac{H(\varphi_n^a, R|a|) - \widetilde{H}(\varphi_n^a, R|a|)}
{\widetilde{H} (\varphi_n^a, R|a|)} } 
\leq 
\frac{2\alpha(1-\alpha)}{R}\frac{1}{1-\frac{c_n}{R-2}}
\leq \frac{2\alpha(1-\alpha)}{R-(2+c_n)}
\]
for all
$R>c_n+2$ and $a\in\Omega$ such that $0<|a|<\frac{\sigma_n}R$.

The conclusion then follows by repeating the argument for all
$n\in\{1,2,\dots,n_0\}$ and choosing $\sigma_0=\min\{\sigma_n:1\leq
n\leq n_0\}$ and $c_0=\max\{2+c_n:1\leq
n\leq n_0\}$.
\end{proof}

\begin{Lemma}\label{lemma:stima_H_sotto}
For $\delta \in (0,\frac12\sqrt{\mu_1})$, let $\Upsilon_\delta$ be such 
that \eqref{eq:poincare2} holds. Let $R_0$ be as above, 
$r_0 \leq R_0$ and $n \in \{1,\dots,n_0\}$. 
If $\Upsilon_\delta |a| \leq r_1 < r_2 \leq r_0$ and $\varphi_n^a$ is a 
solution to \eqref{eq_eigenfunction_n} satisfying \eqref{eq:23}, then
\begin{equation*}
\frac{H(\varphi_n^a,r_2)}{H(\varphi_n^a,r_1)} 
\geq e^{- \Lambda(2+\sqrt{\mu_1})   r_0^2} 
\left(\dfrac{r_2}{r_1} \right)^{2(\sqrt{\mu_1}-\delta)}.
\end{equation*}
\end{Lemma}
\begin{proof}
Combining \eqref{eq:poincare} with \eqref{eq:poincare2} we obtain that, 
for every  $\Upsilon_\delta|a|<r<r_0$,
\begin{align*} 
\dfrac{1}{r^2} \int_{D_r} \abs{\varphi_n^a}^2 \,dx &\leq
\left(1+\frac{1}{\sqrt{\mu_1} - \delta}\right) 
\int_{D_r} \abs{(i\nabla + A_a)\varphi_n^a}^2\,dx \\
&\leq
\left(1+\frac{2}{\sqrt{\mu_1}}\right) 
\int_{D_r} \abs{(i\nabla + A_a)\varphi_n^a}^2\,dx.
\end{align*}
From above, \eqref{eq:74} and \eqref{eq:poincare2},
we have that for every $\Upsilon_\delta|a|<r<r_0$ 
\begin{align*}
\dfrac{d}{dr}H(\varphi_n^a,r) 
& \geq \frac2r \left( 1- \Lambda r^2\Big(1+\tfrac2{\sqrt{\mu_1}}\Big) \right)
\int_{D_r} \abs{(i\nabla + A_a)\varphi_n^a}^2\,dx \\
& \geq \frac2r \left( 1- \Lambda r^2\Big(1+\tfrac2{\sqrt{\mu_1}}\Big) \right)
(\sqrt{\mu_1}-\delta) H(\varphi_n^a,r),
\end{align*}
so that, in view of \eqref{eq:70}, 
\[ 
\dfrac{d}{dr} \log H(\varphi_n^a,r) \geq
\frac{2}r(\sqrt{\mu_1}-\delta)-2\Lambda r(2+\sqrt{\mu_1}).
\]
Integrating between $r_1$ and $r_2$ we obtain the desired inequality.
\end{proof}

\begin{Lemma}\label{l:3.3}
For $n\in\{1,\dots,n_0\}$ and $a\in\Omega$, let $\varphi_n^a$ 
be a solution of \eqref{eq_eigenfunction_n} satisfying \eqref{eq:23}. 
Let $R_0$ be as above, $\sigma_0$ and $c_0>0$
 be as in Lemma \ref{lemma:estimate-Mja} and let $r_0 \leq\min\{ R_0,\sigma_0\}$.
For $\delta \in (0, \frac{\sqrt{\mu_1}}{2})$, let $\Upsilon_\delta > 1$ be such that \eqref{eq:poincare2} holds.
Then, there exists $c_{r_0, \delta} > 0$ 
such that for all $R > \max\{\Upsilon_{\delta},c_0\}$, $|a| < r_0 / R$, 
$R |a| \leq r < r_0$ and $n\in\{1,\dots,n_0\}$,
\[
e^{\frac{\Lambda r^2}{1-\Lambda r_0^2}} 
\left( \mathcal{N}(\varphi_n^a, r, \lambda_n^a, A_a) + 1 \right)
\leq e^{\frac{\Lambda r_0^2}{1 - \Lambda r_0^2}}
\left( \mathcal{N}(\varphi_n^a, r_0, \lambda_n^a, A_a) + 1 \right) 
+ \frac{c_{r_0,\delta}}{R-c_0}.
\]
\end{Lemma}
\begin{proof}
By direct calculations, using the expressions for the derivatives of
the functions $H(\varphi_n^a, r)$ and 
$E(\varphi_n^a, r, \lambda_n^a, A_a)$ written in \eqref{eq:74} and \eqref{eq:77}
and the Cauchy-Schwarz inequality, we obtain that
\begin{equation}\label{eq:Nderivative}
\frac{d}{dr} \mathcal{N}(\varphi_n^a, r, \lambda_n^a, A_a)  \geq 
- \frac{2 |M_n^a|}{r H(\varphi_n^a, r)} - \frac{2 \lambda_n^a}
{r H(\varphi_n^a, r) } \int_{D_r} |\varphi_n^a|^2 \,dx.
\end{equation}
By Lemmas~\ref{lemma:stima_H_sotto} 
and \ref{lemma:estimate-Mja} the first term can be estimated as
\begin{align}\label{eq:Nderivative2}
- \frac{2 |M_n^a|}{r H(\varphi_n^a,r)} 
&= - \frac{2 |M_n^a|}{r H(\varphi_n^a, R|a|)} 
\frac{H(\varphi_n^a, R|a|)}{H(\varphi_n^a, r)} \\
\notag&\geq - \frac{4\alpha(1-\alpha)}{R-c_0}  e^{\Lambda (2 + \sqrt{\mu_1}) r_0^2} 
(R|a|)^{2(\sqrt{\mu_1} - \delta)} r^{- 2 (\sqrt{\mu_1}-\delta) -1},
\end{align}
for all $R > \max\{\Upsilon_{\delta},c_0\}$, $|a| < r_0 / R$, 
$R |a| \leq r < r_0$ and $n\in\{1,\dots,n_0\}$.

For the second term, the Poincaré inequality \eqref{eq:poincare} leads to
\[
\frac{1 - \Lambda r_0^2}{r^2} \int_{D_r} |\varphi_n^a|^2 \,dx 
\leq E(\varphi_n^a, r, \lambda_n^a, A_a) + H(\varphi_n^a, r),
\]
for $r<r_0$, which implies
\begin{equation}\label{eq:Nderivative3}
- \frac{2 r \lambda_n^a}
{r^2 H(\varphi_n^a, r)} \int_{D_r} |\varphi_n^a|^2 \,dx
\geq - \frac{2 \Lambda r}{1 - \Lambda r_0^2} 
(\mathcal{N}(\varphi_n^a, r, \lambda_n^a, A_a) + 1),
\end{equation}
for $r<r_0$.
Using \eqref{eq:Nderivative2} and \eqref{eq:Nderivative3} we can
estimate the right hand side of \eqref{eq:Nderivative} thus obtaining
\begin{multline*}
\frac{d}{dr} \left( e^{\frac{\Lambda r^2}{1 - \Lambda r_0^2}} 
(\mathcal{N}(\varphi_n^a, r, \lambda_n^a, A_a) + 1 ) \right) \\
\geq - 
\frac{4\alpha(1-\alpha)}{R-c_0}  e^{\frac{\Lambda r_0^2}{1 - \Lambda r_0^2}} 
e^{\Lambda (2 + \sqrt{\mu_1}) r_0^2} (R|a|)^{2 (\sqrt{\mu_1} - \delta)} 
r^{- 2 (\sqrt{\mu_1} - \delta) -1},
\end{multline*}
for all $R |a| \leq r < r_0$ with $R >
\max\{\Upsilon_{\delta},c_0\}$.
Integrating between $r$ and $r_0$ and using the fact that $R|a| \leq r < r_0$, 
we obtain the statement with
\[
c_{r_0, \delta} = \frac{2\alpha(1-\alpha)}{\sqrt{\mu_1} - \delta} e^{\Lambda (2 +  \sqrt{\mu_1}) r_0^2 + \frac{\Lambda r_0^2}{1 - \Lambda r_0^2}} . \qedhere
\]
\end{proof}

\begin{Lemma}  \label{lemma:estimate-N}
Let $\varphi_a$ be a solution 
of \eqref{eq:equation_a}--\eqref{eq:normalization} and let $k$ be as in \eqref{eq:37}. 
For every $\delta \in (0, \frac{\sqrt{\mu_1}}{2})$, 
there exist $r_\delta \in( 0,R_0)$ and $K_\delta >
  \Upsilon_\delta$
 such that, 
if $R > K_\delta$, $|a| < r_\delta/R$ and $R|a| \leq r < r_\delta$, then 
\[
\mathcal{N}(\varphi_a, r, \lambda_a, A_a) \leq |\alpha - k| + \delta.
\]
\end{Lemma}

\begin{proof}
 From \eqref{eq:convergence-varphi-a-1}--\eqref{eq:convergence-varphi-a-2} it follows that, for every $r<R_0$,
\[
\lim_{a\to 0} 
\mathcal{N}(\varphi_a, r, \lambda_a, A_a)= \mathcal{N}(\varphi_0, r,
\lambda_0, A_0) .
\]
Moreover, from \cite[Theorem 1.3]{FelliFerreroTerracini2011} we know
that, under assumption \eqref{eq:131}, 
\[
\lim_{r\to 0^+} \mathcal{N}(\varphi_0, r,
\lambda_0, A_0)=|\alpha-k|.
\] 
Then, the  proof is a direct consequence of Lemma
    \ref{l:3.3}, see
\cite[Lemma 7.2]{NorisNysTerracini2015}, 
\cite[Lemma 5.7]{AbatangeloFelli2015-1}, \cite[Lemma
5.7]{AbatangeloFelliNorisNys2016} for details.
\end{proof}

\subsection{Local energy estimates}

\begin{Corollary}\label{cor:Ha}
For $\delta \in (0, \frac12\sqrt{\mu_1})$ let $r_\delta,K_\delta$ be
as in Lemma \ref{lemma:estimate-N} and $\alpha_{r_\delta}$ be as in \eqref{eq:lower-bound-H}. Then 
there exists $C_\delta > 0$ such that
\begin{align}
\label{eq:20}& H(\varphi_a, R |a|) \leq H(\varphi_a, K_\delta |a|)\bigg(\frac{R}{K_\delta}\bigg)^{2(|\alpha-k|+\delta)}
\quad \text{ for all } R>K_\delta\text{ and }|a| < \frac{r_\delta}{R}, \\
\label{eq:21}& H(\varphi_a, K_\delta |a|) \geq C_\delta |a|^{2(|\alpha - k| + \delta)} 
\quad \text{ for all } |a| < \min \left\{ \frac{r_\delta}{K_\delta}, \alpha_{r_\delta} \right\}, \\
\label{eq:24}& H(\varphi_a, K_\delta |a|) =O\left(
  |a|^{2(\sqrt{\mu_1}-\delta)}\right)\quad\text{as }a\to0.
\end{align}
\end{Corollary}
\begin{proof}
  From \eqref{eq:74}, the definition of $\mathcal N$, and Lemma
  \ref{lemma:estimate-N} we have that 
\begin{align*}
\frac1{ H(\varphi_a,r)}\frac{d}{dr} H(\varphi_a,r)&=\frac2r
\mathcal{N}(\varphi_a, r, \lambda_a, A_a)\\
&\leq \frac2r\left( |\alpha - k| +
\delta\right)\quad\text{for all }K_\delta|a|\leq r<r_\delta\text{ with
}|a|<\frac{r_\delta}{K_\delta}
\end{align*}
so that estimate \eqref{eq:20} follows by integration over
$[K_\delta|a|,R|a|]$ and estimate \eqref{eq:21} from integration over
$[K_\delta|a|,r_\delta]$ and \eqref{eq:lower-bound-H}.
Finally \eqref{eq:24} is a direct consequence of Lemma \ref{lemma:stima_H_sotto}.
\end{proof}

\begin{Lemma} \label{lemma:estimate-varphi-a}
For $n\in\{1,\dots,n_0\}$ and $a \in \Omega$, 
let  $\varphi_n^a$ be a 
solution to \eqref{eq_eigenfunction_n} satisfying \eqref{eq:23}. Let $R_0 > 0$ be as 
in \eqref{eq:70}. For every $\delta \in (0, \sqrt{\mu_1}/2)$, there exist 
$\tilde{K}_\delta > 1$ and $\tilde{C}_\delta > 0$ such that, for 
all $R > \tilde{K}_\delta$, $a \in \Omega$ with $R |a| < R_0$, 
and $n\in\{1,\dots,n_0\}$,
\begin{align} \label{eq:estimate-varphi-j-a}
& \int_{D_{R|a|}} | (i \nabla + A_a) \varphi_n^a|^2 \, dx
\leq \tilde{C}_\delta (R|a|)^{2(\sqrt{\mu_1} - \delta)} \\
\label{eq:estimate-varphi-j-a2}& \int_{\partial D_{R|a|}} |\varphi_n^a|^2 \,ds
\leq \tilde{C}_\delta (R|a|)^{2 (\sqrt{\mu_1} - \delta) + 1} \\
\label{eq:estimate-varphi-j-a3}& \int_{D_{R|a|}} |\varphi_n^a|^2 \,dx 
\leq \tilde{C}_\delta (R|a|)^{2 (\sqrt{\mu_1} - \delta) + 2}.
\end{align}
\end{Lemma}

\begin{proof}
 Estimate \eqref{eq:estimate-varphi-j-a2} follows from Lemma
\ref{lemma:stima_H_sotto}. 
From Lemma \ref{l:3.3} it follows that the frequency $\mathcal N$ is bounded in
$r=R|a|$ provided $R$ is sufficiently large; hence $E(\varphi_n^a,
R|a|, \lambda_n^a, A_a)$ is uniformly estimated by $H(\varphi_n^a,
R|a|)$, so that \eqref{eq:estimate-varphi-j-a2} and
\eqref{eq:poincare}--\eqref{eq:poincare2} yield
\eqref{eq:estimate-varphi-j-a}. Estimate
\eqref{eq:estimate-varphi-j-a3} can be proved combining
\eqref{eq:estimate-varphi-j-a}, \eqref{eq:estimate-varphi-j-a2} with
the Poincaré inequality \eqref{eq:poincare}. We refer to \cite[Lemma
5.8]{AbatangeloFelli2015-1} for more details in a related problem.
\end{proof}

\begin{Lemma}\label{lem:phi_tild_bdd}
For $a \in \Omega$ let $\varphi_a \in H^{1,a}_0(\Omega, \C)$ be a solution 
of \eqref{eq:equation_a}--\eqref{eq:normalization}. For some fixed
$\delta \in (0, \sqrt{\mu_1}/2)$, let $K_\delta > \Upsilon_\delta$ 
be as in Lemma~\ref{lemma:estimate-N}. 
Then, for every $R>K_\delta$,
\begin{align} 
\label{eq:estimate-varphi-a1}& \int_{D_{R|a|}} | (i \nabla + A_a) \varphi_a|^2 \,dx 
= O(H(\varphi_a, K_\delta |a|)), \quad \text{ as } |a| \to 0^+, \\
\label{eq:estimate-varphi-a2}& \int_{\partial D_{R|a|}} |\varphi_a|^2 \, ds 
= O( |a| H(\varphi_a, K_\delta |a|)), \quad \text{ as } |a| \to 0^+, \\
\label{eq:estimate-varphi-a3}& \int_{D_{R|a|}} |\varphi_a|^2 \,dx 
= O(|a|^2 H(\varphi_a, K_\delta |a|)), \quad \text{ as } |a| \to 0^+.
\end{align}
\end{Lemma}
\begin{proof}
The proof follows from the boundedness of the
  frequency $\mathcal{N}(\varphi_a, R|a|, \lambda_a, A_a)$ established
  in Lemma \ref{lemma:estimate-N} and by its scaling properties. 
 For 
$\delta \in (0, \sqrt{\mu_1}/2)$ fixed, let $K_\delta >
\Upsilon_\delta$ and $r_\delta$  
be as in Lemma~\ref{lemma:estimate-N}, so that
Lemma~\ref{lemma:estimate-N} yields that 
\begin{align*}
N(\varphi_a, R|a|,\lambda_a,A_a)&=
\dfrac{
    \int_{D_{R|a|}} \abs{(i\nabla + A_a)\varphi_a}^2dx -
      \lambda_a \int_{D_{R|a|}}
                                  \abs{\varphi_a}^2dx}{H(\varphi_a,R|a|)}\\
&\notag\leq|\alpha - k| + \delta,\quad \text{for all $R>
  K_\delta$ and $|a|<\frac{r_\delta}R$}. 
\end{align*}
Then, by~\eqref{eq:poincare} and \eqref{eq:poincare2} it follows that 
\begin{multline*}
 \left(1-\Lambda r_\delta^2\bigg(1+\frac2{\sqrt{\mu_1}}\bigg)\right)\int_{D_{R|a|}}|(i\nabla + A_{a})
      \varphi_a|^2dx\\\leq
  \int_{D_{R|a|}} \abs{(i\nabla + A_a)\varphi_a}^2dx -
      \lambda_a \int_{D_{R|a|}}
                                  \abs{\varphi_a}^2dx\leq
                                  H(\varphi_a,R|a|)
(|\alpha - k| + \delta).
\end{multline*}
Then \eqref{eq:estimate-varphi-a1} follows from
\eqref{eq:20}. Estimates \eqref{eq:estimate-varphi-a2} and
\eqref{eq:estimate-varphi-a3} follow from
\eqref{eq:estimate-varphi-a1} and  the Poincaré type inequalities
\eqref{eq:poincare} and \eqref{eq:poincare2}.
\end{proof}

  \begin{remark}\label{rem:3.7}
 Let us consider the blow-up family
\begin{equation} \label{eq:tilde-varphi-a}
\tilde{\varphi}_a (x) 
:= \frac{\varphi_a (|a| x)}{\sqrt{H(\varphi_a, K_\delta |a|)}},
\end{equation}
with $K_\delta > \Upsilon_\delta$  being as in Lemma~\ref{lemma:estimate-N} for some fixed
$\delta \in (0, \sqrt{\mu_1}/2)$.
  From Lemma \ref{lem:phi_tild_bdd} it follows that, for every
    $p\in{\mathbb S}^1$ fixed,  $r_\delta>0$ as in
    Lemma  \ref{lemma:estimate-N}, and $R > K_\delta$, the  blow-up family
    $\{ \tilde{\varphi}_a \, : \, a = |a| p, \, R |a| < r_\delta \}$
    is bounded in $H^{1,p}(D_R, \C)$.
  \end{remark}

\section{Estimate on \texorpdfstring{ $ \lambda_0- \lambda_a $ }{estimate}}\label{sec:estim-texorpdfstr-}

Respectively, the Courant-Fisher characterization for $\lambda_a$ 
and $\lambda_0$ gives
\begin{equation}\label{eq:18}
\lambda_a = \min \left\{ \max_{ u \in F \setminus \{0\} } 
\tfrac{\int_{\Omega} | (i\nabla + A_a) u|^2 \,dx }{\int_{\Omega} |u|^2 \,dx} 
 :  F \text{ is a linear subspace of } H^{1,a}_0(\Omega, \C), \, \text{dim} F = n_0 \right\}
\end{equation}
and
\begin{equation}\label{eq:19}
\lambda_0 = \min \left\{ \max_{u \in F \setminus \{0\} } 
\tfrac{\int_{\Omega} |(i\nabla + A_0) u|^2 \,dx}{\int_{\Omega} |u|^2 \,dx} \, : \, 
F \text{ is a linear subspace of }  H^{1,0}_0(\Omega,\C), \, \text{dim} F = n_0 \right\}.
\end{equation}
Before proceeding, we find useful to recall the following technical
result which is proved in \cite[Lemma 6.1]{AbatangeloFelli2015-1} and concerns the maximum of quadratic forms depending on the pole $a\to0$.

\begin{Lemma}\label{l:tech}
  For every  $a\in\Omega$, let us consider a quadratic form 
\[
Q_a:\C^{n_0}\to \R,\quad 
Q_a(z_1,z_2,\dots,z_{n_0})=\sum_{j,n=1}^{n_0}M_{j,n}(a)z_j
\overline{z_n},
\]
with $M_{j,n}(a)\in\C$ such that
$M_{j,n}(a)=\overline{M_{n,j}(a)}$. Let us assume that there exist
$\gamma\in(0,+\infty)$, $a\mapsto\sigma(a)\in\R$ 
with $\sigma(a)\geq0$ and $\sigma(a)=O(|a|^{2\gamma})$ as $|a|\to 0^+$, and
$a\mapsto\mu(a)\in\R$ 
with $\mu(a)=O(1)$ as $|a|\to 0^+$, such that the coefficients $M_{j,n}(a)$
satisfy the following conditions:
\begin{enumerate}[\rm (i)]
\item $M_{n_0,n_0}(a)=\sigma(a)\mu(a)$;
\item for all $j<n_0$, $M_{j,j}(a)\to M_{j}$ as $|a|\to 0^+$ for some
$M_j\in\R$,  $M_j<0$;
\item for all $j<n_0$,
  $M_{j,n_0}(a)=\overline{M_{n_0,j}(a)}=O\left(|a|^\gamma\sqrt{\sigma(a)}\right)$
  as $|a|\to 0^+$;
\item for all $j,n<n_0$ with $j\neq n$,   $M_{j,n}(a)=O(|a|^{2\gamma})$
 as $|a|\to 0^+$;
\item there exists $M\in\N$ such that
  $|a|^{(2+M)\gamma}=o(\sigma(a))$ as  $|a|\to 0^+$.  
\end{enumerate}
Then 
\[
\max_{\substack{z\in \C^{n_0}\\
  \|z\|=1}}Q_a(z)=\sigma(a)\big(\mu(a)+o(1)\big)\quad\text{as
  }|a|\to0^+,
\]
where $\|z\|=\|(z_1,z_2,\dots,z_{n_0})\|=\big(\sum_{j=1}^{n_0}|z_j|^2\big)^{1/2}$.
\end{Lemma}

\subsection{Construction of the test functions using \texorpdfstring{$ \varphi_n^0 $}{v0} }

Let $R_0$ be as in \eqref{eq:70}. For every
$R>1$, $a\in\Omega$  with
$|a|<R_0/R$ and $1 \leq n \leq n_0$ we define
\begin{equation*}
w_{n,R,a} = 
\begin{cases}
w_{n,R,a}^{int}, &\text{in } D_{R|a|}, \\[3pt]
w_{n,R,a}^{ext}, &\text{in } \Omega \setminus D_{R|a|},
\end{cases}
\end{equation*}
where
\[
w_{n,R,a}^{ext} = e^{i \alpha (\theta_a - \theta_0^a)} \varphi_n^0 
\quad \text{ in } \Omega \setminus D_{R|a|},
\]
and $w_{n,R,a}^{int}$ is the unique solution to the minimization problem
\[
\min \bigg\{ \int_{D_{R|a|}} | (i \nabla + A_a) u|^2 \,dx: \, u \in H^{1,a}(D_{R|a|},\C), \, 
u = e^{i \alpha (\theta_a - \theta_0^a)} \varphi_n^0 
\text{ on } \partial D_{R|a|} \bigg\}.
\]
We notice that $w_{n,R,a}^{ext}$ and $w_{n,R,a}^{int}$ respectively solve
\[
\begin{cases}
(i \nabla + A_a)^2 w_{n,R,a}^{ext} = \lambda_n^0 w_{n,R,a}^{ext},
&\text{in } \Omega \setminus D_{R|a|}, \\[3pt]
w_{n,R,a}^{ext} = e^{i \alpha (\theta_a - \theta_0^a)} \varphi_n^0,
&\text{on } 
\partial \left( \Omega \setminus D_{R|a|} \right),
\end{cases}
\]
and
\[
\begin{cases}
( i \nabla + A_a)^2 w_{n,R,a}^{int} = 0, &\text{in } D_{R|a|}, \\[3pt]
     w_{n,R,a}^{int} = e^{i \alpha (\theta_a - \theta_0^a)} \varphi_n^0, &\text{on } \partial
     D_{R|a|}.
   \end{cases}
\]
As a consequence of Proposition \ref{prop:fft}
we have  that, for every $R > 1$, $a \in \Omega$ such that $R |a| < R_0$, 
and $1 \leq n \leq n_0$,
\begin{align} \label{eq:estimate-1}
&\int_{D_{R|a|}} | (i \nabla + A_0) \varphi_n^0 |^2 \,dx = O(|a|^{2 \sqrt{\mu_1}}),\quad
 \int_{\partial D_{R|a|}} |\varphi_n^0|^2 \, ds = O(|a|^{2
   \sqrt{\mu_1} + 1}),\\
&\notag \text{and} \quad 
 \int_{D_{R|a|}} |\varphi_n^0|^2 \,dx = O(|a|^{2 \sqrt{\mu_1} +
  2})\quad\text{as }|a|\to0^+.
\end{align}
Using the above estimates \eqref{eq:estimate-1} and the Dirichlet
principle (see the proof of \cite[Lemma 6.2]{AbatangeloFelli2015-1}
for details in the case of  half-integer circulation), we obtain that,
for every $R>2$ and $1 \leq n \leq n_0$, 
\begin{align} \label{eq:estimate-2}
& \int_{D_{R|a|}} | (i \nabla + A_a) w_{n,R,a}^{int} |^2 \,dx =
  O(|a|^{2 \sqrt{\mu_1}}), \quad \int_{\partial D_{R|a|}}
  |w_{n,R,a}^{int}|^2 \, ds = O(|a|^{2 \sqrt{\mu_1} + 1}),  \\
& \notag\text{and}\quad \int_{D_{R|a|}} |w_{n,R,a}^{int}|^2 \,dx = O(|a|^{2 \sqrt{\mu_1} + 2}) \quad\text{as }|a|\to0^+.
\end{align}
The above estimates can be made more precise in the
case $n=n_0$ in view of \eqref{eq:131}: for every $R>2$ and
$a\in\Omega$  with $R|a|<R_0$
\begin{align} \label{eq:estWa}
&\int_{D_{R|a|}} | (i \nabla + A_0) \varphi_0 |^2 \,dx = O(|a|^{2 |\alpha-k|}),\quad
 \int_{\partial D_{R|a|}} |\varphi_0|^2 \, ds = O(|a|^{2
   |\alpha-k|+ 1}),\\
&\notag \text{and} \quad 
 \int_{D_{R|a|}} |\varphi_0|^2 \,dx = O(|a|^{2 |\alpha-k|+
  2})\quad\text{as }|a|\to0^+,
\end{align}
and consequently, in view of the Dirichlet principle,
\begin{align} \label{eq:estimate-2-n0}
& \int_{D_{R|a|}} | (i \nabla + A_a) w_{n_0,R,a}^{int} |^2 \,dx =
  O(|a|^{2 |\alpha-k|}), \quad \int_{\partial D_{R|a|}}
  |w_{n_0,R,a}^{int}|^2 \, ds = O(|a|^{2 |\alpha-k|+ 1}),  \\
& \notag\text{and}\quad \int_{D_{R|a|}} |w_{n_0,R,a}^{int}|^2 \,dx = O(|a|^{2 |\alpha-k| + 2}) \quad\text{as }|a|\to0^+,
\end{align}
with $k$ as in \eqref{eq:37}.
Furthermore, defining  
\begin{equation} \label{eq:def-Wa}
W_a(x) := \frac{\varphi_0(|a|x)}{|a|^{|\alpha - k|}}
\end{equation}
for all $R > 2$ and $a \in \Omega$ such that $R |a| < R_0$, 
\eqref{eq:131} implies that
\begin{equation} \label{eq:conv-Wa}
W_a \to \beta \psi_k \quad\text{in } H^{1,0}(D_R,\C)
\quad\text{as } |a| \to 0,
\end{equation}
where $\psi_k$ is defined in \eqref{eq:psi_j_definition}.

\subsection{Estimate of the Rayleigh quotient for \texorpdfstring{$\lambda_a $}{lambda-a} }

\begin{Lemma} \label{lemma:lower-estimate}
There exists $\mathfrak
  c\in\R$ such that 
\[
\lambda_0 - \lambda_a\geq {\mathfrak c}|a|^{2 |\alpha - k|}
\quad\text{for all }a\in\Omega,
\]
where $k$ is as in \eqref{eq:37}.
\end{Lemma}
\begin{proof}
The proof follows the lines of \cite[Lemma 6.7]{AbatangeloFelli2015-1} 
and \cite[Lemma 7.2]{AbatangeloFelliNorisNys2016}. 
Let us fix $R>2$.
By proceeding with a Gram-Schmidt process 
we define
\[
\tilde{w}_{n,a} = \frac{\Hat{w}_{n,a}}{\| \Hat{w}_{n,a} \|_{L^2(\Omega,\C)} }, 
\quad 1 \leq n \leq n_0,
\]
where 
\begin{align*}
& \Hat{w}_{n_0,a} = w_{n_0,R,a} \\
& \Hat{w}_{n,a} = w_{n,R,a} - \sum_{\ell=n+1}^{n_0} c_{\ell,n}^{a} \Hat{w}_{\ell,a}, 
\quad 1 \leq n \leq n_0 - 1,
\end{align*}
and 
\[
c_{\ell,n}^{a} = \frac{\int_{\Omega} w_{n,R,a} \overline{ \Hat{w}_{\ell,a} } \,dx }
{\| \Hat{w}_{\ell,a} \|_{L^2(\Omega,\C)}^2 }, \quad 1\leq n\leq n_0-1,\ n+1\leq\ell\leq n_0.
\]
From \eqref{eq:estimate-1}, \eqref{eq:estimate-2} 
and an induction argument it follows that, for all $\ell,n$ such that $1\leq n\leq n_0-1$ and $n+1\leq\ell\leq n_0$,
\begin{equation} \label{eq:estimate-w-j}
\| \Hat{w}_{n,a} \|^2_{L^2(\Omega,\C)} = 1 + O( |a|^{2 \sqrt{\mu_1} + 2})
\quad \text{ and } \quad 
c_{\ell,n}^{a} = O(|a|^{2 \sqrt{\mu_1} + 2})
\end{equation}
as $|a| \to 0$. Morever, from \eqref{eq:estWa} and \eqref{eq:estimate-2-n0} we have that
\begin{equation} \label{eq:estimate-w-n0}
\| \Hat{w}_{n_0,a} \|_{L^2(\Omega,\C)}^2 = \| w_{n_0,R,a} \|_{L^2(\Omega,\C)}^2 
= 1 + O(|a|^{2|\alpha - k| + 2}) \quad \text{ as } |a| \to 0,
\end{equation}
and
\begin{equation} \label{eq:estimate-c-n0j}
c_{n_0, n}^{a} = O(|a|^{|\alpha - k| + \sqrt{\mu_1} + 2}) 
\quad \text{ as } |a| \to 0, \, \text{ for } 1\leq n \leq n_0-1.
\end{equation}
Since $\text{dim}(\text{span} \{ w_{1,R,a}, \ldots, w_{n_0,R,a} \}) =
n_0$, we have that also $\text{dim}(\text{span} \{ \tilde{w}_{1,a}, \ldots, \tilde{w}_{n_0,a} \}) =
n_0$, and hence from \eqref{eq:18} we deduce that
\begin{equation*}
\lambda_a \leq \max_{ \substack{ (\alpha_1, \ldots, \alpha_{n_0}) \in \C^{n_0} \\ \sum_{n=1}^{n_0} |\alpha_n|^2 = 1 } } 
\int_{\Omega} \bigg| (i\nabla + A_a) \bigg( \sum_{n=1}^{n_0} \alpha_n \tilde w_{n,a} \bigg) \bigg|^2 \,dx,
\end{equation*}
which leads to 
\begin{equation} \label{eq:diff-eig-a-0}
\lambda_a - \lambda_0 
\leq \max_{ \substack{ (\alpha_1, \ldots, \alpha_{n_0}) \in \C^{n_0} \\ \sum_{n=1}^{n_0} |\alpha_n|^2 = 1 } } 
\sum_{n,j=1}^{n_0} \alpha_n \overline{\alpha_j} p_{n,j}^{a},
\end{equation}
where $p_{n,j}^{a} = \int_{\Omega} (i \nabla + A_a) \tilde{w}_{n,a} 
\cdot \overline{(i \nabla + A_a) \tilde{w}_{j,a}} \,dx - \lambda_0 \delta_{nj}$, with $\delta_{nj}=1$ if $n=j$ and $\delta_{nj}=0$ otherwise.
Using the estimates above we can now estimate $p_{n,j}^{a}$. First, using \eqref{eq:estWa}, 
\eqref{eq:estimate-2-n0}, and \eqref{eq:estimate-w-n0}
\begin{align*}
p_{n_0 ,n_0}^{a} 
& = \frac{\lambda_0}{\int_{\Omega} |w_{n_0,R,a}|^2 \,dx} \left( 1 - \int_{\Omega} |w_{n_0,R,a}|^2\,dx \right) \\
& + \frac{1}{\int_{\Omega} |w_{n_0,R,a}|^2\,dx} 
\left( \int_{D_{R|a|}} |(i \nabla + A_{a}) w_{n_0,R,a}^{int}|^2 \, dx 
- \int_{D_{R|a|}} | (i \nabla + A_0) \varphi_0|^2 \,dx \right) \\
&= O(|a|^{2 |\alpha - k|+2})+O(|a|^{2 |\alpha - k|})=|a|^{2 |\alpha -
  k|} O(1),\quad\text{as }|a|\to0^+.
\end{align*}
Next \eqref{eq:estimate-1}, \eqref{eq:estimate-2} and \eqref{eq:estimate-w-j} provide for $n < n_0$
\begin{align*}
p_{n,n}^{a} 
& = - \lambda_0 + \frac{1}{\| \Hat{w}_{n,a} \|_{L^2(\Omega,\C)}^2} 
\left( \lambda_n^0 + \int_{D_{R|a|}} |(i\nabla + A_a) w_{n,R,a}^{int}|^2 \, dx 
- \int_{D_{R|a|}} |(i \nabla + A_0) \varphi_n^0|^2 \,dx \right) \\
& + \frac{1}{\| \Hat{w}_{n,a} \|_{L^2(\Omega,\C)}^2}
\int_{\Omega} \bigg| (i \nabla + A_a) \bigg( \sum_{\ell=n+1}^{n_0} c_{\ell,n}^{a} \Hat{w}_{\ell,a} \bigg) \bigg|^2 \,dx \\
& - \frac{2}{\| \Hat{w}_{n,a} \|_{L^2(\Omega,\C)}^2} 
\mathfrak{Re} \sum_{\ell=n+1}^{n_0} \left\{ \overline{c_{\ell,n}^{a}} 
\int_{\Omega} (i \nabla + A_a) w_{n,R,a} \cdot \overline{(i \nabla + A_a) \Hat{w}_{\ell,a}} \, dx \right\} \\
& = (\lambda_n^0 - \lambda_0) + o(1),
\end{align*}
as $|a| \to 0$. Using \eqref{eq:estimate-1},
\eqref{eq:estimate-2}, \eqref{eq:estWa}, \eqref{eq:estimate-2-n0}, \eqref{eq:estimate-w-j}   
and \eqref{eq:estimate-c-n0j}, we have that, for all $n < n_0$,
\[
p_{n,n_0}^{a} = \overline{p_{n_0,n}^{a}} = O(|a|^{\sqrt{\mu_1} + |\alpha - k|}), 
\quad \text{ as } |a| \to 0, 
\]
while the same estimates imply that, for all $n \neq \ell < n_0$,
\[
p_{n,\ell}^{a} = \overline{p_{\ell,n}^{a}} = O(|a|^{2\sqrt{\mu_1}}), 
\quad \text{ as } |a| \to 0.
\]
Therefore, the quadratic form in \eqref{eq:diff-eig-a-0} satisfies the hypothesis of 
Lemma \ref{l:tech} with $\sigma(a) = |a|^{2 |\alpha - k|}$, 
$\gamma= \sqrt{\mu_1}$, 
$M_j=\lambda_j^0-\lambda_0<0$ for $j<n_0$ and
$M\in\N$ such that $(2+M)\sqrt{\mu_1}>2|\alpha-k|$, so that
\[
\max_{ \substack{ (\alpha_1, \ldots, \alpha_{n_0}) \in \C^{n_0} \\ \sum_{n=1}^{n_0} |\alpha_n|^2 = 1 } } 
\sum_{n,j=1}^{n_0} \alpha_n \overline{\alpha_j} p_{n,j}^{a} 
=|a|^{2 |\alpha - k|}O(1), 
\quad \text{ as } |a| \to 0.
\]
The proof is thereby complete.
\end{proof}
We notice that Lemma \ref{lemma:lower-estimate} does not give any
information about the sign of the constant $\mathfrak c$.

\subsection{Construction of the test functions using \texorpdfstring{$ \varphi_n^a $}{va} }\label{subsec:vjRa}

Let $R_0$ be as in \eqref{eq:70}, $R > 1$ and $|a| <\frac{R_0}R$. For every $1 \leq n \leq n_0$ we define
\begin{equation*}
v_{n,R,a} =
\begin{cases}
v_{n,R,a}^{int}, & \text{in } D_{R|a|}, \\[3pt]
v_{n,R,a}^{ext}, &\text{in } \Omega \setminus D_{R|a|},
\end{cases}
\end{equation*}
where 
\[
v_{n,R,a}^{ext} = e^{i \alpha ( \theta_0^a-\theta_a)} \varphi_n^a 
\quad \text{ in } \Omega \setminus D_{R|a|},
\]
and $v_{n,R,a}^{int}$ is the unique solution to the minimization problem 
\begin{equation} \label{eq:vint}
\min \left\{ \int_{D_{R|a|}} |(i \nabla + A_0) u|^2 \,dx \, : \, u \in H^{1,0}(D_{R|a|},\C), \,
u = e^{i \alpha ( \theta_0^a-\theta_a)} \varphi_n^a \text{ on } \partial D_{R|a|} \right\}.
\end{equation}
We notice that $v_{n,R,a}^{ext}$ and $v_{n,R,a}^{int}$ respectively solve
\begin{equation*}
\begin{cases}
(i \nabla + A_0)^2 v_{n,R,a}^{ext} = \lambda_n^a v_{n,R,a}^{ext}, &\text{in } \Omega \setminus D_{R|a|}, \\[3pt]
     v_{n,R,a}^{ext} = e^{- i \alpha (\theta_a - \theta_0^a)}
     \varphi_n^a, &\text{on } \partial \left( \Omega \setminus
       D_{R|a|} \right),
   \end{cases}
 \end{equation*}
and 
\begin{equation}\label{eq:vextinteq}
\begin{cases}
( i \nabla + A_0)^2 v_{n,R,a}^{int} = 0, &\text{in } D_{R|a|}, \\[3pt]
 v_{n,R,a}^{int} = e^{- i \alpha (\theta_a - \theta_0^a)}
     \varphi_n^a, &\text{on } \partial D_{R|a|}.
\end{cases}
\end{equation}
The energy estimates obtained in Lemmas \ref{lemma:estimate-varphi-a}
and \ref{lem:phi_tild_bdd} imply the following estimates for the
functions $v_{n,R,a}^{int}$.

\begin{Lemma}
For $\delta \in (0, \sqrt{\mu_1}/2)$ fixed, let $\tilde{K}_\delta$ be as in Lemma~\ref{lemma:estimate-varphi-a} 
and $R_0$ be as in \eqref{eq:70}. Let $R >
\max\{2,\tilde{K}_\delta\}$ and $1 \leq n \leq n_0$ be fixed.
For every $a\in \Omega$ 
with $|a| < R_0/R$, let $v_{n,R,a}^{int}$ be defined as in \eqref{eq:vint}. Then
\begin{align} \label{eq:estimate-vint-j-1}
& \int_{D_{R|a|}} | (i \nabla + A_0) v_{n,R,a}^{int}|^2 \, dx=O(|a|^{2(\sqrt{\mu_1} - \delta)}) \\
& \notag\int_{D_{R|a|}} |v_{n,R,a}^{int}|^2
  \,dx=O(|a|^{2 (\sqrt{\mu_1} - \delta) + 2})\quad\text{and}
\quad \int_{\partial D_{R|a|}} |v_{n,R,a}^{int}|^2 \,ds
=O(|a|^{2 (\sqrt{\mu_1} - \delta) + 1}) 
\end{align}
as $|a|\to0^+$.
\end{Lemma}

\begin{proof}
The proof follows by
combining the Dirichlet principle, a suitable cutting-off procedure,
 and
Lemma~\ref{lemma:estimate-varphi-a} (see the proof of \cite[Lemma 6.2]{AbatangeloFelli2015-1}
for details in the case of  half-integer circulation).
\end{proof}

\begin{Lemma}\label{lem:ZRa_bdd}
For  $R > \max\{2, K_\delta\}$ fixed, with $K_\delta$ being as in 
Lemma~\ref{lemma:estimate-N}, let $v_{n_0,R,a}^{int}$ be defined 
as in \eqref{eq:vint}. Then
\begin{align}
 \label{eq:estimate-vint-n0-1}
& \int_{D_{R|a|}} |(i \nabla + A_0) v_{n_0,R,a}^{int}|^2\,dx = O(H(\varphi_a, K_\delta |a|)), \\
&\label{eq:estimate-vint-n0-2} 
\int_{D_{R|a|}} |v_{n_0,R,a}^{int}|^2 \,dx = O( |a|^2 H(\varphi_a,
  K_\delta |a|)) , \quad 
\int_{\partial D_{R|a|}}
  |v_{n_0,R,a}^{int}|^2 \,ds  = O( |a| H(\varphi_a, K_\delta |a|)),
\end{align}
as  $|a| \to 0^+$.
\end{Lemma}

\begin{proof}
The proof follows from the estimates of Lemma
  \ref{lem:phi_tild_bdd},  a suitable cutting-off procedure,
and the Dirichlet principle (see \eqref{eq:vint}).
\end{proof}

  \begin{remark}\label{rem:4.5}
For all $R > 2$ and $a\in\Omega$ with $|a| < R_0/R$ we consider the blow-up family
\begin{equation} \label{eq:ZaR}
Z_a^R(x) := \frac{ v_{n_0,R,a}^{int} (|a|x) }{ \sqrt{ H( \varphi_a, K_\delta |a|) }},
\end{equation}
with $K_\delta$ as in Lemma \ref{lemma:estimate-N}  for some fixed
$\delta \in (0, \sqrt{\mu_1}/2)$. From Lemma \ref{lem:ZRa_bdd} it
follows that, for every
    $p\in{\mathbb S}^1$ fixed,  $r_\delta>0$ as in
    Lemma  \ref{lemma:estimate-N}, and $R > \max\{K_\delta,2\}$, the
    family of functions 
$\{ Z_a^R \, : \, a = |a| p \in \Omega, \, |a| < r_\delta/R \}$  is bounded in $H^{1,0}(D_R,\C)$.
\end{remark}

\subsection{Estimate of the Rayleigh quotient for \texorpdfstring{$ \lambda_0 $}{lambda-0}}

An estimate from above for the limit eigenvalue $\lambda_0$ in terms
of the approximating eigenvalue $\lambda_a$ can be obtained by
choosing as test functions in \eqref{eq:19} an orthonormal family
constructed starting from the functions
$\{v_{n,R,a}\}_{n=1,\dots,n_0}$, as done in the following.

\begin{Lemma} \label{lemma:upper-estimate}
For $\delta \in (0, \frac12\sqrt{\mu_1})$ fixed,  let $r_\delta,K_\delta$ be
as in Lemma \ref{lemma:estimate-N} and $\alpha_{r_\delta}$ be as in \eqref{eq:lower-bound-H}. Then 
there exists ${\mathfrak d}_\delta > 0$ such that
\[
\lambda_0 - \lambda_a \leq {\mathfrak d}_\delta H(\varphi_a, K_\delta |a|),
\]
for all $a\in\Omega$ such that $ |a| < \min \big\{
  \frac{r_\delta}{K_\delta}, \alpha_{r_\delta} \big\}$.
\end{Lemma}

\begin{proof}
In view of \eqref{eq:21} it is enough to prove that $\lambda_0 -
\lambda_a\leq O (H(\varphi_a, K_\delta |a|))$ as $|a|\to0^+$.

Let us fix $R> \max\{2, K_\delta, \tilde{K}_\delta\}$, with $\tilde{K}_\delta$ as in Lemma~\ref{lemma:estimate-varphi-a}. 
As in the proof of Lemma~\ref{lemma:lower-estimate}, we use a Gram-Schmidt 
process, that is we define
\[
\tilde{v}_{n,a} = \frac{\Hat{v}_{n,a}}{\| \Hat{v}_{n,a} \|_{L^2(\Omega,\C)}}
\quad 1 \leq n \leq n_0,
\]
where
\begin{align*}
& \Hat{v}_{n_0,a} = v_{n_0,R,a}, \\
& \Hat{v}_{n,a} = v_{n,R,a} - \sum_{\ell=n+1}^{n_0} d_{\ell,n}^{a} \Hat{v}_{\ell,a},
\quad 1 \leq n \leq n_0 - 1,
\end{align*}
and 
\[
d_{\ell,n}^{a} = \frac{ \int_{\Omega} v_{n,R,a} \overline{\Hat{v}_{\ell,a}} \,dx }
{\| \Hat{v}_{\ell,a} \|_{L^2(\Omega,\C)}^2 }, \quad 1\leq n\leq n_0-1,\ n+1\leq\ell\leq n_0.
\]
From \eqref{eq:estimate-varphi-j-a3}, \eqref{eq:estimate-vint-j-1} and an induction argument 
it follows that, for every $1\leq n\leq n_0-1$ and $n+1\leq\ell\leq n_0$,
\begin{equation} \label{eq:estimate-v-j}
\| \Hat{v}_{n,a} \|_{L^2(\Omega,\C)}^2 = 1 + O(|a|^{2( \sqrt{\mu_1}- \delta) + 2 })
\quad \text{ and } \quad d_{\ell,n}^{a} = O(|a|^{2( \sqrt{\mu_1}- \delta) + 2}), 
\end{equation}
as $|a| \to 0$. Moreover, from \eqref{eq:estimate-varphi-a3} and \eqref{eq:estimate-vint-n0-2}, 
we have that
\begin{equation} \label{eq:estimate-v-n0}
\| \Hat{v}_{n_0,a} \|_{L^2(\Omega,\C)}^2 = 1 + O(|a|^2 H(\varphi_a, K_\delta |a|)) 
\quad \text{ as } |a| \to 0,
\end{equation}
and, for $1\leq n\leq n_0-1$,
\begin{equation} \label{eq:estimate-d-n0j}
d_{n_0, n}^{a} = O(|a|^{\sqrt{\mu_1} - \delta + 2} \sqrt{H(\varphi_a, K_\delta |a|)}) 
\quad \text{ as } |a| \to 0.
\end{equation}
Since $\text{dim}(\text{span} \{ v_{1,R,a}, \ldots, v_{n_0,R,a} \}) =
n_0$, we have that also $\text{dim}(\text{span} \{ \tilde{v}_{1,a}, \ldots, \tilde{v}_{n_0,a} \}) =
n_0$, and hence from \eqref{eq:19} we deduce that
\begin{equation*}
\lambda_0 \leq \max_{ \substack{ (\alpha_1, \ldots, \alpha_{n_0}) \in \C^{n_0} \\ \sum_{n=1}^{n_0} |\alpha_n|^2 = 1 } } 
\int_{\Omega} \bigg| (i\nabla + A_0) \bigg( \sum_{n=1}^{n_0} \alpha_n  \tilde v_{n,a} \bigg) \bigg|^2 \,dx,
\end{equation*}
which leads to 
\begin{equation} \label{eq:diff-eig-0-a}
\lambda_0 - \lambda_a 
\leq \max_{ \substack{ (\alpha_1, \ldots, \alpha_{n_0}) \in \C^{n_0} \\ \sum_{n=1}^{n_0} |\alpha_n|^2 = 1 } } 
\sum_{n,j=1}^{n_0} \alpha_n \overline{\alpha_j} q_{n,j}^{a},
\end{equation}
where $q_{n,j}^{a} = \int_{\Omega} (i \nabla + A_0) \tilde{v}_{n,a} 
\cdot \overline{(i \nabla + A_0) \tilde{v}_{j,a}} \,dx - \lambda_a \delta_{nj}$.
Using the results above we can now estimate $q_{n,j}^{a}$. 
First, using 
\eqref{eq:estimate-vint-n0-1}, \eqref{eq:estimate-varphi-a1}, and \eqref{eq:estimate-v-n0}
\begin{align*}
q_{n_0 n_0}^{a} 
& = \frac{\lambda_a}{\int_{\Omega} |v_{n_0,R,a}|^2 \,dx} \left( 1 - \int_{\Omega} |v_{n_0,R,a}|^2\,dx \right) \\
& + \frac{1}{\int_{\Omega} |v_{n_0,R,a}|^2\,dx} 
\left( \int_{D_{R|a|}} |(i \nabla + A_0) v_{n_0,R,a}^{int}|^2 \, dx 
- \int_{D_{R|a|}} | (i \nabla + A_{a}) \varphi_a|^2 \,dx \right) \\
& = H(\varphi_a, K_\delta |a|)O(1),
\end{align*}
as $|a| \to 0^+$.
Next 
\eqref{eq:estimate-vint-j-1}, \eqref{eq:estimate-varphi-j-a},
\eqref{eq:estimate-v-j}, and the fact that $\lambda_n^a\to\lambda_n^0$ as $|a|\to0$, provide, for $n < n_0$,
\begin{align*}
q_{n,n}^{a} 
& = - \lambda_a + \frac{1}{\| \Hat{v}_{n,a} \|_{L^2(\Omega,\C)}^2} 
\left( \lambda_n^a + \int_{D_{R|a|}} |(i\nabla + A_0) v_{n,R,a}^{int}|^2 \, dx 
- \int_{D_{R|a|}} |(i \nabla + A_a) \varphi_n^a|^2 \,dx \right) \\
& + \frac{1}{\| \Hat{v}_{n,a} \|_{L^2(\Omega,\C)}^2}
\int_{\Omega} \bigg| (i \nabla + A_0) \bigg( \sum_{\ell=n+1}^{n_0} d_{\ell,n}^{a} \Hat{v}_{\ell,a} \bigg) \bigg|^2 \,dx \\
& - \frac{2}{\| \Hat{v}_{n,a} \|_{L^2(\Omega,\C)}^2} 
\mathfrak{Re} \sum_{\ell=n+1}^{n_0} \left\{ \overline{d_{\ell,n}^{a}} 
\int_{\Omega} (i \nabla + A_0) v_{n,R,a} \cdot \overline{(i \nabla + A_0) \Hat{v}_{\ell,a}} \, dx \right\}\\
& = \lambda_n^0 - \lambda_0 + o(1),
\end{align*}
as $|a| \to 0$. Now, using \eqref{eq:estimate-varphi-j-a}, \eqref{eq:estimate-varphi-a1}, 
\eqref{eq:estimate-vint-j-1}, \eqref{eq:estimate-vint-n0-1},
\eqref{eq:estimate-v-j}, \eqref{eq:estimate-v-n0},   
and \eqref{eq:estimate-d-n0j}, we prove that, for all $n < n_0$,
\[
q_{n,n_0}^{a} = \overline{q_{n_0,n}^{a}} = O\left(|a|^{\sqrt{\mu_1} - \delta} 
\sqrt{H(\varphi_a, K_\delta |a|)}\right), 
\quad \text{ as } |a| \to 0^+, 
\]
while the same estimates imply that, for all $n \neq \ell < n_0$,
\[
q_{n,\ell}^{a} = \overline{q_{\ell,n}^{a}} = O(|a|^{2(\sqrt{\mu_1} - \delta)}), 
\quad \text{ as } |a| \to 0^+.
\]
Therefore, the quadratic form in \eqref{eq:diff-eig-0-a} satisfies the hypothesis of 
Lemma \ref{l:tech} with $\gamma = \sqrt{\mu_1} - \delta$,  
$\sigma(a) = H(\varphi_a, K_\delta |a|)=O(|a|^{2\gamma})$ (by \eqref{eq:24}), 
$M_j=\lambda_j^0-\lambda_0<0$ and $M$ any natural number such that $M > 2 (|\alpha -k| - \sqrt{\mu_1} + 2\delta)/(\sqrt{\mu_1}-\delta)$ 
by Corollary \ref{cor:Ha}.
Therefore the right hand side in \eqref{eq:diff-eig-0-a} satisfies
\[
\max_{ \substack{ (\alpha_1, \ldots, \alpha_{n_0}) \in \C^{n_0} \\ \sum_{n=1}^{n_0} |\alpha_n|^2 = 1 } } 
\sum_{n,j=1}^{n_0} \alpha_n \overline{\alpha_j} q_{n,j}^{a} 
= H(\varphi_a, K_\delta |a|) O(1),
\]
as $|a| \to 0^+$. Then the conclusion follows from \eqref{eq:diff-eig-0-a}.
\end{proof}

\subsection{Energy estimates} \label{subsec:energy}

\begin{Corollary}\label{cor:l0-la_preliminary}
For $\delta \in (0, \frac12\sqrt{\mu_1})$ fixed,  let $K_\delta$ be
as in Lemma \ref{lemma:estimate-N}. Then
\begin{enumerate}[\rm (i)]
\item $|\lambda_0-\lambda_a|=O(1)\max\{H(\varphi_a, K_\delta |a|) ,|a|^{2|\alpha-k|}\}$ as
$a\to0$;
\item $|\lambda_0-\lambda_a|=O\Big((H(\varphi_a, K_\delta |a|))^{\frac{|\alpha-k|}{|\alpha-k|+\delta}}\Big)$ as $a \to0$.
\end{enumerate}
\end{Corollary}
\begin{proof}
Estimate (i) is a direct consequence of Lemmas
\ref{lemma:lower-estimate} and \ref{lemma:upper-estimate}.
Corollary \ref{cor:Ha} implies that 
$|a|^{2|\alpha-k|}=O\big((H(\varphi_a, K_\delta |a|))^{\frac{|\alpha-k|}{|\alpha-k|+\delta}}\big)$ 
as $a\to0$, so that (ii) follows from (i).
\end{proof}

\section{Blow-up analysis}\label{sec:blow-up-analysis}
In order to obtain a more precise estimate of the order of vanishing
of the eigenvalue variation $|\lambda_0-\lambda_a|$ than Corollary
\ref{cor:l0-la_preliminary}, we have now to compare the order of
$H(\varphi_a, K_\delta |a|)$ with $|a|^{2|\alpha-k|}$. We observe that
the estimates obtained so far (in particular Corollary \ref{cor:Ha})
are not enough to decide what is the dominant term among
$H(\varphi_a, K_\delta |a|)$ and $|a|^{2|\alpha-k|}$. To this aim, our
next step is a blow-up analysis for scaled eigenfunctions
\eqref{eq:tilde-varphi-a} along a fixed direction $p\in{\mathbb S}^1$. In order to identify the limit profile of
the blow-up family \eqref{eq:tilde-varphi-a}, the following energy estimate of the difference 
between approximating and limit scaled eigenfunctions plays a
crucial role.

Let ${\mathcal
  D}^{1,2}_{ 0}(\R^2,\C)$ be the completion of $C^\infty_{\rm
  c}(\R^2\setminus\{0\},\C)$ with respect to the magnetic Dirichlet
norm 
\begin{equation*}
  \|u\|_{{\mathcal D}^{1,2}_{0}(\R^2,\C)}:=
\bigg(\int_{\R^2}\big|(i\nabla +A_{0})u(x)\big|^2\,dx\bigg)^{\!\!1/2}.
\end{equation*}

\begin{Theorem}[Energy estimates for eigenfunction variation]\label{thm:energy_estimates}
Let $p\in{\mathbb S}^1$ be fixed. For some fixed 
$\delta \in (0, \sqrt{\mu_1}/2)$, let $K_\delta > \Upsilon_\delta$ 
be as in Lemma~\ref{lemma:estimate-N}. 
For every $R>\max\{2,K_\delta\}$ and $a=|a|p\in\Omega$ such that
$|a|<R_0/R$, 
let $v_{n_0,R,a}$ be as in $\S$ \ref{subsec:vjRa}. 
Then 
\[
\| v_{n_0,R,a} - \varphi_0\|_{H^{1,0}_0(\Omega,\C)} \leq C\Big(h(p,a,R)+g(p,a,R)\Big)\sqrt{H(\varphi_a, K_\delta |a|)}
\]
where $C>0$ is independent of $a,R,p$, 
\begin{equation*}
h(p,a,R)
=\sup_{\substack{\varphi\in \mathcal D^{1,2}_0(\R^2,\C)\\
      \|\varphi\|_{\mathcal D^{1,2}_0(\R^2,\C)}=1}} \bigg|\int_{\partial D_{R}}
\left(e^{i\alpha (\theta_0^p-\theta_p)} (i\nabla+A_p)\tilde\varphi_a-(i\nabla+A_0) Z_a^R\right)\cdot
  \nu\,\overline{\varphi}\,d\sigma\bigg|,
\end{equation*}
and, 
for $p$ and $R$ fixed, 
\[
h(p,a,R)=O(1)\quad\text{and}\quad g(p,a,R)=o(1)
\]
as $|a|\to0^+$.
\end{Theorem}
\begin{proof}
The proof exploits the invertibility 
of the differential of the function $F$ defined below, in the spirit
of \cite[Theorem 8.2]{AbatangeloFelliNorisNys2016} and \cite[Theorem
7.2]{AbatangeloFelli2015-1}.
Let 
\begin{align*}
& F \, : \, \C \times H^{1,0}_{0}(\Omega,\C) 
\longrightarrow \R \times \R \times (H^{1,0}_{0,\R}(\Omega,\C))^\star \\
& (\lambda,\varphi) \longmapsto 
\Big( { \textstyle{ \|\varphi\|_{ H^{1,0}_0(\Omega,\C) }^2 -\lambda_0, \,
\mathfrak{Im} \big(\int_{\Omega} \varphi \overline{\varphi_0} \, dx \big), \, 
(i\nabla +A_0)^2 \varphi - \lambda \varphi } } \Big).
\end{align*}
In the above definition, $(H^{1,0}_{0,\R}(\Omega,\C))^\star$ is the real dual space of 
$H^{1,0}_{0,\R}(\Omega,\C)=H^{1,0}_{0}(\Omega,\C)$, which is here meant as a vector 
space over $\R$ endowed with the norm
\[
\| u \|_{ H^{1,0}_0(\Omega,\C) } = \bigg( \int_{\Omega} \big| (i\nabla +A_0)u \big|^2 \, dx \bigg)^{1/2},
\]
and $(i\nabla +A_0)^2  \varphi-\lambda \varphi\in (H^{1,0}_{0,\R}(\Omega,\C))^\star$ acts as 
\[
\phantom{a}_{ ( H^{1,0}_{0,\R}(\Omega,\C) )^\star } \, 
\Big\langle (i\nabla + A_0)^2 \varphi - \lambda \varphi , u \Big\rangle_{ \, H^{1,0}_{0,\R}(\Omega,\C) } 
= \mathfrak{Re} \left( { \textstyle{ \int_{\Omega} (i\nabla+A_0) \varphi \cdot \overline{ (i\nabla+A_0)u } \, dx 
- \lambda \int_{\Omega} \varphi \overline{u} \, dx } } \right)
\]
for all $\varphi\in H^{1,0}_{0}(\Omega,\C)$.
It is easy to prove that the function $F$ is Fréchet-differentiable at $(\lambda_0,\varphi_0)$, 
with differential
$dF(\lambda_0,\varphi_0)\in \mathcal{L}(\C\times H^{1,0}_{0}(\Omega,\C),\R\times\R\times(H^{1,0}_{0,\R}(\Omega,\C))^*)$ given by
\begin{multline*}
dF(\lambda_0,\varphi_0)(\lambda,\varphi)=\left(
2\mathfrak{Re}\left(\int_\Omega(i\nabla+A_0)\varphi_0\cdot\overline{(i\nabla+A_0)\varphi}\,dx\right), \right. \\
\left. \mathfrak{Im}\left(\int_\Omega\varphi\overline{\varphi_0}\,dx\right),
(i\nabla+A_0)^2\varphi-\lambda_0\varphi-\lambda\varphi_0
\right),
\end{multline*}
for every $(\lambda,\varphi)\in \C\times H^{1,0}_{0}(\Omega,\C)$.
From the simplicity assumption \eqref{eq:1} it follows that $dF(\lambda_0,\varphi_0)$ is invertible, see \cite[Lemma 7.1]{AbatangeloFelli2015-1} for details. 

From the definition of $v_{n_0,R,a}$, \eqref{eq:convergence-varphi-a-2}, \eqref{eq:24}, \eqref{eq:estimate-varphi-a1},
\eqref{eq:estWa}, and \eqref{eq:estimate-vint-n0-1}
it follows that 
\begin{align*}
   \int_{\Omega}\big|(i\nabla+A_0)(v_{n_0,R,a}-\varphi_0)\big|^2\,dx=&
  \int_{\Omega}| e^{i\alpha(\theta_0^{{a}} - \theta_a)}
(i\nabla+A_a)\varphi_a-
  (i\nabla+A_0)\varphi_0|^2\,dx \\
  &\notag -
\int_{D_{R|a|}}| e^{i\alpha(\theta_0^{{a}} - \theta_a)}
(i\nabla+A_a)\varphi_a-
  (i\nabla+A_0)\varphi_0|^2\,dx \\
  &\notag +
\int_{D_{R|a|}}
\big|(i\nabla+A_0)(v_{n_0,R,a}^{int}-\varphi_0)\big|^2\,dx=o(1)
\end{align*} 
as $|a|\to0$, so that $v_{n_0,R,a}\to \varphi_0$ in
$H^{1}_0(\Omega,\C)$ as $|a|\to 0^+$.
Then, from the invertibility of 
$dF(\lambda_0,\varphi_0)$ we have that 
\begin{align}\label{eq:25}
&| \lambda_{a} - \lambda_0 | + \| v_{n_0,R,a} - \varphi_0 \|_{H^{1,0}_0(\Omega,\C)} \\
\notag\quad \leq \| ( &dF(\lambda_0,\varphi_0))^{-1} \|_{ \mathcal L(\R \times \R \times ( H^{1,0}_{0,\R}(\Omega,\C) )^\star 
, \C \times H^{1,0}_{0}(\Omega,\C) ) } 
\| F( \lambda_a , v_{n_0,R,a} ) \|_{\R \times \R \times ( H^{1,0}_{0,\R}(\Omega,\C) )^\star } (1+o(1))
\end{align}
as $|a|\to 0^+$. We denote
\[
F( \lambda_a , v_{n_0,R,a} ) = \left( \alpha_a, \beta_a, w_a \right)
\]
where 
\begin{align*}
\left\{ \begin{aligned}
\alpha_a & = \| v_{n_0,R,a} \|_{ H^{1,0}_0(\Omega,\C) }^2 - \lambda_0 \in \R , \\
\beta_a & = \mathfrak{Im} \left( { \textstyle{ \int_{\Omega} v_{n_0,R,a} \overline{\varphi_0} \, dx } } \right) \in \R , \\
w_a & = (i\nabla+A_0)^2 v_{n_0,R,a} - \lambda_a v_{n_0,R,a} \in (H^{1,0}_{0,\R}(\Omega,\C))^\star.
\end{aligned} \right.
\end{align*}
In view of \eqref{eq:estimate-vint-n0-1}, \eqref{eq:estimate-varphi-a1}, and
Corollary \ref{cor:l0-la_preliminary}
we have that 
\begin{align}\label{eq:26}
\alpha_a 
& = \left( \int_{ D_{R|a|} } | (i\nabla+A_0) v_{n_0,R,a}^{int} |^2 \, dx - 
\int_{ D_{R|a|} } | (i\nabla+A_a) \varphi_a |^2 \, dx \right) + ( \lambda_a - \lambda_0 ) \\
\notag& = O(H(\varphi_a, K_\delta |a|))+O\Big((H(\varphi_a, K_\delta
  |a|))^{\frac{|\alpha-k|}{|\alpha-k|+\delta}}\Big)
=o(\sqrt{H(\varphi_a, K_\delta |a|)})
\end{align}
as $|a|\to0^+$.
The normalization condition for the phase in \eqref{eq:normalization}
together with \eqref{eq:estimate-vint-n0-2}, \eqref{eq:estWa}, and \eqref{eq:estimate-varphi-a3} yield
\begin{align}\label{eq:27}
\beta_a 
& = \Im \left( \int_{ D_{R|a|} } v_{n_0,R,a}^{int} \overline{\varphi_0} \, dx 
- \int_{D_{R|a|}} e^{ i \alpha (\theta_0^a-\theta_a) } \varphi_a \overline{\varphi_0} \, dx
+ \int_{\Omega} e^{ i \alpha (\theta_0^a-\theta_a) } \varphi_a
  \overline{\varphi_0} \, dx \right) \\
\notag& = \Im \left( \int_{ D_{R|a|} } v_{n_0,R,a}^{int} \overline{\varphi_0} \, dx 
- \int_{D_{R|a|}} e^{ i \alpha (\theta_0^a-\theta_a) } \varphi_a \overline{\varphi_0} \, dx
\right) \\
\notag& = O(|a|^{2+|\alpha-k|}\sqrt{H(\varphi_a, K_\delta |a|)}) =o(\sqrt{H(\varphi_a, K_\delta |a|)})
\end{align}
as $|a|\to0^+$.

For every $a\in\Omega$, the map
\[
\mathcal T_a:\mathcal D^{1,2}_0(\R^2,\C)\to \mathcal
D^{1,2}_0(\R^2,\C),\quad  \mathcal T_a\varphi(x)=\varphi(|a|x).
\]
is an isometry of $\mathcal
D^{1,2}_0(\R^2,\C)$. 

Since $H^{1,0}_{0}(\Omega,\C)$ is continuously
embedded into $\mathcal D^{1,2}_0(\R^2,\C)$ by trivial extension
outside $\Omega$ and $\|u\|_{\mathcal D^{1,2}_0(\R^2,\C)}=
\|u\|_{H^{1,0}_{0}(\Omega,\C)}$ for every $u\in H^{1,0}_{0}(\Omega,\C)$, we have that 
\begin{align}\label{eq:111}
 &\|w_a\|_{(H^{1,0}_{0,\R}(\Omega,\C))^\star} \\
 \notag & = \sup_{\substack{\varphi\in H^{1,0}_{0}(\Omega,\C) \\
      \|\varphi\|_{H^{1,0}_0(\Omega,\C)}=1}} \bigg|\Re
  \left(\int_{\Omega}(i\nabla+A_0)v_{n_0,R,a}\cdot\overline{(i\nabla+A_0)\varphi}\,dx
    -\lambda_a \int_{\Omega} v_{n_0,R,a}\overline{\varphi}
    \,dx\right)\bigg| \\
 \notag & \leq \sup_{\substack{\varphi\in \mathcal D^{1,2}_0(\R^2,\C)\\
      \|\varphi\|_{\mathcal D^{1,2}_0(\R^2,\C)}=1}} \bigg|\Re
  \left(\int_{\Omega}(i\nabla+A_0)v_{n_0,R,a}\cdot\overline{(i\nabla+A_0)\varphi}\,dx
    -\lambda_a \int_{\Omega} v_{n_0,R,a}\overline{\varphi}
    \,dx\right)\bigg|.
\end{align}
For every $\varphi \in  \mathcal D^{1,2}_0(\R^2,\C) $ we have that 
\begin{align}\label{eq:112}
  &\int_{\Omega}(i\nabla+A_0)v_{n_0,R,a}\cdot\overline{(i\nabla+A_0)\varphi}\,dx
    -\lambda_a \int_{\Omega} v_{n_0,R,a}\overline{\varphi}
    \,dx\\
\notag&=\int_{\Omega\setminus D_{R|a|}}
 e^{i\alpha (\theta_0^a-\theta_a)}  (i\nabla+A_a)\varphi_a\cdot\overline{(i\nabla+A_0)\varphi}\,dx
    -\lambda_a \int_{\Omega\setminus D_{R|a|}} e^{i\alpha (\theta_0^a-\theta_a)} \varphi_a\overline{\varphi}
    \,dx\\
&\notag\qquad+
\int_{D_{R|a|}}(i\nabla+A_0)v_{n_0,R,a}\cdot\overline{(i\nabla+A_0)\varphi}\,dx
    -\lambda_a \int_{D_{R|a|}} v_{n_0,R,a}\overline{\varphi}
    \,dx.
\end{align}
From scaling and integration by parts we have that, letting
$\tilde\varphi_a$ be defined in \eqref{eq:tilde-varphi-a},
\begin{multline}\label{eq:113}
 \int_{\Omega\setminus D_{R|a|}}
 e^{i\alpha (\theta_0^a-\theta_a)}  (i\nabla+A_a)\varphi_a\cdot\overline{(i\nabla+A_0)\varphi}\,dx
    -\lambda_a \int_{\Omega\setminus D_{R|a|}} e^{i\alpha(\theta_0^a-\theta_a)} \varphi_a\overline{\varphi}
    \,dx\\
=
i\sqrt{H(\varphi_a, K_\delta |a|)}\int_{\partial D_{R}}\overline{\mathcal T_a\varphi}\,
e^{i\alpha (\theta_0^p-\theta_p)} (i\nabla+A_p)\tilde\varphi_a\cdot
  \nu\,d\sigma
\end{multline}
being $\nu=\frac x{|x|}$ the outer unit normal vector. 
In a similar way we have that, defining $Z_a^R$ as in \eqref{eq:ZaR} and using \eqref{eq:vextinteq}, 
\begin{multline}\label{eq:114}
  \int_{D_{R|a|}}(i\nabla+A_0)v_{n_0,R,a}\cdot\overline{(i\nabla+A_0)\varphi}\,dx
    -\lambda_a \int_{D_{R|a|}} v_{n_0,R,a}\overline{\varphi}
    \,dx\\
=\sqrt{H(\varphi_a, K_\delta |a|)}\left(-i\int_{\partial D_{R}}
   (i\nabla+A_0)Z_a^R\cdot\nu \overline{\mathcal T_a \varphi}\,d\sigma
    -\lambda_a |a|^2\int_{D_{R}}Z_a^R\overline{\mathcal T_a\varphi}
    \,dx\right).
\end{multline}
Combining \eqref{eq:111}, \eqref{eq:112}, \eqref{eq:113}, \eqref{eq:114}, and
recalling that $\mathcal T_a$ is an isometry of $\mathcal D^{1,2}_0(\R^2,\C)$, we obtain
that
\begin{equation}\label{eq:28}
(H(\varphi_a, K_\delta
|a|))^{-1/2}\|w_a\|_{(H^{1,0}_{0,\R}(\Omega,\C))^\star}\leq h(p,a,R)+\lambda_a |a|^2\sup_{\substack{\varphi\in \mathcal D^{1,2}_0(\R^2,\C)\\
      \|\varphi\|_{\mathcal D^{1,2}_0(\R^2,\C)}=1}}\bigg|
\int_{D_{R}}Z_a^R\overline{\varphi}
    \,dx\bigg|
  \end{equation}
where 
\[
h(p,a,R)=\sup_{\substack{\varphi\in \mathcal D^{1,2}_0(\R^2,\C)\\
      \|\varphi\|_{\mathcal D^{1,2}_0(\R^2,\C)}=1}} \bigg|\int_{\partial D_{R}}
\left(e^{i\alpha (\theta_0^p-\theta_p)} (i\nabla+A_p)\tilde\varphi_a-(i\nabla+A_0) Z_a^R\right)\cdot
  \nu\,\overline{\varphi}\,d\sigma\bigg|.
\]
From Remarks \ref{rem:3.7} and \ref{rem:4.5} it follows that, for $R>\max\{2,K_\delta\}$
and $p\in{\mathbb S}^1$ fixed, 
\begin{equation*}
\left\{\left(e^{i\alpha (\theta_0^p-\theta_p)} (i\nabla+A_p)\tilde\varphi_a-(i\nabla+A_0) Z_a^R\right)\cdot
  \nu\right\}_{|a|<r_\delta/R}\quad\text{is bounded in }H^{-1/2}(\partial D_R)
\end{equation*}
so that, for $p$ and $R$ fixed, $h(p,a,R)=O(1)$ as $a\to0$.
Moreover Remark \ref{rem:4.5} implies that,  for $R>\max\{2,K_\delta\}$
and $p\in{\mathbb S}^1$ fixed, 
\[
\sup_{\substack{\varphi\in \mathcal D^{1,2}_0(\R^2,\C)\\
      \|\varphi\|_{\mathcal D^{1,2}_0(\R^2,\C)}=1}}\bigg|
\int_{D_{R}}Z_a^R\overline{\varphi}
    \,dx\bigg|=O(1)\quad\text{as }|a|\to0^.
\]
 Hence
the conclusion follows from \eqref{eq:25}, \eqref{eq:26},
\eqref{eq:27}, and \eqref{eq:28}.
\end{proof}

The previous theorem allows estimating the energy variation of scaled
eigenfunctions and improving the results of Corollary \ref{cor:Ha} as follows.

\begin{Corollary}\label{cor:improved}
Let $p\in {\mathbb S}^1$ be fixed. Then 
\begin{itemize}
\item[$(i)$] $|a|^{2|\alpha-k|}=O(H(\varphi_a, K_\delta |a|))$ as $a=|a|p\to0$;
\item[$(ii)$] letting $\tilde{\varphi}_a$ and $W_a$ be as in
  \eqref{eq:tilde-varphi-a} and \eqref{eq:def-Wa}, for every $R>\max\{2,K_\delta\}$ there holds
\begin{equation}\label{eq:tilde_phi-W}
\int_{\frac{\Omega}{|a|} \setminus D_R} 
\left| (i\nabla + A_{p}) 
\left( \tilde{\varphi}_a - e^{i\alpha(\theta_{p} - \theta_0^{p})} W_a \tfrac{|a|^{|\alpha-k|}}{\sqrt{H(\varphi_a,K_\delta |a|})} 
\right) \right|^2 \, dx  = O(1),
\quad \text{ as } a=|a|p \to 0.
\end{equation}
\end{itemize}
\end{Corollary}
\begin{proof}
 Estimate \eqref{eq:tilde_phi-W} follows from  scaling  and Theorem
 \ref{thm:energy_estimates}.
From \eqref{eq:tilde_phi-W} it follows that
\begin{align*}
\tfrac{|a|^{|\alpha-k|}}{\sqrt{H(\varphi_a,K_\delta |a|})} &\left(\int_{D_{2R}\setminus D_{R}}\big|(i\nabla+A_{0})
 W_a\big|^2dx\right)^{1/2} \\
&=\tfrac{|a|^{|\alpha-k|}}{\sqrt{H(\varphi_a,K_\delta |a|})} \left(\int_{D_{2R}\setminus D_{R}}\Big|(i\nabla+A_{p})
 \Big(e^{i\alpha(\theta_{p}-\theta_0^p)}W_a\Big)\Big|^2dx\right)^{1/2}
  \\
&
 \leq O(1) + \left(\int_{D_{2R}\setminus D_{R}}\bigg|(i\nabla+A_{p})\tilde \varphi_a(x)\bigg|^2dx\right)^{1/2} 
\end{align*}
as $a=|a|p\to0$. From Remark \ref{rem:3.7} and \eqref{eq:conv-Wa}, the
above estimate implies (i).
\end{proof}

In the following lemma we prove the existence and uniqueness of the
function $\Psi_p$ satisfying \eqref{eq:Psip1} and \eqref{eq:Psip2}, which will
turn out to be the limit of the blowed-up family
\eqref{eq:tilde-varphi-a} as $a\to0$ along the fixed direction $p\in{\mathbb S}^1$.

\begin{Lemma}\label{lem:uniqueness}
Let $p\in{\mathbb S}^1$. There exists a unique 
$\Psi_{p} \in H^{1,p}_{\rm loc}(\R^2,\C)$
satisfying \eqref{eq:Psip1} and \eqref{eq:Psip2}.
\end{Lemma}

\begin{proof}
Let $\eta$ be a smooth cut-off function such that $\eta \equiv 0$ in $D_1$ 
and $\eta \equiv 1$ in $\R^2 \setminus D_R$ for some $R > 1$. Recalling the definition of $\psi_k$ \eqref{eq:psi_j_definition}, we have
\[
\begin{split}
F & = (i\nabla + A_{p})^2 (\eta e^{ i \alpha ( \theta_{p}-\theta_0^{p})} \psi_k) \\
& = - \Delta \eta e^{ i \alpha ( \theta_{p} - \theta_0^{p} ) } \psi_k 
+ 2 i \nabla \eta \cdot (i\nabla+A_{p})( e^{ i \alpha ( \theta_{p} - \theta_0^{p} ) } \psi_k) 
+ \eta (i\nabla+A_{p})^2 ( e^{ i \alpha ( \theta_{p} - \theta_0^{p} ) } \psi_k) \\
& = - \Delta \eta e^{ i \alpha ( \theta_{p} - \theta_0^{p} ) } \psi_k 
+ 2 i \nabla \eta \cdot (i\nabla+A_{p})( e^{ i \alpha ( \theta_{p} - \theta_0^{p} ) } \psi_k) 
\in (\mathcal{D}^{1,2}_{p}(\R^2,\C))^*.
\end{split}
\]
Here $\mathcal{D}^{1,2}_{p}(\R^2,\C)$ is the completion of 
$C^\infty_c(\R^2\setminus\{0\},\C)$ with respect to 
\[
\| u \|_{\mathcal{D}^{1,2}_{p}(\R^2,\C)}
= \left( \int_{\R^2} |(i\nabla +A_{p}) u(x) |^2 \, dx \right)^{\!\!1/2}.
\]
By the Lax-Milgram's Theorem, there exists a unique 
$g \in \mathcal{D}^{1,2}_{p}(\R^2,\C)$ which solves
\[
(i\nabla+A_{p})^2 g = - F, \quad \text{in }
(\mathcal{D}^{1,2}_{p}(\R^2,\C))^*.
\]
Then, $\Psi_{p} = g + \eta e^{ i \alpha ( \theta_{p} - \theta_0^{p} ) } \psi_k$ 
satisfies \eqref{eq:Psip1} and \eqref{eq:Psip2}, so that the existence is proved.

The uniqueness follows from the fact that, if $\Psi_{p}^1$, 
$\Psi_{p}^2 \in H^{1,p}_{\rm loc}(\R^2,\C)$ satisfy 
\eqref{eq:Psip1} and \eqref{eq:Psip2}, then
\begin{equation}\label{eq:30}
(i\nabla+A_{p})^2 ( \Psi_{p}^1 - \Psi_{p}^2 ) = 0, 
\quad \text{in } (\mathcal{D}^{1,2}_{p}(\R^2,\C))^*,
\end{equation}
and
\[
\int_{\R^2} |(i\nabla +A_{p}) (\Psi_{p}^1 - \Psi_{p}^2) |^2 \, dx <+\infty,
\]
which, in view of the Hardy inequality \eqref{eq:hardy}, implies that 
\[
\int_{\R^2} \frac{|\Psi_{p}^1 - \Psi_{p}^2 |^2}{|x-p|^2} \, dx <+\infty,
\]
and hence that $\Psi_{p}^1 - \Psi_{p}^2\in
\mathcal{D}^{1,2}_{p}(\R^2,\C)$. Therefore we can test equation
\eqref{eq:30} with $\Psi_{p}^1 - \Psi_{p}^2$ thus concluding that
\[
\int_{\R^2} |(i\nabla +A_{p}) (\Psi_{p}^1 - \Psi_{p}^2) |^2 \, dx = 0,
\]
which implies that $\Psi_{p}^1 \equiv \Psi_{p}^2$.
\end{proof}

We are now in position to prove that the scaled eigenfunctions
\eqref{eq:tilde-varphi-a} converge to a multiple of $\Psi_p$ as $a=|a|p\to0$.

\begin{Lemma} \label{lemma:convergence-blow-up}
Let $p\in{\mathbb S}^1$ and  
$\delta \in (0, \sqrt{\mu_1}/2)$ be fixed and let $K_\delta > \Upsilon_\delta$ 
be as in Lemma~\ref{lemma:estimate-N}. For $a = |a| p \in \Omega$ 
let $\tilde\varphi_a$ be as in \eqref{eq:tilde-varphi-a}. Then
\begin{equation*}
\tilde\varphi_a \to \frac{\beta}{|\beta|} 
\left( \frac{K_\delta}{\int_{\partial D_{K_\delta}} |\Psi_{p}|^2 ds } \right)^{\!\!1/2} 
\Psi_{p}  \quad \text{ as } a=|a|p \to 0,
\end{equation*}
in $H^{1,p}(D_R,\C)$ for every $R > 1$ and in $C^2_{\rm loc}(\R^2\setminus\{p\},\C)$, 
where $\Psi_{p} $ is the function defined in Lemma \ref{lem:uniqueness}. 
Moreover, 
\begin{equation}\label{eq:HabigO}
\lim_{a=|a|p\to0} \frac{ |a|^{|\alpha-k|} }{ \sqrt{H( \varphi_a,K_\delta|a|) } } 
= \frac{1}{|\beta|} \left( \frac{K_\delta}{\int_{\partial D_{K_\delta}} |\Psi_{p}|^2 ds } \right)^{\!\!1/2} .
\end{equation}
\end{Lemma}
\begin{proof}
From Remark \ref{rem:3.7} and Corollary
  \ref{cor:improved} it follows that,
for every 
sequence $a_n = |a_n|p$ with $|a_n| \to 0$, there exist a subsequence 
$a_{n_\ell}$, $c \in [0,+\infty)$ and $\tilde\Phi \in H^{1,p}_{\rm loc}(\R^2,\C)$ 
such that
\begin{equation*}
\tilde\varphi_{a_{n_\ell}} \rightharpoonup \tilde\Phi
\text{ weakly in } H^{1,p}(D_R,\C) \, \text{ as } \ell \to + \infty
\quad \text{ and } \quad
\lim_{\ell \to + \infty} \frac{ |a_{n_\ell}|^{|\alpha-k|} }{ \sqrt{H(\varphi_{a_{n_\ell}},K_\delta|a_{n_\ell}|)} } = c
\end{equation*}
for every $R>1$. Passing to the limit in the
equation satisfied by $\tilde\varphi_{a}$, i.e. $(i\nabla + A_p)^2
\tilde\varphi_a = \lambda_a|a|^2 \tilde\varphi_a$ in $\frac1{|a|}\Omega$, we obtain that
$\tilde\Phi$ satisfies
\begin{equation}\label{eq:tildePhieq}
(i\nabla+A_{p})^2 \tilde\Phi = 0 \quad\text{in } \mathbb{R}^2.
\end{equation}
Moreover, by compact trace embeddings, 
\begin{equation}\label{eq:tildePhinorm}
\frac{1}{K_\delta} \int_{\partial D_{K_\delta}} |\tilde{\Phi}|^2 \, ds = 1,
\end{equation}
so that $\tilde{\Phi}$ is not identically zero.
Testing the equation for $\tilde\varphi_{a}$ with
  $\tilde\varphi_{a}$ itself, integrating by parts and exploiting the
  $C^2_{\rm loc}$-convergence of $\tilde\varphi_{a}$ in $\R^2\setminus\{p\}$
  (which follows from classic  elliptic estimates) we obtain that
  $\int_{D_R}|(i\nabla+A_p) \tilde\varphi_{a_{n_\ell}}|^2\,dx\to
  \int_{D_R}|(i\nabla+A_p) \tilde\Phi|^2\,dx$ as $\ell\to\infty$ for
  every $R>1$. Hence we conclude that, for all $R>1$,
  $\tilde\varphi_{a_{n_\ell}}\to \tilde\Phi$ strongly in 
 $H^{1,p}(D_R,\C)$ as $\ell \to + \infty$.

By the strong $H^{1,p}_{\rm loc}(\R^2,\C)$-convergence and recalling \eqref{eq:conv-Wa}, 
we can pass to the limit along $a_{n_\ell}$ in \eqref{eq:tilde_phi-W}, to obtain
\begin{equation*}
\int_{\R^2\setminus D_R} |(i\nabla + A_{p}) 
(\tilde\Phi - c \beta e^{ i \alpha ( \theta_{p} - \theta_0^{p} ) } \psi_k)|^2 \, dx < + \infty.
\end{equation*}
This implies $c \neq 0$ (and hence $c>0$), otherwise we would have $\int_{\R^2\setminus D_R} 
|(i\nabla +A_{p}) \tilde\Phi |^2 \, dx < + \infty$, which together with 
\eqref{eq:tildePhieq} implies $\tilde\Phi\equiv0$, thus contradicting \eqref{eq:tildePhinorm}.

Then Lemma~\ref{lem:uniqueness} and \eqref{eq:tildePhinorm} provide
\begin{equation*}
\tilde \Phi = c \beta \Psi_{p} 
\quad \text{ and } \quad
c = \frac{1}{|\beta|} \left( \frac{K_\delta}{\int_{\partial D_{K_\delta}}|\Psi_{p}|^2} \right)^{\!\!1/2}.
\end{equation*}
Since these limits depend neither on the sequence, nor on the subsequence, the proof is complete.~\end{proof}

\begin{proof}[Proof of Theorem \ref{thm:eigenvalues}]
Let $p \in \mathbb{S}^1$. 
From Corollary \ref{cor:l0-la_preliminary} part (i) and
\eqref{eq:HabigO} we conclude that $\lambda_0 - \lambda_a = O(|a|^{2 |\alpha - k|})$ as $a=|a|p \to 0$.
Since the function $a \mapsto \lambda_a$ is analytic in a neighborhood of $0$, 
being $\lambda_0$ simple (see \cite[Theorem 1.3]{Lena2015}),
and since  $2 |\alpha - k|$ 
is non-integer, we have that the Taylor polynomials of the function $\lambda_0 - \lambda_a$ 
with center $0$ and  degree less or equal than $\lfloor 2|\alpha-k|\rfloor$ vanish, 
thus yielding the conclusion.
\end{proof}

\begin{proof}[Proof of Theorem \ref{t:bu}]
It is a direct consequence of Lemma \ref{lemma:convergence-blow-up}.
\end{proof}

\section{Rate of convergence for eigenfunctions}\label{sec:rate-conv-eigenf}

Taking inspiration from \cite{AbatangeloFelli2017}, we now estimate the rate of convergence of the eigenfunctions. We then take into 
account the quantity
\[
\left\| (i\nabla+A_a) \varphi_a - e^{i\alpha (\theta_a-\theta_0^a)}
(i\nabla + A_0)\varphi_0 \right\|_{L^2(\Omega,\C)}
\]
and we split the argument in two different steps, the first considering the 
energy variation inside small disks of radius $R|a|$, the second considering 
the energy variation outside these disks.

\begin{Lemma}\label{l:energy_inside}
Under the same assumptions as in Theorems \ref{thm:eigenvalues} and \ref{t:bu}, we have that, for every $p \in {\mathbb S}^1$ and $R > 1$,
\begin{equation}\label{eq:31}
\lim_{ a = |a| p \to 0} \frac{1}{|a|^{2|\alpha-k|}} \int_{D_{R|a|}} 
\left| (i\nabla+A_a) \varphi_a - e^{ i \alpha (\theta_a -  \theta_0^a)} 
(i\nabla+A_0)\varphi_0 \right|^2 \, dx = |\beta|^2 \, \mathcal{F}_p(R)
\end{equation}
where
\[
\mathcal{F}_p(R) = \int_{D_{R}} \left| (i\nabla+A_{p}) \Psi_{p}
- e^{ i \alpha ( \theta_{p} - \theta_0^{p} ) } 
(i\nabla+A_0) \psi_k  \right|^2 \, dx.
\]
Moreover 
\[
\mathfrak{L}_{p}:=\lim_{R \to + \infty} \mathcal{F}_p(R) \in (0,+\infty).
\]
\end{Lemma}
\begin{proof}
We notice that, in view of \eqref{eq:Psip2}, $\mathfrak{L}_p<+\infty$.
The proof of \eqref{eq:31} relies on a change of variables and on the convergences stated in 
\eqref{eq:conv-Wa} and in Theorem \ref{t:bu}. We have that 
\begin{align*}
 \lim_{R\to+\infty} \mathcal{F}_p(R)&= 
\int_{\R^2} \left| (i\nabla+A_{p}) \Psi_{p}
- e^{ i \alpha ( \theta_{p} - \theta_0^{p} ) } 
(i\nabla+A_0) \psi_k \right|^2 \, dx\\
&=
\int_{\R^2 \setminus \Gamma_{p}} 
 \left|(i\nabla+A_{p}) (\Psi_{p}
 - e^{i\alpha(\theta_{p}-\theta_0^{p})}\psi_k)\right|^2 \,dx>0,
\end{align*}
where $\Gamma_p$ is defined in \eqref{eq:gamma-b}.
Indeed, suppose by contradiction that the above limit is zero. 
Since, for every $r_1>r_2>1$,
$\Psi_{p} - e^{i\alpha(\theta_{p}-\theta_0^{p})}\psi_k\in
 H^{1,p}(D_{r_1}(p)\setminus D_{r_2}(p),\C)$,
 the Hardy inequality \eqref{eq:anello} implies that $\Psi_{p} -
 e^{i\alpha(\theta_{p}-\theta_0^{p})}\psi_k\equiv 0$ in $\R^2\setminus
 D_1(p)$. Moreover, since $(i\nabla+A_{p})^2\big( \Psi_{p} 
- e^{ i \alpha ( \theta_{p} - \theta_0^{p} )}\psi_k\big)=~0$ in
$\R^2\setminus \Gamma_p$, a classical unique continuation principle
(see e.g. \cite{wolff}) implies that $\Psi_{p} -
 e^{i\alpha(\theta_{p}-\theta_0^{p})}\psi_k\equiv 0$ in $\R^2\setminus
 \Gamma_p$  necessarily.
But this is impossible since, by \eqref{eq:Psip1}
and classical elliptic estimates away from $p$, 
$\Psi_p$ is smooth in $\R^2\setminus \{p\}$, whereas $e^{ i \alpha (  \theta_{p}- \theta_0^{p} )} \psi_k$ 
is discontinuous on $\Gamma_{p}\setminus\{0\}$ since it is the product of the continuous 
non-zero function $\psi_k$ and of the discontinuous function 
$e^{ i \alpha (  \theta_{p}- \theta_0^{p} )}$ 
(see the definitions \eqref{eq:theta-b}, \eqref{eq:theta_0^b} and \eqref{eq:gamma-b}).
\end{proof}

Before addressing the energy variation outside the disk, 
it is worthwhile introducing a preliminary result. For all $R > 2$ and 
$p \in \mathbb{S}^1$, let $z_{p,R}$ be the unique solution to 
\begin{equation}\label{eq:zpReq}
\begin{cases}
(i \nabla + A_0)^2 z_{p,R} = 0, & \text{in } D_R, \\
z_{p,R} = e^{i \alpha (\theta_0^{p} - \theta_{p})} \Psi_{p}, & \text{on } \partial D_R.
\end{cases}
\end{equation}
From Lemma~\ref{lemma:convergence-blow-up} it follows that the family of functions 
$Z_a^R$ introduced in \eqref{eq:ZaR} converges in $H^{1,0}(D_R,\mathbb{C})$ to some 
multiple of $z_{p,R}$.

\begin{Lemma}\label{l:blowZ}
Let $p \in {\mathbb S}^1$ and $R>2$. For $a =|a|p \in \Omega$, let $Z_a^R$ be as in \eqref{eq:ZaR}. Then
\[
Z_a^R \to \frac{\beta}{|\beta|} \left(\frac{K_\delta}{\int_{\partial D_{K_\delta}} 
|\Psi_{p}|^2 ds } \right)^{\!\!1/2} z_{p,R},
\]
in $H^{1,0}(D_R,\mathbb{C})$, as $|a|\to0^+$.
\end{Lemma}
\begin{proof}
Denote
$
\gamma_{p,\delta}=
\frac{\beta}{|\beta|}\Big(\frac{K_\delta}{\int_{\partial D_{K_\delta}} |\Psi_{p}|^2 ds } \Big)^{\!1/2}$.
By \eqref{eq:vextinteq} and \eqref{eq:zpReq} we have that 
$Z_a^R - \gamma_{p,\delta}z_{p,R}$ solves
\begin{equation*}
\begin{cases}
(i \nabla + A_0)^2 ( Z_a^R - \gamma_{p,\delta}z_{p,R}) 
= 0, & \text{in } D_R \\
Z_a^R - \gamma_{p,\delta}z_{p,R} 
= e^{i \alpha ( \theta_0^{p} - \theta_{p} ) } 
(\tilde{\varphi}_a - \gamma_{p,\delta}\Psi_{p}) ,
 & \text{on } \partial D_R.
\end{cases}
\end{equation*}
For $R > 2$, let $\eta_R \, : \, \R^2 \to \R$ be a smooth cut-off function such that 
\begin{equation}\label{eq:eta_def}
\eta_R \equiv 0 \textrm{ in }D_{R/2}, \quad \eta_R \equiv 1 \textrm{ in }\R^2 \setminus D_{R}, \quad 0 \leq \eta_R  \leq 1.
\end{equation}
Then, by the Dirichlet principle and  Lemma~\ref{lemma:convergence-blow-up},
\begin{align*}
& \int_{D_R} |(i\nabla+A_0)( Z_a^R - \gamma_{p,\delta}z_{p,R} ) |^2 \, dx 
\leq  \int_{D_R} |(i\nabla+A_0) ( \eta_R e^{ i \alpha ( \theta_0^{p} - \theta_{p} ) } 
( \tilde{\varphi}_a - \gamma_{p,\delta} \Psi_{p} ) ) |^2 \,dx \\
& \leq 2 \int_{D_R} | \nabla \eta_R |^2 | \tilde{\varphi}_a - \gamma_{p,\delta} \Psi_{p} |^2 \, dx 
+ 2 \int_{D_R \setminus D_{R/2}} \eta_R^2 | (i\nabla+A_{p})
(\tilde{\varphi}_a - \gamma_{p,\delta}\Psi_{p})|^2 \, dx =o(1),
\end{align*}
as $a=|a|p\to0$. 
Finally, the Hardy type inequality \eqref{eq:hardy} allows us to conclude.
\end{proof}

\begin{Lemma}\label{l:energy_outside}
  Let $\varphi_0 \in H^{1,0}_{0}(\Omega,\C)$ be a solution to
  \eqref{eq:equation_lambda0} satisfying \eqref{eq:1}.  Let
  $p \in \mathbb{S}^1$. For $a = |a| p \in \Omega$, let
  $\varphi_a \in H^{1,a}_{0}(\Omega,\C)$ satisfy
  \eqref{eq:equation_a}--\eqref{eq:normalization}. Then, for all
  $R>\max\{2,K_\delta\}$,
\begin{equation*}
\| e^{i\alpha ( \theta_0^a - \theta_a ) } (i\nabla+A_a) \varphi_a - 
(i\nabla+A_0)\varphi_0 \|^2_{L^2(\Omega\setminus D_{R|a|},\C)} \leq |a|^{2|\alpha-k|} G(p,a,R),
\end{equation*}
where
$\lim_{a=|a|p\to0}G(p,a,R)=G(p,R)$ for some $G(p,R)$ such that 
\begin{equation}\label{eq:33}
\lim_{R\to+\infty}G(p,R)=0.
\end{equation}
\end{Lemma}

\begin{proof}
Let $R>\max\{2,K_\delta\}$. From Theorem \ref{thm:energy_estimates} and
\eqref{eq:HabigO} we have that 
\begin{align*}
\| e^{i\alpha ( \theta_0^a - \theta_a ) } &(i\nabla+A_a) \varphi_a - 
(i\nabla+A_0)\varphi_0 \|_{L^2(\Omega\setminus D_{R|a|},\C)}
\leq
\| v_{n_0,R,a} - \varphi_0\|_{H^{1,0}_0(\Omega,\C)} \\
&\leq C\Big(h(p,a,R)+g(p,a,R)\Big) |a|^{|\alpha-k|}
\end{align*}
where $g(p,a,R)=o(1)$ as $|a|\to0^+$ and
\begin{align*}
&h(p,a,R)
=\sup_{\substack{\varphi\in \mathcal D^{1,2}_0(\R^2,\C)\\
      \|\varphi\|_{\mathcal D^{1,2}_0(\R^2,\C)}=1}} \bigg|\int_{\partial D_{R}}
(i\nabla+A_0) \left(e^{i\alpha (\theta_0^p-\theta_p)} \tilde\varphi_a- Z_a^R\right)\cdot
  \nu\,\overline{\varphi}\,d\sigma\bigg|\\
&\leq {\rm const\,}\left\|(i\nabla+A_0) \left(e^{i\alpha (\theta_0^p-\theta_p)} \tilde\varphi_a- Z_a^R\right)\!\cdot
\!  \nu-
\gamma_{p,\delta} (i\nabla+A_0) \left(e^{i\alpha (\theta_0^p-\theta_p)}\Psi_p- z_{p,R}\right)\!\cdot\!\nu
\right\|_{H^{-\frac12}(\partial D_R)}\\
&+\gamma_{p,\delta}
\sup_{\substack{\varphi\in \mathcal D^{1,2}_0(\R^2,\C)\\
      \|\varphi\|_{\mathcal D^{1,2}_0(\R^2,\C)}=1}} \bigg|\int_{\partial D_{R}}
(i\nabla+A_0) \left(e^{i\alpha (\theta_0^p-\theta_p)} \Psi_p- z_{p,R}\right)\cdot
  \nu\,\overline{\varphi}\,d\sigma\bigg|
\end{align*}
where $\gamma_{p,\delta}=
\frac{\beta}{|\beta|}\Big(\frac{K_\delta}{\int_{\partial D_{K_\delta}}
  |\Psi_{p}|^2 ds } \Big)^{\!1/2}$ and the constant ${\rm const}>0$ is
independent of $a$.
From Lemmas \ref{lemma:convergence-blow-up} and \ref{l:blowZ}  we have that 
\[
(i\nabla+A_0) \left(e^{i\alpha (\theta_0^p-\theta_p)} \tilde\varphi_a- Z_a^R\right)\cdot
  \nu\to 
\gamma_{p,\delta} (i\nabla+A_0) \left(e^{i\alpha (\theta_0^p-\theta_p)}\Psi_p- z_{p,R}\right)\cdot\nu
\]
in $H^{-1/2}(\partial D_R)$ as $a=|a|p\to0$. Therefore
$h(p,a,R)\leq f(p,a,R)$
with 
\[
\lim_{a=|a|p\to0}f(p,a,R)=\gamma_{p,\delta}
\sup_{\substack{\varphi\in \mathcal D^{1,2}_0(\R^2,\C)\\
      \|\varphi\|_{\mathcal D^{1,2}_0(\R^2,\C)}=1}} \bigg|\int_{\partial D_{R}}
(i\nabla+A_0) \left(e^{i\alpha (\theta_0^p-\theta_p)} \Psi_p- z_{p,R}\right)\cdot
  \nu\,\overline{\varphi}\,d\sigma\bigg|.
\]
To complete the proof is then enough to show that 
\begin{equation}\label{eq:32}
\lim_{R\to+\infty}
\sup_{\substack{\varphi\in \mathcal D^{1,2}_0(\R^2,\C)\\
      \|\varphi\|_{\mathcal D^{1,2}_0(\R^2,\C)}=1}} \bigg|\int_{\partial D_{R}}
(i\nabla+A_0) \left(e^{i\alpha (\theta_0^p-\theta_p)} \Psi_p- z_{p,R}\right)\cdot
  \nu\,\overline{\varphi}\,d\sigma\bigg|=0.
\end{equation}
Using an integration by parts we can rewrite
\begin{multline*}
\bigg|\int_{\partial D_{R}}
(i\nabla+A_0) \left(e^{i\alpha (\theta_0^p-\theta_p)} \Psi_p- z_{p,R}\right)\cdot
  \nu\,\overline{\varphi}\,d\sigma\bigg| \\
= \bigg| \int_{\partial D_R} e^{i\alpha(\theta_0^p-\theta_p)}(i\nabla+A_p)(\Psi_p-e^{i\alpha(\theta_p-\theta_0^p)}\psi_k)\cdot\nu\overline{\varphi}\,d\sigma 
+\int_{\partial D_R}(i\nabla+A_0)(\psi_k-z_{p,R})\cdot\nu \overline{\varphi}\,d\sigma\bigg|\\
=\bigg|-i\int_{\R^2\setminus D_R} (i\nabla+A_p)(\Psi_p-e^{i\alpha(\theta_p-\theta_0^p)}\psi_k)\cdot\overline{(i\nabla+A_0)\varphi}e^{i\alpha(\theta_0^p-\theta_p)} \,dx \\
+i\int_{D_R} (i\nabla+A_0)(\psi_k-z_{p,R})\cdot\overline{(i\nabla+A_0)\varphi}\,dx\bigg|,
\end{multline*}
which implies
\begin{multline}\label{eq:sup_estimate}
\sup_{\substack{\varphi\in \mathcal D^{1,2}_0(\R^2,\C)\\
      \|\varphi\|_{\mathcal D^{1,2}_0(\R^2,\C)}=1}} \bigg|\int_{\partial D_{R}}
(i\nabla+A_0) \left(e^{i\alpha (\theta_0^p-\theta_p)} \Psi_p- z_{p,R}\right)\cdot
  \nu\,\overline{\varphi}\,d\sigma\bigg| \\\leq 
  \left(\int_{\R^2\setminus D_R} |(i\nabla+A_p)(\Psi_p-e^{i\alpha(\theta_p-\theta_0^p)}\psi_k)|^2\,dx\right)^{1/2}
  +\left(\int_{D_R}|(i\nabla+A_0)(\psi_k-z_{p,R})|^2 \,dx\right)^{1/2}.
\end{multline}
The first term in the right hand side of \eqref{eq:sup_estimate} goes
to zero as $R\to+\infty$ because of \eqref{eq:Psip2}.
 To estimate the
second term, we consider a test function $\eta_R$ satisfying
\eqref{eq:eta_def} and the additional property
$|\nabla\eta_R|\leq 4/R$ in $D_R\setminus D_{R/2}$. Recalling that
$\psi_k-z_{p,R}$ satisfies 
$(i\nabla+A_0)^2(\psi_k-z_{p,R})=0$ in $D_R$ with the boundary condition
$\psi_k-z_{p,R}=\psi_k-e^{i\alpha(\theta_0^p-\theta_p)}\Psi_p$ on
$\partial D_R$, the Dirichlet principle and the Hardy inequality
\eqref{eq:anello} provide
\begin{align*}
\int_{D_R} &|(i\nabla+A_0)(\psi_k-z_{p,R})|^2\,dx \leq 
\int_{D_R} |(i\nabla+A_0)(\eta_R(\psi_k-e^{i\alpha(\theta_0^p-\theta_p)}\Psi_p))|^2\,dx  \\
&\leq 2\int_{D_R} |\nabla\eta_R|^2|\psi_k-e^{i\alpha(\theta_0^p-\theta_p)}\Psi_p|^2\,dx
+2\int_{\R^2\setminus D_{R/2}} |(i\nabla+A_0)(\psi_k-e^{i\alpha(\theta_0^p-\theta_p)}\Psi_p)|^2\,dx\\
&\leq \frac{32}{R^2} \int_{D_R\setminus D_{R/2}} |\Psi_p-e^{i\alpha(\theta_p-\theta_0^p)}\psi_k|^2\,dx +2\int_{\R^2\setminus D_{R/2}} |(i\nabla+A_p)(\Psi_p-e^{i\alpha(\theta_p-\theta_0^p)}\psi_k)|^2\,dx\\
&\leq \frac{32(R+1)^2}{R^2} \int_{D_{R+1}(p)\setminus D_{\frac{R-2}2}(p)}
\frac{|\Psi_p-e^{i\alpha(\theta_p-\theta_0^p)}\psi_k|^2}
{|x-p|^2}\,dx \\
&\quad\qquad+2\int_{\R^2\setminus D_{R/2}}
  |(i\nabla+A_p)(\Psi_p-e^{i\alpha(\theta_p-\theta_0^p)}\psi_k)|^2\,dx\\
&\leq  \frac{32(R+1)^2}{R^2\mu_1} \int_{D_{R+1}(p)\setminus D_{\frac{R-2}2}(p)}|(i\nabla+A_p)(\Psi_p-e^{i\alpha(\theta_p-\theta_0^p)}\psi_k)|^2\,dx \\
&\quad\qquad+2\int_{\R^2\setminus D_{R/2}}
  |(i\nabla+A_p)(\Psi_p-e^{i\alpha(\theta_p-\theta_0^p)}\psi_k)|^2\,dx
\end{align*}
which goes to zero again thanks to \eqref{eq:Psip2}. Therefore we have obtained \eqref{eq:32} and the proof is complete.
\end{proof}

\begin{proof}[Proof of Theorem \ref{thm:eigenfunctions}]
Let $p\in {\mathbb S}^1$ and $\eps>0$. From Lemma
\ref{l:energy_inside} and \eqref{eq:33} there exists
some 
$R_0>\max\{2,K_\delta\}$ sufficiently large such that 
\[
|{\mathcal F}_p(R_0)-{\mathfrak L}_p|<\eps\quad\text{and}\quad 
|G(p,R_0)|<\eps.
\]
Moreover, again from Lemmas \ref{l:energy_outside} and \ref{l:energy_inside} there exists
$\rho>0$ (depending on $p$, $\eps$, and $R_0$) such that, if
$a=|a|p$ and $|a|<\rho$, then 
\[
|G(p,a,R_0)-G(p,R_0)|<\eps
\]
and 
\[
\left|\frac1{|a|^{2|\alpha-k|}}\int_{D_{R_0|a|}}\left|
(i\nabla+A_a)\varphi_a(x)-e^{i\alpha(\theta_a-\theta_0^a)(x)}(i\nabla+A_0)\varphi_0(x)
\right|^2\,dx-|\beta|^2{\mathcal F}_p(R_0)\right|<\eps.
\]
Therefore, taking into account Lemma \ref{l:energy_outside},
 we have that, for all $a=|a|p$ with $|a|<\rho$, 
\begin{align*}
&  \bigg||a|^{-2|\alpha-k|}\int_\Omega \left|
 (i\nabla+A_a)\varphi_a-e^{i\alpha(\theta_a-\theta_0^a)}(i\nabla+A_0)\varphi_0\right|^2\,dx-|\beta|^2{\mathfrak
  L}_p\bigg|\\
&\leq\bigg||a|^{-2|\alpha-k|}\int_{D_{R_0|a|}} \left|
(i\nabla+A_a)\varphi_a-e^{i\alpha(\theta_a-\theta_0^a)}(i\nabla+A_0)\varphi_0\right|^2\,dx-|\beta|^2 {\mathcal F}_p(R_0)\bigg|\\
&\quad +|a|^{-2|\alpha-k|}\int_{\Omega\setminus D_{R_0|a|}} \left|
 (i\nabla+A_a)\varphi_a-e^{i\alpha(\theta_a-\theta_0^a)}(i\nabla+A_0)\varphi_0\right|^2\,dx+|\beta|^2|{\mathfrak
  L}_p-{\mathcal F}_p(R_0)|\\
&<\eps+G(p,a,R_0)+|\beta|^2\eps\leq
  \eps+|G(p,a,R_0)-G(p,R_0)|+|G(p,R_0)|+|\beta|^2\eps =(3+|\beta|^2)\eps,
\end{align*}
thus concluding the proof.
\end{proof}

\bigbreak

\bigskip\noindent {\bf Acknowledgments.}  The authors are
partially supported by the project ERC Advanced Grant
2013 n. 339958 : ``Complex Patterns for Strongly Interacting Dynamical
Systems -- COMPAT''.  They also acknowledge the support of the
projects MIS F.4508.14 (FNRS) \& ARC AUWB-2012-12/17-ULB1- IAPAS for a
research visit at Universit\'e Libre de Bruxelles, where part of this
work has been achieved.  L. Abatangelo and V. Felli are partially
supported by the PRIN2015 grant ``Variational methods, with
applications to problems in mathematical physics and geometry''
and by the 2017-GNAMPA project ``Stabilità e analisi
  spettrale per problemi alle derivate parziali''.  Finally, the
authors would like to thank Susanna Terracini for her encouragement
and for fruitful discussions.


\end{document}